\documentclass[11pt,reqno]{amsart}
\usepackage[a4paper]{geometry}
\usepackage[utf8]{inputenc} \usepackage[T1]{fontenc} \usepackage{lmodern}
\usepackage{amssymb} \usepackage[all,2cell]{xy} \UseAllTwocells
\usepackage{nicefrac,mathtools,enumitem} \usepackage{microtype}
\usepackage{amsxtra} \usepackage{MnSymbol}

\usepackage{tikz} \usetikzlibrary{matrix} \tikzset{cd/.style=matrix of math
  nodes,row sep=2em,column sep=2em, text height=1.5ex, text depth=0.5ex}
\tikzset{cdar/.style=->,auto}

\usepackage{xcolor}

\usepackage{marginnote}
\newcommand{\corr}[2]{{#2}}

\usepackage[pdftitle=The\ maximal\ quantum\ group-twisted\ tensor\ products\ of\
C*-algebras, pdfauthor={Sutanu Roy and Thomas Timmermann},
pdfsubject={Mathematics; MSC } ]{hyperref} \usepackage[lite]{amsrefs}
\newcommand*{\MRref}[2]{\ MR
  #1}

\renewcommand{\PrintDOI}[1]{\href{http://dx.doi.org/#1}{DOI #1}
  \IfEmptyBibField{volume}{, (to appear in print)}{}}

\numberwithin{equation}{section}

\theoremstyle{plain}
\newtheorem{theorem}{Theorem}[section]
\newtheorem{lemma}[theorem]{Lemma}
\newtheorem{proposition}[theorem]{Proposition}

\newtheorem{corollary}[theorem]{Corollary}

\theoremstyle{definition} \newtheorem{definition}[theorem]{Definition}

\theoremstyle{remark} \newtheorem{remark}[theorem]{Remark}

\newtheorem{example}[theorem]{Example}

\DeclareMathOperator{\Ad}{Ad}
\newcommand*{\C}{\mathbb{C }}
\newcommand*{\Z}{\mathbb{Z}}
\newcommand*{\T}{\mathbb{T}}

\newcommand*{\univ}{\mathrm{u}}
\newcommand*{\red}{\mathrm{r}}
\newcommand{\normal}{\mathrm{n}} 
\newcommand*{\transpose}{\mathsf T}
\newcommand*{\defeq}{\mathrel{\vcentcolon=}}
\newcommand*{\conj}[1]{\overline{#1}}
\newcommand*{\nb}{\nobreakdash} \newcommand*{\Star}{*-}

\newcommand*{\Hils}[1][H]{\mathcal{#1}}
\newcommand*{\Bound}{\mathbb{B}}
\newcommand*{\Comp}{\mathbb{K}}

\newcommand*{\Contvin}{\mathrm{C}_0}
\newcommand*{\Contb}{\mathrm{C}_\mathrm{b}}
\newcommand*{\Cst}{\mathrm{C}^*}
\newcommand*{\Cred}{\mathrm{C}^*_\mathrm{r}}
\newcommand*{\CLS}{\mathrm{CLS}}
\newcommand*{\Mult}{\mathcal M}
\newcommand*{\U}{\mathcal U}

\newcommand*{\Mor}{\mathrm{Mor}}
\newcommand*{\Id}{\mathrm{id}}
\newcommand*{\Cstcat}{\mathfrak{C^*\!alg}}

\newcommand*{\Flip}{\Sigma}
\newcommand*{\flip}{\sigma}

\newcommand*{\maxtensor}{\mathrm{max}} \newcommand*{\mintensor}{\mathrm{min}}
\newcommand*{\maxotimes}{\otimes_{\maxtensor}}

\newcommand*{\Bialg}[1]{(#1,\Comult[#1])}
\newcommand*{\DuBialg}[1]{(\hat {#1},\DuComult[#1])}

\newcommand*{\G}[1][G]{\mathbb{#1}}
\newcommand*{\DuG}[1][G]{\widehat{\mathbb{#1}}}

\newcommand*{\Qgrp}[2]{\mathbb{#1}=\Bialg{#2}}
\newcommand*{\DuQgrp}[2]{\widehat{\mathbb{#1}}=\DuBialg{#2}}

\newcommand*{\UDuQgrp}[2]{\widehat{\mathbb{#1}}^\univ=(\hat
  {#2}^\univ,\DuComult[#2]^\univ)}

\newcommand*{\UG}[1][G]{\mathbb{#1}^\univ}
\newcommand*{\UDG}[1][G]{\widehat{\mathbb{#1}}^\univ}

\newcommand*{\Comult}[1][]{\Delta_{#1}}
\newcommand*{\DuComult}[1][]{\hat {\Delta}_{#1}}

\newcommand*{\Coinv}{R}

\DeclareRobustCommand{\rchi}{{\mathpalette\irchi\relax}}
\newcommand*{\irchi}[2]{\raisebox{0pt}{$#1\chi$}} 

\newcommand*{\bichar}{\rchi}
\newcommand*{\Dubichar}{\hat {\bichar}}

\newcommand*{\Multunit}{\mathbb{W}}
\newcommand*{\multunit}[1][]{{W}^{#1}}
\newcommand*{\DuMultunit}{\widehat{\mathbb
    W}}
\newcommand*{\Dumultunit}[1][]{\widehat{
    W}{}^{#1}}

\newcommand*{\unibich}{\mathcal W}

\newcommand*{\GenDrinfdouble}[3]{\mathfrak{D}_{#3}(#1,#2)}
\newcommand*{\GenDrinfdoubleU}[3]{\mathfrak{D}^{\univ}_{#3}(#1,#2)}

\newcommand*{\Rmat}{\mathcal{R}}

\newcommand*{\YDcat}{\mathcal{YD}\text{-}\Cstcat}

\newcommand*{\corep}[1]{#1}

\newcommand*{\maxcorep}[1][]{\mathcal{V}^{#1}}
\newcommand*{\dumaxcorep}[1][]{\tilde{\mathcal{V}}^{#1}}
 
\newcommand*{\DrinAlg}{\mathcal{D}}
\newcommand*{\CodoubAlg}{\widehat{\mathcal{D}}}

\newcommand{\chiprod}{\square}

\begin{document}
\title[The maximal quantum group-twisted $\Cst$-tensor product]{The maximal
  quantum group--twisted tensor product of $\Cst$--algebras}

\author{Sutanu Roy} \thanks{First author was partially supported by DFG 1493,
  Fields-Ontario postdoctoral fellowship and INSPIRE faculty award given by
  D.S.T., Government of India} \email{sutanu@niser.ac.in}
\address{School of Mathematical Sciences\\
  National Institute of Science Education and Research,  Bhubaneswar, HBNI\\
  Jatni, 752050\\
  India}

\author{Thomas Timmermann} \thanks{Second author supported by SFB 878 of the
  DFG} \email{timmermt@uni-muenster.de} \address{Fachbereich Mathematik und
  Informatik\\ Westfälische Wilhelms-Universität Münster\\ Einsteinstraße 62 \\
  48149 Münster \\ Germany}

\begin{abstract}
  We construct a maximal counterpart to the minimal quantum group-twisted tensor
  product of $\Cst$-algebras studied by Meyer, Roy and
  Woronowicz~\cites{Meyer-Roy-Woronowicz:Twisted_tensor,
    Meyer-Roy-Woronowicz:Twisted_tensor_2}, which is universal with respect to
  representations satisfying certain braided commutation relations. Much like
  the minimal one, this product yields a monoidal structure on the coactions of
  a quasi-triangular $\Cst$-quantum group, the horizontal composition in a
  bicategory of Yetter-Drinfeld $\Cst$-algebras, and coincides with a Rieffel
  deformation of the non-twisted tensor product in the case of group coactions.
\end{abstract}


\maketitle


\section{Introduction}
\label{sec:introduction}
Let~$G$ be a locally compact group acting on~$\Cst$\nb-algebras~$C$ and~$D$.
Then the minimal and the maximal tensor products $C \otimes_{\mintensor} D$
\corr{and$C \otimes_{\maxtensor} D$}{and $C \otimes_{\maxtensor} D$} carry
canonical \emph{diagonal} actions of $G$. However, this is no longer true
when~$G$ is replaced by a quantum group~$\G$. This problem appears already on
the purely algebraic level, where it can be solved if the quantum group~$\G$ is
quasi-triangular and the multiplication of the tensor product is twisted
accordingly \cite{Majid:Quantum_grp}*{Corollary 9.2.14}. In the setting
of~$\Cst$\nb-algebras, building on the work of
Vaes~\cite{Vaes:Induction_Imprimitivity}*{Proposition 8.3}, such a twisted
tensor product was first constructed by Nest and Voigt in the case where $\G$ is
the quantum double or the Drinfeld double of some regular locally compact
quantum group~$\G[H]$ or, equivalently, when $C$ and $D$ are Yetter-Drinfeld
$\Cst$\nb-algebras over $\G[H]$ \cite{Nest-Voigt:Poincare}*{Proposition 3.2}. A
systematic study of quantum group\nb-twisted tensor products, in the general
framework of manageable multiplicative unitaries, was taken up by Meyer, Roy and
Woronowicz in \cite{Meyer-Roy-Woronowicz:Twisted_tensor}, and carried on in
\cite{Meyer-Roy-Woronowicz:Twisted_tensor_2}. These constructions are
\emph{reduced} or \emph{minimal} in the sense that they use canonical Hilbert
space representations and reduce to the minimal tensor product $C\otimes D$ if
$\G$ is trivial.\newline In this article, we introduce a \emph{universal} or
\emph{maximal} counterpart to the constructions in
\cites{Meyer-Roy-Woronowicz:Twisted_tensor, Nest-Voigt:Poincare}. As in
\cite{Meyer-Roy-Woronowicz:Twisted_tensor}, we start with two~$\Cst$\nb-quantum
groups $\G$ and $\G[H]$ (in the sense
of~\cite{Soltan-Woronowicz:Multiplicative_unitaries}), a bicharacter
$\chi\in\U(\hat {A}\otimes\hat {B})$ and two $\Cst$\nb-algebras~$C$ and~$D$
equipped with coactions of~$\G$ and~$\G[H]$, respectively. We then consider
representations of $C$ and $D$ on the same $\Cst$-algebra that commute in a
braided fashion with respect to $\chi$, and construct
a~$\Cst$\nb-algebra~$C\boxtimes^{\chi}_{\maxtensor} D$ with nondegenerate
\Star{}homomorphisms~$j_{C}^{\univ}\colon
C\to\Mult(C\boxtimes^{\chi}_{\maxtensor} D)$ and $j^{\univ}_D\colon
D\to\Mult(C\boxtimes^{\chi}_{\maxtensor} D)$ such that
$(j^{\univ}_{C},j^{\univ}_{D})$ is a universal pair of braided-commuting
representations.

For example, this construction subsumes the following special cases.
\begin{enumerate}
\item If $\chi$ is trivial, it reduces to the maximal tensor product
  $C\maxotimes D$.
\item If $\G=\Z/2\Z$, so that $C$ and $D$ are $\Z/2\Z$-graded
  $\Cst$\nb-algebras, we obtain the universal $\Cst$\nb-algebra generated by a
  copy of $C$ and a copy of $D$ satisfying $dc = (-1)^{|c||d|}cd$ for
  homogeneous elements $c \in C$ and $d\in D$ of degrees $|c|,|d|\in \{0,1\}$.
\item If $\G[H]=\DuG$ and $D$ is the universal dual $\Cst$-algebra of $\G$, we
  obtain the universal crossed product $C\rtimes_{\univ} \UG$.
\end{enumerate}

Our maximal twisted tensor product shares many of the properties of the minimal
twisted tensor product, which we denote by~ $\boxtimes^{\chi}_{\min}$, established
in \cite{Meyer-Roy-Woronowicz:Twisted_tensor} and
\cite{Meyer-Roy-Woronowicz:Twisted_tensor_2}. We show that it carries a canonical
coaction of the generalised Drinfeld double $\GenDrinfdouble{\G}{\G[H]}{\chi}$
(see \cite{Roy:Codoubles}), yields a monoidal structure on the category of
$\G$-$\Cst$-algebras in the case where $\G$ is quasi-triangular, and is
functorial with respect to $\G$ and $\G[H]$ in a natural sense.

We also show that the maximal and the minimal twisted and non-twisted tensor
products are related by a commutative diagram of the form
\begin{align} \label{eq:cd} \xymatrix@R=15pt{
  (C \boxtimes^{\chi}_{\max} D) \rtimes \GenDrinfdouble{\G}{\G[H]}{\chi}  \ar[r]^{\cong} \ar@{->>}[d] & (C \maxotimes D) \rtimes (\G \times \G[H]) \ar@{->>}[d] \\
  (C \boxtimes^{\chi}_{\min} D) \rtimes \GenDrinfdouble{\G}{\G[H]}{\chi}
  \ar[r]^{\cong} & (C \otimes D) \rtimes (\G \times \G[H]), }
\end{align}
where the lower isomorphism extends
\cite{Meyer-Roy-Woronowicz:Twisted_tensor}*{Theorem 6.5}. The upper isomorphism
yields a quick proof of the following result, which extends our list of
examples:
\begin{enumerate} \setcounter{enumi}{3}
\item If $\G$ and $\G[H]$ are duals of locally compact abelian groups, then
  $C\boxtimes^{\chi}_{\max} D$ is a Rieffel deformation of the maximal tensor
  product $C\maxotimes D$ in the sense of Kasprzak
  \cite{Kasprzak:Rieffel_deformation}.
\end{enumerate}
The corresponding assertion for the minimal twisted tensor product is contained
in \cite{Meyer-Roy-Woronowicz:Twisted_tensor}*{Theorem 6.2}. If $C$ or $D$ are
nuclear, then \eqref{eq:cd} implies that the images of $C\boxtimes^{\chi}_{\max}
D$ and $C\boxtimes^{\chi}_{\min} D$ in the respective crossed products on the left
hand side, but not necessarily the algebras themselves, are isomorphic. In
particular, the canonical map from the maximal to the minimal twisted tensor
product is an isomorphism if (i) $C$ or $D$ is nuclear and additionally (ii) the
coaction of $\GenDrinfdouble{\G}{\G[H]}{\chi}$ on $C\boxtimes^{\chi}_{\max} D$ is
injective.

Finally, we consider the case where $C$ and $D$ are generalized Yetter-Drinfeld
$\Cst$-algebras, show that the maximal twisted tensor product is a generalized
Yetter-Drinfeld $\Cst$-algebra again, and obtain a bicategory whose objects are
$\Cst$-quantum groups and 1-morphisms are generalized Yetter-Drinfeld
$\Cst$-algebras. Here, we need to work with coactions of universal
$\Cst$-quantum groups.

A recurring issue that arises here is to verify that certain pairs of
representations of the $\Cst$-algebras $C$ and $D$ or of the $\Cst$-algebras $A$
and $B$ underlying the $\Cst$-quantum groups $\G$ and $\G[H]$ satisfy braided
commutation relations, and to check how such relations transform if various
representations are put together. Instead of case-by-case calculations, we
present a categorical approach where braided-commuting representations are
interpreted as 2-morphisms in cubical tricategory, and the horizontal and
vertical compositions account for all constructions that we need to consider.

\medskip

This article is organized as follows. In Section \ref{sec:prelim}, we first
recall notation and preliminaries concerning $\Cst$-quantum groups and their
morphisms. In Sections \ref{sec:commutation-universal} and
\ref{sec:commutation-reduced}, we introduce the notion of braided commutation
relations for representations of $\Cst$-quantum groups, first on the universal
and then on the reduced level. After these preparations, we define the maximal
twisted tensor product in Section \ref{sec:max}, establish several of its
properties in Section \ref{sec:prop}, and construct the isomorphisms in the
fundamental diagram \eqref{eq:cd} in Section \ref{sec:iso}. In Section
\ref{sec:universal}, we pass to coactions of universal $\Cst$-quantum groups,
and in Section \ref{sec:yd}, we consider the maximal twisted tensor product of
generalized Yetter-Drinfeld $\Cst$-algebras. In the appendix, we summarize the
relation between coactions of $\Cst$-quantum groups and their universal
counterparts, and consider the push-forward of non-injective coactions along
morphisms of $\Cst$-quantum groups.

\section{Preliminaries}
\label{sec:prelim}
Throughout we use the symbol ``:='' to abbreviate the phrase ``defined by''.

All Hilbert spaces and $\Cst$\nb-algebras are assumed to be separable.
     
For two norm\nb-closed subsets~$X$ and~$Y$ of a~$\Cst$\nb-algebra, let
\[
  X\cdot Y\defeq\{xy : x\in X, y\in Y\}^{\mathrm{CLS}},
\]
where CLS stands for the~\emph{closed linear span}.
   
For a~$\Cst$\nb-algebra~$A$, let~$\Mult(A)$ be its multiplier algebra and
$\U(A)$ be the group of unitary multipliers of~$A$. The unit of~$\Mult(A)$ is
denoted by~$1_{A}$. Next we recall some standard facts about multipliers and
morphisms of~$\Cst$\nb-algebras
from~\cite{Masuda-Nakagami-Woronowicz:C_star_alg_qgrp}*{Appendix A}. Let~$A$
and~$B$ be~$\Cst$\nb-algebras. A \Star{}homomorphism~$\varphi\colon
A\to\Mult(B)$ is called \emph{nondegenerate} if~$\varphi(A)\cdot B=B$. Each
nondegenerate \Star{}homomorphism $\varphi\colon A\to\Mult(B)$ extends uniquely
to a unital \Star{}homomorphism $\widetilde{\varphi}$ from~$\Mult(A)$ to
$\Mult(B)$. Let $\Cstcat$ be the category of $\Cst$\nb-algebras with
nondegenerate \Star{}homomorphisms $A\to\Mult(B)$ as morphisms $A\to B$; let
Mor(A,B) denote this set of morphisms. We use the same symbol for an element
of~$\Mor(A,B)$ and its unique extenstion from~$\Mult(A)$ to~$\Mult(B)$.
  
Let~$\conj{\Hils}$ be the conjugate Hilbert space to the Hilbert space~$\Hils$.
The \emph{transpose} of an operator $x\in\Bound(\Hils)$ is the operator
$x^\transpose\in\Bound(\conj{\Hils})$ defined by $x^\transpose(\conj{\xi}) \defeq
\conj{x^*\xi}$ for all $\xi\in\Hils$. The transposition is a linear, involutive
anti- isomorphism $\Bound(\Hils)\to\Bound(\conj{\Hils})$.

A \emph{representation} of a $\Cst$\nb-algebra~$A$ on a Hilbert space~$\Hils$ is
a nondegenerate \Star{}homomorphism $\pi\colon A\to\Bound(\Hils)$. Since
$\Bound(\Hils)=\Mult(\Comp(\Hils))$, the nondegeneracy conditions
$\pi(A)\cdot\Comp(\Hils)=\Comp(\Hils)$ is equivalent to $\pi(A)(\Hils)$ being
norm dense in~$\Hils$, and hence this \corr{is same}{is the same} as having a
morphism from~$A$ to~$\Comp(\Hils)$. The identity representation
of~$\Comp(\Hils)$ on $\Hils$ is denoted by~$\Id_{\Hils}$. The group of unitary
operators on a Hilbert space~$\Hils$ is denoted by $\U(\Hils)$. The identity
element in $\U(\Hils)$ is denoted by~$1_{\Hils}$.

We use~$\otimes$ both for the tensor product of Hilbert spaces and minimal
tensor product of $\Cst$\nb-algebras, which is well understood from the context.
We write~$\Flip$ for the tensor flip $\Hils\otimes\Hils[K]\to
\Hils[K]\otimes\Hils$, $x\otimes y\mapsto y\otimes x$, for two Hilbert spaces
$\Hils$ and~$\Hils[K]$. We write~$\flip$ for the tensor flip isomorphism
$A\otimes B\to B\otimes A$ for two $\Cst$\nb-algebras $A$ and~$B$.
  
We use the leg numbering on the level of $\Cst$-algebras as follows. Let~$A_{1}$,
$A_{2}$, $A_{3}$ be $\Cst$\nb-algebras. For~$t\in\Mult(A_{1}\otimes A_{2})$, we
write
\begin{gather*}
  t_{12}\defeq t\otimes 1_{A_{3}} \in\Mult(A_{1}\otimes A_{2}\otimes A_{3}), \quad
  t_{23}\defeq\corr{$1_{A_{3}}\otimes t_{12}$}{1_{A_{3}}\otimes t}\in\Mult(A_{3} \otimes A_{1}\otimes A_{2}), \\
  \text{and} \quad t_{13}\defeq\flip_{12}(t_{23})=\flip_{23}(t_{12})\in\Mult(A_{1}
  \otimes A_{3}\otimes A_{2}).
\end{gather*}
In particular, we apply this notation in case $A_{i}=\Bound(\Hils_{i})$ for some
Hilbert spaces~$\Hils_{i}$, where $i=1,2,3$, and then $\flip$ amounts to
conjugation by ~$\Flip$.

\subsection{$\Cst$-bialgebras}
\label{sec:bialgebras}

A \emph{$\Cst$\nb-bialgebra} is a $\Cst$\nb-algebra $A$ with a comultiplication
$\Comult[A] \in\Mor(A,A\otimes A)$ that is coassociative in the sense that
\begin{align*}
  (\Comult[A] \otimes \Id_{A})\circ \Comult[A] = (\Id_{A} \otimes \Comult[A]) \circ \Comult[A].
\end{align*}
It satisfies the \emph{cancellation conditions} if
\begin{equation}
  \label{eq:Cancellations}
  \Comult[A](A)\cdot(1_A\otimes A)       = A\otimes A       = (A\otimes 1_A)\cdot\Comult[A](A).
\end{equation}

A \emph{morphism} of $\Cst$\nb-bialgebras $(A,\Comult[A])$ and $(B,\Comult[B])$
is a morphism $f\in \Mor(A,B)$ satisfying
\begin{align*}
  \Comult[B] \circ f &= (f\otimes f)\circ \Comult[A].
\end{align*}

A \emph{left corepresentation} of a $\Cst$-bialgebra $(A,\Comult[A])$ on a
$\Cst$-algebra $C$ is a unitary $U \in \U(A \otimes C)$ satisfying
\begin{align}
  \label{eq:corep-left}
  (\Delta_{A} \otimes \Id_{C})(U) &=U_{23}U_{13}.
\end{align}
A \emph{right corepresentation} of $(A,\Comult[A])$ on $C$ is a unitary $U \in
\U(C \otimes A)$ satisfying
\begin{align} \label{eq:corep-right} (\Id_{C} \otimes \Delta_{A})(U) &=
  U_{12}U_{13}.
\end{align}
If $U$ is a left or right corepresentation, then $\hat U:=\sigma(U)^{*}$ is a
right or left corepresentation, called the \emph{dual} of $U$. A
corepresentation on a Hilbert space $\mathcal{H}$ is just a corepresentation on
the $\Cst$-algebra $C=\Comp(\mathcal{H})$.

A \emph{bicharacter} between $\Cst$-bialgebras $(A,\Comult[A])$ and
$(B,\Comult[B])$ is a unitary $\chi \in \U(A\otimes B)$ that is a left
corepresentation of $(A,\Delta_{A})$ and a right corepresentation of
$(B,\Delta_{B})$. Every such bicharacter has a \emph{dual bicharacter}
\begin{align}
  \label{eq:dual_bicharacter}
  \hat  \chi = \sigma(\chi^{*}) \in \U(B\otimes A).
\end{align}

\subsection{$\Cst$-quantum groups
  \cites{Baaj-Skandalis:Unitaires,Soltan-Woronowicz:Remark_manageable,
    Soltan-Woronowicz:Multiplicative_unitaries, Woronowicz:Mult_unit_to_Qgrp}}
\label{sec:multunit_quantum_groups}

We follow the approach of Woronowicz, where a $\Cst$\nb-bialgebra is regarded as
a $\Cst$-quantum group if it arises from a well-behaved multiplicative unitary,
and which includes the locally compact quantum groups or, more precisely, the
reduced $\Cst$-algebraic quantum groups of Kustermans and Vaes
\cite{Kustermans-Vaes:LCQG}.
\begin{definition}[\cite{Baaj-Skandalis:Unitaires}*{Definition 1.1}]
  \label{def:multunit}
  Let~$\Hils$ be a Hilbert space. A unitary $\Multunit\in\U(\Hils\otimes\Hils)$
  is \emph{multiplicative} if it satisfies the \emph{pentagon equation}
  \begin{equation}
    \label{eq:pentagon}
    \Multunit_{23}\Multunit_{12}     = \Multunit_{12}\Multunit_{13}\Multunit_{23}
    \qquad
    \text{in $\U(\Hils\otimes\Hils\otimes\Hils).$}
  \end{equation}
\end{definition}
Technical assumptions such as manageability (\cite{Woronowicz:Mult_unit_to_Qgrp})
or, more generally, modularity (\cite{Soltan-Woronowicz:Remark_manageable}) are
needed in order to construct $\Cst$\nb-bialgebras out of a multiplicative
unitary.

\begin{theorem}[\cites{Soltan-Woronowicz:Remark_manageable,
    Soltan-Woronowicz:Multiplicative_unitaries, Woronowicz:Mult_unit_to_Qgrp}]
  \label{the:Cst_quantum_grp_and_mult_unit}
  Let~$\Hils$ be a separable Hilbert space and
  $\Multunit\in\U(\Hils\otimes\Hils)$ a modular multiplicative unitary.
  \begin{enumerate}
  \item The spaces
    \begin{alignat}{2}
      \label{eq:first_leg_slice}
      A &\defeq \{(\omega\otimes\Id_{\Hils})\Multunit :
      \omega\in\Bound(\Hils)_*\}^\CLS,\\
      \label{eq:second_leg_slice}
      \hat {A} &\defeq \{(\Id_{\Hils}\otimes\omega)\Multunit :
      \omega\in\Bound(\Hils)_*\}^\CLS
    \end{alignat}
    are separable, nondegenerate $\Cst$\nb-subalgebras of~$\Bound(\Hils)$.
  \item We have $\Multunit\in\U(\hat {A}\otimes
    A)\subseteq\U(\Hils\otimes\Hils)$. We write~$\multunit[A]$ for~$\Multunit$
    viewed as a unitary multiplier of $\hat {A}\otimes A$.
  \item There exist unique comultiplications
    \begin{align*}
      \Comult[A]\in\Mor(A,A\otimes A) \quad\text{and} \quad \DuComult[A] \in \Mor(\hat  A,\hat  A\otimes \hat  A)
    \end{align*}
    such that $\multunit[A]$ is a bicharacter for the $\Cst$\nb-bialgebras
    $(A,\Comult[A])$ and $(\hat A,\DuComult[A])$. Explicitly, for all $a\in A$
    and $\hat a\in \hat A$,
    \begin{align} \label{eq:comult} \Delta_{A}(a) &= \corr{$W_{12}(a\otimes
        1)W_{12}^{*}$}{\Multunit_{12}(a\otimes 1)\Multunit_{12}^{*}}, & \hat
      \Delta_{A}(\hat a) &= \corr{\\ $\sigma(W_{12}^{*}(1 \otimes \hat
        a)W_{12})$}{\sigma(\Multunit_{12}^{*}(1 \otimes \hat a)\Multunit_{12})}.
    \end{align}
    These two comultiplications satisfy the cancellation
    condition~\eqref{eq:Cancellations}.
  \item There exists a unique ultraweakly continuous, linear
    anti-automorphism~$\Coinv_A$ of~$A$ with
    \begin{align}
      \label{eq:opp_comult_via_antipode}
      \Comult[A]\circ \Coinv_A &=       \flip\circ(\Coinv_A\otimes \Coinv_A)\circ\Comult[A],
    \end{align}
    where $\flip(x\otimes y)=y\otimes x$. It satisfies $\Coinv_A^2=\Id_A$.
  \end{enumerate}
\end{theorem}

A \emph{$\Cst$\nb-quantum group} is a $\Cst$\nb-bialgebra $\Qgrp{G}{A}$
constructed from a modular multiplicative unitary as above.

If $(A,\Delta_{A})$ is a reduced $\Cst$-algebraic quantum group in the sense of
Kustermans and Vaes \cite{Kustermans-Vaes:LCQG}, that is, if it satisfies
certain density conditions and carries an analogue of a left and of a right Haar
measure, then one can associate to it a right regular representation
$\Multunit$, which is a modular multiplicative unitary, and identify
$(A,\Delta_{A})$ with the $\Cst$-bialgebra constructed from $\Multunit$ as
above. Thus, $(A,\Delta_{A})$ is a $\Cst$-quantum group.

The \emph{dual} multiplicative unitary is $\DuMultunit\defeq
\Flip\Multunit^*\Flip\in\U(\Hils\otimes\Hils)$, where $\Flip(x\otimes
y)=y\otimes x$. It is modular or manageable if~$\Multunit$ is. The
$\Cst$\nb-quantum group generated by~$\DuMultunit$ is the \emph{dual}
$\DuQgrp{G}{A}$ of $\G$.

Let $\G=(A,\Delta_{A})$ be a $\Cst$-quantum group constructed from modular
multiplicative unitary $\Multunit$ as above and let $C$ be a $\Cst$-algebra. By
\eqref{eq:comult}, a unitary $U \in \U(C\otimes A)$ is a right corepresentation
of $(A,\Delta_{A})$ and a unitary $V \in \U(\hat A \otimes C)$ is a left
corepresentation of $(\hat A,\hat \Delta_{A})$ if and only if
\begin{align} \label{eq:corep-explicit}
  \begin{aligned}
    U_{12}U_{13}\multunit[A]_{23}&= \multunit[A]_{23}U_{12}\quad \text{and} \quad
    \multunit[A]_{12}V_{13}V_{23} = V_{23}\multunit[A]_{12}
  \end{aligned}
\end{align}
in $\U(C \otimes \Comp(\Hils) \otimes A)$ or $\U(\hat A \otimes \Comp(\Hils)
\otimes C)$, respectively.

If $\Qgrp{G}{A}$ and $\Qgrp{H}{B}$ are $\Cst$\nb-quantum groups and $\chi\in
\U(A\otimes B)$ is a bicharacter, then by
\cite{Meyer-Roy-Woronowicz:Homomorphisms}*{Proposition 3.15},
\begin{align}
  \label{eq:antipode-bicharacter}
  (\Coinv_{A} \otimes\Coinv_{B})(\chi)=\chi.
\end{align}

\subsection{Universal quantum groups
  \cites{Soltan-Woronowicz:Multiplicative_unitaries,Meyer-Roy-Woronowicz:Homomorphisms}}
\label{sec:univ_qgr}

The \emph{universal dual quantum group} $\UDuQgrp{G}{A}$ associated to
$\DuQgrp{G}{A}$, introduced in~\cite{Kustermans:LCQG_universal} in the presence
of Haar weights and in~\cite{Soltan-Woronowicz:Multiplicative_unitaries} in the
general framework of modular multiplicative unitaries, is a $\Cst$-bialgebra
that satisfies the cancellation conditions and comes with a universal
bicharacter
\begin{align*}
  \dumaxcorep[A]\in\U(\hat  A^{\univ} \otimes A)
\end{align*}
such that
\begin{align*}
  \hat{A}^{\univ} =  \{ (\Id \otimes \omega)(\dumaxcorep[A]) : \omega \in A'\}^{\mathrm{CLS}}
\end{align*}
and the following universal property holds. For every right corepresentation
$\corep{U}$ of $(A,\Delta_{A})$ on a $\Cst$-algebra $C$, there exists a unique
morphism $\rho \in \Mor(\hat A^{\univ},C)$ such that
\begin{equation}
  \label{eq:univ_prop_dumaxcorep}
  (\rho\otimes\Id_A)\dumaxcorep[A]=\corep{U}  \qquad\text{in~$\U(C \otimes A)$.}
\end{equation}

Taking $\corep{U}=\multunit[A]$, we obtain a \emph{reducing map} $\hat
\Lambda_{A}\in \Mor(\hat A^\univ, \hat A)$ such that $(\hat
\Lambda_{A}\otimes\Id_{A})\dumaxcorep[A]=\multunit[A]$.

Taking $\corep{U}=1 \in \U(\C \otimes A)$, we obtain the \emph{counit} $\hat
\varepsilon^{\univ}_{A}\colon \hat A^{\univ} \to \C$. By
\cite{Soltan-Woronowicz:Multiplicative_unitaries}*{Proposition 31}, it satisfies
\begin{align}
  \label{eq:counit}
  (\hat  \varepsilon^{\univ}_{A} \otimes \Id_{\hat  A^{\univ}}) \circ \hat  \Delta^{\univ}_{A} = \Id_{\hat  A^{\univ}} = 
  ( \Id_{\hat  A^{\univ}} \otimes \hat  \varepsilon^{\univ}_{A}) \circ \hat  \Delta^{\univ}_{A}.
\end{align}

Taking $U=(j \otimes R_{A})(\dumaxcorep[A])$, where $j$ denotes the canonical
$*$-anti-isomorphism from $\hat A^{\univ}$ to the opposite $\Cst$-algebra, we
obtain the \emph{unitary antipode} $\hat \Coinv^{\univ}_{A}$, which we can regard
as a $*$-anti-isomorphism of $\hat A^{\univ}$. By
\cite{Soltan-Woronowicz:Multiplicative_unitaries}*{Proposition 42}, it satisfies
$(\hat\Coinv_{A}^{\univ})^{2} =\Id_{\hat A^{\univ}}$ and \begin{align}
                                                           \label{eq:antipode-universal}
                                                           \hat \Lambda_{A} \circ \hat \Coinv_{A}^{\univ} &= \hat\Coinv_{A}\circ \hat  \Lambda_{A}, &
                                                                                                                                                 \hat \Delta_{A}^{\univ} \circ \hat \Coinv_{A}^{\univ} &= \sigma \circ (\hat \Coinv_{A}^{\univ} \otimes \hat \Coinv_{A}^{\univ})\circ \hat \Delta_{A}^{\univ}.
                                                         \end{align} 

                                                         Similarly, there exist
                                                         unique bicharacters
                                                         \begin{align*}
                                                           \maxcorep[A] \in \U(\hat  A \otimes A^{\univ}) \quad\text{and}\quad \unibich^{A}\in \U(\hat {A}^\univ\otimes A^\univ)
                                                         \end{align*}
                                                         that lift
                                                         $\multunit[A]\in
                                                         \U(\hat {A}\otimes A)$.
                                                         The latter is
                                                         constructed
                                                         in~\cite{Kustermans:LCQG_universal}
                                                         in presence of Haar
                                                         weights and
                                                         in~\cite{Meyer-Roy-Woronowicz:Homomorphisms}
                                                         in the general
                                                         framework of modular
                                                         multiplicative
                                                         unitaries.

                                                         \subsection{Morphisms
                                                           of quantum groups
                                                           \cite{Meyer-Roy-Woronowicz:Homomorphisms}}
                                                         \label{sec:bicharacters_morphisms}
                                                         Let $\Qgrp{G}{A}$ and
                                                         $\Qgrp{H}{B}$ be
                                                         $\Cst$\nb-quantum
                                                         groups with duals
                                                         $\DuQgrp{G}{A}$ and
                                                         $\DuQgrp{H}{B}$,
                                                         respectively. According
                                                         to
                                                         \cite{Meyer-Roy-Woronowicz:Homomorphisms},
                                                         morphisms from $\G$ to
                                                         $\DuG[H]$ can be
                                                         described in terms of
                                                         bicharacters $\chi \in
                                                         \U(\hat A \otimes \hat
                                                         B)$, of morphisms from
                                                         $\UG$ to $\UDG[H]$, and
                                                         in terms of right or
                                                         left quantum group
                                                         homomorphisms.
                                                         \begin{definition}
                                                           \label{def:right_quantum_morphism}
                                                           An
                                                           element~$\Delta_{R}\in\Mor(A,A\otimes\hat
                                                           {B})$ is a
                                                           \emph{right quantum
                                                             group homomorphism}
                                                           from~$\G$
                                                           to~$\DuG[H]$ if it
                                                           satifies
                                                           \begin{align}
                                                             \label{eq:right_homomorphism}
                                                             \begin{aligned}
                                                               (\Comult[A]\otimes\Id_{\hat
                                                                 {B}})\circ\Delta_{R}
                                                               =
                                                               (\Id_{A}\otimes\Delta_{R})\circ\Comult[A], \\
                                                               (\Id_{A}\otimes\DuComult[B])\circ\Delta_{R}
                                                               =
                                                               (\Delta_{R}\otimes\Id_{\hat
                                                                 {B}})\circ\Delta_{R}.
                                                             \end{aligned}
                                                           \end{align}           
                                                           A \emph{left quantum
                                                             group homomorphism}
                                                           from~$\G$
                                                           to~$\DuG[H]$ is an
                                                           element
                                                           $\Delta_L\in\Mor(A,\hat
                                                           {B}\otimes A)$ that
                                                           satisfies
                                                           \begin{align}
                                                             \label{eq:left_homomorphism}
                                                             \begin{aligned}
                                                               (\Id_{\hat
                                                                 {B}}\otimes\Comult[A])\circ\Delta_{L}
                                                               =
                                                               (\Delta_{L}\otimes\Id_{A})\circ\Comult[A], \\
                                                               (\DuComult[B]\otimes\Id_{A})\circ\Delta_{L}
                                                               = (\Id_{\hat
                                                                 {B}}\otimes\Delta_{L})\circ\Delta_{L}.
                                                             \end{aligned}
                                                           \end{align}
                                                           A morphism $f\in
                                                           \Mor(A^{\univ},\hat{B}^{\univ})$
                                                           of $\Cst$-bialgebras
                                                           is a \emph{morphism}
                                                           from $\G^{\univ}$ to
                                                           $\DuG[H]^{\univ}$.
                                                         \end{definition}
                                                         The following theorem
                                                         summarises some of the
                                                         main results
                                                         of~\cite{Meyer-Roy-Woronowicz:Homomorphisms}.
                                                         \begin{theorem}
                                                           \label{the:equivalent_notion_of_homomorphisms}
                                                           There are natural
                                                           bijections between
                                                           the following sets:
                                                           \begin{enumerate}
                                                           \item bicharacters
                                                             $\bichar\in\U(\hat
                                                             {A}\otimes\hat
                                                             {B})$ from~$\G$
                                                             to~$\DuG[H]$;
                                                           \item bicharacters
                                                             $\Dubichar\in\U(\hat
                                                             {B}\otimes\hat
                                                             {A})$ from~$\G[H]$
                                                             to~$\DuG$;
                                                           \item right quantum
                                                             group homomorphisms
                                                             $\Delta_R\in
                                                             \Mor(A, A\otimes
                                                             \hat {B})$;
                                                           \item left quantum
                                                             group homomorphisms
                                                             $\Delta_L \in
                                                             \Mor(A, \hat
                                                             {B}\otimes A)$;
                                                           \item morphisms $f
                                                             \in \Mor(A^{\univ},
                                                             \hat B^{\univ})$
                                                             from $\G^{\univ}$
                                                             to
                                                             $\DuG[H]^{\univ}$;
                                                           \item morphisms $\hat
                                                             f\in
                                                             \Mor(B^{\univ},
                                                             \hat A^{\univ})$
                                                             from
                                                             $\G[H]^{\univ}$ to
                                                             $\DuG^{\univ}$;
                                                           \item bicharacters
                                                             $\chi^{\univ} \in
                                                             \U(\hat A^{\univ}
                                                             \otimes \hat
                                                             B^{\univ})$.
                                                           \end{enumerate}
                                                           The mutually
                                                           corresponding objects
                                                           are related by the
                                                           following equations:
                                                           \begin{gather} \label{eq:bicharacters}
                                                             \begin{aligned}
                                                               \hat \chi &=
                                                               \sigma(\chi)^{*},
                                                               & \chi &= (\hat
                                                               \Lambda_{A}
                                                               \otimes \hat
                                                               \Lambda_{B})\chi^{\univ},
                                                             \end{aligned} \\
                                                             \label{eq:morphism-bicharacter}
                                                             \chi^{\univ} =  (\Id_{\hat  A^{\univ}} \otimes f)(\unibich^{A}) =  (\hat  f \otimes \Id_{\hat  B^{\univ}})(\hat  \unibich^{B}), \\
                                                             \label{eq:left-right-homomorphisms}
                                                             \begin{aligned}
                                                               (\Id_{\hat {A}}
                                                               \otimes
                                                               \Delta_R)\multunit[A]
                                                               &=
                                                               \multunit[A]_{12}\bichar_{13},
                                                               & (\Id_{\hat {A}}
                                                               \otimes
                                                               \Delta_L)\multunit[A]&=
                                                               \bichar_{12}\multunit[A]_{13}.
                                                             \end{aligned}
                                                           \end{gather}
 \end{theorem} 
We denote the bicharactes $\chi$ and $\chi^{\univ}$ associated to a morphism $f \colon \UG \to \UDG[H]$ by $W^{f}$ and $\unibich^{f}$, respectively, so that $W^{A} = W^{\Id_{A}}$ and $\unibich^{A} = \unibich^{\Id_{A}}$.

 By \cite{Meyer-Roy-Woronowicz:Homomorphisms}*{Proposition 3.15}, every bicharacter $\chi \in \U(\hat A \otimes \hat B)$ satisfies the relation $ (\hat R_{A} \otimes \hat R_{B})(\chi) = \chi$.  Using \eqref{eq:antipode-universal} and uniqueness of the lift of a bicharacter, we get:
\begin{align}
  \label{eq:bicharacter-antipode-universal}
   (\hat  R^{\univ}_{A} \otimes \hat  R^{\univ}_{B})(\chi^{\univ}) &= \chi^{\univ}.
\end{align}
The bicharacter relations \eqref{eq:corep-left} and \eqref{eq:corep-right} together with \eqref{eq:counit} imply
\begin{align} \label{eq:bicharacter-counit}
   (\hat \varepsilon^{\univ}_{A} \otimes \Id_{\hat 
    B^{\univ}})(\chi^{\univ}) &= 1_{\hat  B^{\univ}}, & (\Id_{\hat 
    A^{\univ}} \otimes \hat  \varepsilon^{\univ}_{B})(\chi^{\univ}) &=
  1_{\hat  A^{\univ}}.  
\end{align}
Combining these relations with \eqref{eq:morphism-bicharacter}, we conclude that the morphism $f$ corresponding to $\chi$ and $\chi^{\univ}$ intertwines the unitary antipodes and counits,
\begin{align}
  \label{eq:morphism-antipode-counit}
  f \circ \hat \Coinv^{\univ}_{A}  &=  \hat  \Coinv^{\univ}_{B} \circ f, &
\hat \varepsilon^{\univ}_{B} \circ f &= \hat  \varepsilon^{\univ}_{A}.
\end{align}
The corresponding left and the right quantum group homomorphisms make the following diagram commute,
\begin{align} \label{eq:morphism-deltal-deltar-universal}
  \xymatrix@C=65pt@R=20pt{
    \hat  B^{\univ}\otimes A^{\univ} \ar[d]_{\hat \Lambda_{B} \otimes \Lambda_{A}}  &
\ar[l]_(0.45){(f\otimes \Id_{A^{\univ}})\Delta_{A}^{\univ}}    A^{\univ}\ar[r]^(0.45){(\Id_{A^{\univ}}\otimes f)\Delta_{A}^{\univ}} \ar[d]_{\Lambda_{A}} &
 A^{\univ} \otimes     \hat  B^{\univ} \ar[d]^{\Lambda_{A} \otimes \hat \Lambda_{B}}\\
 \hat  B\otimes A & \ar[l]_(0.45){\Delta_{L}}
    A\ar[r]^(0.45){\Delta_{R}} &  A \otimes A.
  }
\end{align}
In particular, this diagram and the relations \eqref{eq:antipode-universal}  and \eqref{eq:morphism-antipode-counit}  imply
\begin{align}
  \label{eq:antipode-morphisms}
 \Delta_{R} \circ\Coinv_{A} = \sigma\circ (\Coinv_{\hat  B} \otimes\Coinv_{A})\circ\Delta_{L}.
\end{align}

\subsection{Coactions  of $\mathrm{C}^{*}$-quantum groups}
\label{sec:action}

A \emph{(right) coaction} of a $\Cst$-bialgebra $(A,\Delta_{A})$ on a $\Cst$\nb-algebra~$C$ is a morphism $\gamma\in\Mor(C, C\otimes A)$ satisfying
  \begin{align}
    \label{eq:coaction}
    (\Id_C\otimes\Comult[A])\circ\gamma = (\gamma\otimes\Id_A)\circ\gamma.
  \end{align}
  Note that we do not assume injectivity of $\gamma$.  A morphism $\pi$ between $\Cst$-algebras $C$ and $D$ with coactions $\gamma$ and $\delta$ of $(A,\Delta_{A})$ is \emph{equivariant} if $\delta \circ \pi= (\pi\otimes \Id_{A}) \circ \gamma$.

Following \cite{Baaj-Skandalis-Vaes:Non-semi-regular}, we call  a coaction $(C,\gamma)$ \emph{(strongly) continuous} if it satisfies the  the \emph{Podle\'s condition}
    \begin{equation}
     \label{eq:Podles_cond}
     \gamma(C)\cdot(1_C\otimes A)=C\otimes A.
   \end{equation}
Note that every such coaction is \emph{weakly continuous} in the sense that
 \begin{align}
   \label{eq:action-slices}
    \{(\Id_{C} \otimes
        \omega)(\gamma(C)) : \omega \in A'\}^{\text{CLS}}= C.
 \end{align}
 The following straighforward result is well known:
 \begin{lemma}
   Let $\Qgrp{G}{A}$ be a $\Cst$-quantum group with universal $\Cst$-bialgebra
   $(A^{\univ},\Delta^{\univ}_{A})$.  Then for every coaction $(C,\gamma)$ of
   $(A^{\univ},\Delta^{\univ}_{A})$, the following conditions are equivalent:
   \begin{enumerate}
   \item  $\gamma$ is injective;
   \item  $\gamma$ is weakly continuous;
   \item  $(\Id_{C} \otimes \varepsilon^{\univ}_{A})\gamma = \corr{$\Id_{C})$}{\Id_C}$.
   \end{enumerate}
If $\gamma$ is continuous, then (1)--(3) hold.
\end{lemma}
\begin{proof}
  The equivalence of (1) and (3) is straightforward, and clearly, (3) implies (2). For the converse, observe that $(\Id_{C} \otimes \varepsilon_{A}^{\univ})\gamma(\Id_{C}\otimes \omega)\gamma =(\Id_{C}\otimes \omega)\gamma$ if $\omega\in (A^{\univ})'$.
\end{proof}

 Suppose now that $\G=(A,\Delta_{A})$ is a $\Cst$-quantum group.
 \begin{definition}
   \label{def:cont_coaction}
   We call a $\Cst$-algebra with a  continuous coaction of $(A,\Delta_{A})$ or $(A^{\univ},\Delta^{\univ}_{A})$  a
   \emph{$\G$\nb-$\Cst$-algebra} or  \emph{$\G^{\univ}$\nb-$\Cst$-algebra}, respectively, and denote by $\Cstcat(\G)$ and $\Cstcat(\G^{\univ})$, respectively, the categories formed by all such coactions and equivariant morphisms. 
 \end{definition}
Note that in case of $(A,\Delta_{A})$, we do not assume injectivity here.

\section{Braided commutation relations}
\label{sec:commutation-universal}

Let $\Qgrp{G}{A}$ and $\Qgrp{H}{B}$ be $\Cst$\nb-quantum groups.  To define the twisted maximal tensor product of a $\G$-$\Cst$-algebra and an $\G[H]$-$\Cst$-algebra with respect to a morphism from $\G$ to $\DuG[H]$, we need to consider certain braided commutation relations for representations of $A$ and $B$ which generalize the Heisenberg and anti-Heisenberg commutation relations considered in \cite{Meyer-Roy-Woronowicz:Twisted_tensor}.  We begin with pairs of representations that lift to the universal $\Cst$-algebras $A^{\univ}$ and $B^{\univ}$ and interprete them as 2-cells in a tricategory, where the vertical and horizontal compositions account for various constructions that will come up later.

\subsection{Braided-commuting representations}
 Denote by  $\DuQgrp{G}{A}$ and $\DuQgrp{H}{B}$ the duals of $\G$ and~$\G[H]$, and
let $f,g$ be morphisms from $\UG$ to $\UDG[H]$.  Denote by
\begin{align*}
  \unibich^{A} &\in \U(\hat{A}^{\univ} \otimes A^{\univ}), & \unibich^{B} &\in \U(\hat{B}^{\univ} \otimes B^{\univ}), &\unibich^{f},\unibich^{g} \in \U(\hat{A}^{\univ} \otimes \hat{B} ^{\univ})
\end{align*}
the universal bicharacters associated to $A$, $B$, $f$ and $g$, respectively, see  Subsection \ref{sec:bicharacters_morphisms},  by $W^{A},W^{B},W^{f},W^{g}$ their reduced counterparts, and by $\sigma$ the flip on a minimal tensor product of $\Cst$-algebras.
\begin{lemma} \label{lemma:pairs-universal-equivalence} 
Let $\alpha$ and $\beta$ be representations of $A^{\univ}$ and $B^{\univ}$, respectively, on the same Hilbert space $\Hils$. Then the following relations are equivalent:
  \begin{align} \label{eq:pairs-universal}
     \unibich^{f}_{12}\unibich^{A}_{1\alpha}\unibich^{B}_{2\beta} &=
    \unibich^{B}_{2\beta}\unibich^{A}_{1\alpha}\unibich^{g}_{12}  
&&\text{in }\U(\hat  A^{\univ} \otimes \hat  B^{\univ} \otimes \Comp(\Hils)), \\ \label{eq:pairs-universal-weak}
  W^{f}_{12}\maxcorep[A]_{1\alpha}\maxcorep[B]_{2\beta} &=
    \maxcorep[B]_{2\beta}\maxcorep[A]_{1\alpha}W^{g}_{12}  
    &&\text{in } \U(\hat  A \otimes \hat  B \otimes \Comp(\Hils)), \\
 \label{eq:pairs-universal-deltal-deltar}
(f \otimes \alpha)  \Delta^{\univ}_{A}(a) &= \unibich^{B}_{1\beta}(g \otimes \alpha)\sigma\Delta^{\univ}_{A}(a)(\unibich^{B}_{1\beta})^{*} &&\text{for all } a\in A^{\univ},
\end{align}
where $\unibich^{A}_{1\alpha} = ((\Id_{\hat{A}^{\univ}}\otimes \alpha)\unibich^{A})_{13}$, $\unibich^{B}_{2\beta} = ((\Id_{\hat{B}^{\univ}}\otimes \beta)\unibich^{B})_{23}$ et cetera.
\end{lemma}
\begin{proof}
  If \eqref{eq:pairs-universal} holds, then an application of $\hat \Lambda_{A} \otimes \hat \Lambda_{B} \otimes \Id$ yields \eqref{eq:pairs-universal-weak}. 
Conversely,
 suppose \eqref{eq:pairs-universal-weak} holds. Let $\dumaxcorep[f] =(\Id_{\hat A^{\univ}} \otimes \hat \Lambda_{B})(\unibich^{f})$ and $\dumaxcorep[g]=(\Id_{\hat A^{\univ}} \otimes \hat \Lambda_{B})(\unibich^{g})$.  Then
\begin{align*}
  (\hat  \Lambda_{A} \otimes \Id_{\hat  B} \otimes \Id_{\Comp(\Hils)})(\dumaxcorep[f]_{12}\unibich^{A}_{1\alpha}) &= 
    W^{f}_{12}\maxcorep[A]_{1\alpha} =
    \maxcorep[B]_{2\beta}\maxcorep[A]_{1\alpha}W^{g}_{12}  (\maxcorep[B]_{2\beta})^{*} =
(\hat  \Lambda_{A} \otimes \Ad(\maxcorep[B]_{1\beta}))(\unibich^{A}_{1\alpha}\dumaxcorep[g]_{12}).
\end{align*}
Since $\dumaxcorep[f]_{12}\unibich^{A}_{1\alpha}$ and
$\unibich^{A}_{1\alpha}\dumaxcorep[g]_{12}$ are left corepresentations,  \cite{Meyer-Roy-Woronowicz:Homomorphisms}*{Lemma 4.13} implies
\begin{align*}
  \dumaxcorep[f]_{12}\unibich^{A}_{1\alpha}  =(\Id_{\hat  A^{\univ}} \otimes \Ad(\maxcorep[B]_{1\beta}))(\unibich^{A}_{1\alpha}\dumaxcorep[g]_{12}) = \maxcorep[B]_{2\beta}\unibich^{A}_{1\alpha}\dumaxcorep[g]_{12}(\maxcorep[B]_{2\beta})^{*}.
\end{align*}
Thus,  our initial relation lifts from $\hat{A} \otimes \hat{B} \otimes \Comp(\Hils)$ to $\hat{A}^{\univ} \otimes \hat{B} \otimes \Comp(\Hils)$. 
Repeating the argument similarly  as in \cite{Meyer-Roy-Woronowicz:Homomorphisms}*{Proposition 4.14}, we can conclude that the relation lifts to $\hat{A}^{\univ} \otimes \hat{B}^{\univ} \otimes \Comp(\Hils)$ as well so that \eqref{eq:pairs-universal} holds.

Finally, \eqref{eq:pairs-universal} is equivalent to \eqref{eq:pairs-universal-deltal-deltar}
because
\begin{align*}
(\Id_{\hat  A^{\univ}} \otimes (f \otimes \alpha)\Delta^{\univ}_{A})(\unibich^{A}) =
   (\Id_{\hat  A^{\univ}} \otimes f \otimes \alpha)(\unibich^{A}_{12}\unibich^{A}_{13}) = 
\unibich^{f}_{12}\unibich^{A}_{1\alpha}
\end{align*}
and similarly $(\Id_{\hat  A^{\univ}} \otimes (g \otimes \alpha)\sigma\Delta^{\univ}_{A})(\unibich^{A})=\unibich^{A}_{1\alpha}\unibich^{g}_{12}  $.
\end{proof}

\begin{definition} 
  \label{definition:pairs-universal}
  An \emph{$(f,g)$\nb-pair} consists of non-degenerate representations $\alpha$ of $A^{\univ}$ and $\beta$ of $B^{\univ}$ on the same Hilbert space $\Hils$ satisfying \eqref{eq:pairs-universal}--\eqref{eq:pairs-universal-deltal-deltar}.
\end{definition}
We are primarily interested in the four combinations that arise \corr{when we
  one}{when one} of the morphisms is the trivial morphism $\tau \in \Mor(A^{\univ},\hat{B}^{\univ})$, given by $\tau(a)b=\varepsilon^{\univ}_{A}(a)b$ for all $a\in A$, $b\in B$, or when  $\G[H]=\DuG$ and one of the morphisms is the identity on $A^{\univ}=\hat{B}^{\univ}$. Note that the associated bicharacters are just $W^{\tau}=1_{A} \otimes 1_{\hat{B}}$ and $W^{\Id}=W^{A}$, respectively.
 \begin{definition}[{\cite{Meyer-Roy-Woronowicz:Twisted_tensor}, \cite{Roy:Codoubles}}]
\label{definition:universal-heisenberg}  
A \emph{Heisenberg pair for $f$} is a $(\tau,f)$\nb-pair, an \emph{anti-Heisenberg pair for $f$} is an $(f,\tau)$\nb-pair, and a \emph{Drinfeld pair for $f$} is an $(f,f)$\nb-pair of representations.
  \end{definition}
\begin{example} \label{example:pairs-trivial}
   A $(\tau,\tau)$-pair of representations is a commuting pair of representations. 
\end{example}  
\begin{example} \label{example:pairs-unit}
 The counits $\varepsilon^{\univ}_{A}$ and ${\varepsilon}^{\univ}_{B}$  on $A^{\univ}$ and ${B}^{\univ}$, respectively, form an $(f,f)$-pair for every  $f$ because $(\Id_{\hat{A}^{\univ}} \otimes\varepsilon^{\univ}_{A})(\unibich^{A}) = 1 \in M({\hat{A}^{\univ}})$  and $(\Id_{\hat{B}^{\univ}} \otimes \varepsilon^{\univ}_{B})(\unibich^{B})=1\in M(\hat{B}^{\univ})$.
  \end{example}
\begin{example}
\label{example:anti-hb-discrete}
Let $\Gamma$ be a discrete group and  consider the $\Cst$\nb-bialgebras $A=\Contvin(\Gamma)$ and $\hat  A=\Cred(\Gamma)$ that arise from the multiplicative unitary $\Multunit\defeq\sum_{g} \rho_{g} \otimes \delta_{g}$ acting on $l^{2}(G)\otimes l^{2}(G)$, where $\delta_{g}$ and $\rho_{g}$ denote the canonical projection  and  right translation operators on $l^{2}(G)$.  Denote by $U_{g} \in \hat{A}^{\univ}=\Cst(\Gamma)$, where $g\in \Gamma$, the canonical generators, so that $\unibich^{A} = \sum_{g} U_{g} \otimes \delta_{g}$.
Then a pair of representations $(\alpha,\beta)$ of $A^{\univ}$ and $\hat  A^{\univ}$  is a Heisenberg pair, anti-Heisenberg pair or Drinfeld pair for   $f=\Id_{A^{\univ}}$  if and only  if for all $g,h\in G$, the product $\alpha(\delta_{h})\beta(U_{g})$ is equal to
\begin{align*}
  \beta(U_{g})\alpha(\delta_{hg}), \qquad \beta(U_{g})\alpha(\delta_{g^{-1}h}) \qquad\text{or}\qquad \beta(U_{g})\alpha(\delta_{g^{-1}hg}),
\end{align*}
respectively.
 \end{example}
Let us collect a few useful formulas for Heisenberg pairs and anti-Heisenberg pairs.
 \begin{remark}
Taking $f$ or $g$ equal to $\tau$  in  \eqref{eq:pairs-universal-deltal-deltar}, we find that a pair of representations
 $(\alpha,\beta)$ is a Heisenberg pair  for a morphism $g$ from $\UG$ to $\UDG[H]$ if and only if
 \begin{align}
   \label{eq:hb-universal}
    \alpha(a) \otimes 1_{\hat B^{\univ}} =     (\hat\unibich^{B}_{\beta 1})^{*}(\alpha \otimes g) \Delta^{\univ}_{A}(a) (\hat\unibich^{B}_{\beta 1}) \quad\text{for all } a\in A^{\univ},
 \end{align}
and an anti-Heisenberg pair for a morphism $f$ from $\UG$ to $\UDG[H]$ if and only if
\begin{align}
  \label{eq:anti-hb-universal}
  (\unibich^{B}_{1\beta})^{*}(f\otimes \alpha)\Delta^{\univ}_{A}(a) \unibich^{B}_{1\beta} &= 1_{\hat B^{\univ}} \otimes \alpha(a) \quad\text{for all } a\in A^{\univ}.
\end{align}
In particular, if $\DuG[H] = \G$ so that $ B^{\univ}=\hat{A}^{\univ}$, then $(\alpha,\beta)$ is a Heisenberg pair for the identity on $\G$ if and only if
\begin{align}
  \label{eq:id-hb-universal}
  \unibich^{A}_{\beta 2}(\alpha(a) \otimes 1)(\unibich^{A}_{\beta 2})^{*} = (\alpha \otimes \Id_{A^{\univ}})\Delta^{\univ}_{A}(a)  \quad\text{for all } a\in A^{\univ},
\end{align}
and an anti-Heisenberg pair for the identity on $\G$ if and only if
\begin{align}
  \label{eq:id-anti-hb-universal}
  \hat \unibich^{A}_{1\beta}(1 \otimes \alpha(a))(\corr{$\unibich^{A}_{1\beta}$}{\hat \unibich^{A}_{1\beta}})^{*} = (\Id_{A^{\univ}} \otimes \alpha) \Delta^{\univ}_{A}(a)  \quad\text{for all } a\in A^{\univ}.    
\end{align}
 \end{remark}

Intertwiners of $(f,g)$-pairs are defined in a natural way.
 \begin{definition}
An  \emph{intertwiner} from an $(f,g)$-pair $(\alpha,\beta)$ on some Hilbert space $\Hils$ to an $(f,g)$-pair $(\alpha',\beta')$ on some Hilbert space $\Hils'$ is an operator $T \in \Bound(\Hils,\Hils')$ satisfying
 $T\alpha(a)=\alpha'(a)T$ and $T\beta(b)=\beta'(b)T$ for all $a\in A^{\univ}$ and $b\in B^{\univ}$.
We call two such $(f,g)$-pairs \emph{isomorphic} and write $(\alpha,\beta) \cong (\alpha',\beta')$ if they admit a unitary intertwiner,
   \end{definition}
Evidently, all $(f,g)$-pairs with intertwiners form a category.  We denote it by $\mathfrak{R}(f,g)$.

We shall also need a weaker notion of equivalence.
\begin{lemma}
  Let  $(\alpha,\beta)$  be an $(f,g)$\nb-pair. Then $\alpha(A^{\univ}) \cdot \beta(B^{\univ})$ is a $\Cst$-algebra.
\end{lemma}
\begin{proof}
  Apply slice maps of the form $\omega \otimes \omega' \otimes \Id$ to \eqref{eq:pairs-universal-weak} to see that
  $\beta(B^{\univ}) \cdot \alpha(A^{\univ}) = \alpha(A^{\univ})\cdot \beta(B^{\univ})$.
\end{proof}
\begin{definition}
We call two $(f,g)$-pairs $(\alpha,\beta)$ and $(\alpha',\beta')$  \emph{equivalent} and write $(\alpha,\beta) \sim (\alpha',\beta')$ if there exists an isomorphism of $\Cst$\nb-algebras $\Phi$ from $\alpha(A^{\univ})  \beta(B^{\univ})$ to $\alpha'(A^{\univ}) \beta'(B^{\univ})$ such that $\Phi\circ \alpha =\alpha'$ and $\Phi \circ \beta = \beta'$.
  \end{definition}

The unitary antipode yields a bijective correspondence between $(f,g)$-pairs and $(g,f)$-pairs as follows.
Given representations $\alpha$ and $\beta$ of $A^{\univ}$ and $B^{\univ}$ on some Hilbert  space $\Hils$, we define representations $\bar \alpha$ and $\bar \beta$ of $A^{\univ}$ and $B^{\univ}$ on the conjugate Hilbert space $\conj{\Hils}$  as in \cite{Meyer-Roy-Woronowicz:Twisted_tensor}*{Section 3} by
\begin{align} 
 \label{eq:pairs-bar}
  \bar \alpha (a) &:=\alpha(\Coinv^{\univ}_{A}(a))^{\transpose}, &
  \bar \beta(b) &:=\beta(\Coinv_{B}^{\univ}(b))^{\transpose},
\end{align}
where $\Coinv_{A}^{\univ}$ and $\Coinv_{B}^{\univ}$ denote the unitary antipodes and  $T^{\transpose}$  the transpose of an operator  $T\in\Bound(\Hils)$. 

\begin{lemma} 
 \label{lemma:pairs-symmetry}
Let $(\alpha,\beta)$ be a pair of non-degenerate representations of $A^{\univ}$ and $B^{\univ}$ on the same Hilbert space. Then the following assertions are equivalent:\newline
\begin{tabular}[h]{llll}
  (1)  & $(\alpha,\beta)$ is an $(f,g)$\nb-pair, &
  (2)  & $(\bar \alpha,\bar \beta)$ is a $(g,f)$\nb-pair, \\
  (3) & $(\beta,\alpha)$ is an $(\hat  f,\hat  g)$\nb-pair, &
  (4) & $(\bar \beta,\bar \alpha)$ is a $(\hat  g,\hat  f)$\nb-pair.
\end{tabular}
\end{lemma}
\begin{proof}
Copy the proof of  \cite{Meyer-Roy-Woronowicz:Twisted_tensor}*{Lemma 3.6, 3.7}.
\end{proof}
Since
 the  assignment $T \mapsto  T^{\transpose}$ is anti-multiplicative, the assignments
 \begin{align*}
   (\alpha,\beta) \mapsto \overline{(\alpha,\beta)}:= (\bar\alpha,\bar\beta) \quad \text{and} \quad T \mapsto T^{\transpose}
 \end{align*}
form a contravariant functor $\mathfrak{R}(f,g) \to \mathfrak{R}(g,f)$.  Moreover, clearly
\begin{align} 
 \label{eq:pairs-bar-equivalence}
  (\alpha,\beta)\sim(\alpha',\beta') \quad \Leftrightarrow \quad \overline{(  \alpha,  \beta)} \sim \overline{( \alpha', \beta')}.
\end{align}

\subsection{The tensor product}
Next, we assemble the categories $\mathfrak{R}(f,g)$ associated to  morphisms $f,g$ from $\UG$ to $\UDG[H]$  into a bicategory.
\begin{lemma}
  Let $f,g,h\colon \UG \to \UDG[H]$ be morphisms of universal $\Cst$-quantum groups.
  \begin{enumerate}
  \item If $(\alpha,\beta)$ is an $(f,g)$-pair on some Hilbert space $\Hils$ and $(\alpha',\beta')$ is a $(g,h)$-pair on some Hilbert space $\Hils'$, then
    \begin{align}
      \label{eq:vertical-composition}
      (\alpha,\beta) \otimes (\alpha',\beta') &:= ((\alpha\otimes \alpha')\circ \Comult[A]^{\univ}, (\beta\otimes \beta') \circ \Comult[B]^{\univ})
    \end{align}
    is an $(f,h)$-pair on $\Hils \otimes \Hils'$. Moreover,  the flips  $\Hils \otimes \Hils' \rightleftarrows \Hils'\otimes \Hils$ are isomorphisms
    \begin{align} \label{eq:antipode-tensor}
      \overline{(\alpha,\beta)} \otimes \overline{(\alpha',\beta')} \cong \overline{(\alpha',\beta') \otimes(\alpha,\beta)}.
    \end{align}
    \item The assignments
    \begin{align} \label{eq:pairs-vertical-functor} 
((\alpha',\beta'),(\alpha,\beta)) \mapsto (\alpha,\beta) \otimes (\alpha',\beta') \quad \text{and} \quad (T,S) \mapsto S\otimes T
    \end{align}
    define a functor $\mathfrak{R}(g,h) \times \mathfrak{R}(f,g) \to \mathfrak{R}(f,h)$.
  \end{enumerate}
\end{lemma}
\begin{proof}
  (1) Denote the pair on the right hand side in \eqref{eq:vertical-composition} by $(\alpha'',\beta'')$. Then \eqref{eq:corep-right}, applied to $\unibich^{A}$ and $\unibich^{B}$, implies
\begin{align}
  \label{eq:pairs-vertical-legs}
  \unibich^{A}_{1\alpha''} &= \unibich^{A}_{1\alpha}\unibich^{A}_{1\alpha'} \quad\text{and} \quad \unibich^{B}_{2\beta''}=\unibich^{B}_{2\beta}\unibich^{B}_{2\beta'},
\end{align}  
where $\unibich^{A}_{1\alpha} = (\Id_{\hat{A}^{\univ}} \otimes \alpha)(\unibich^{A})_{12}$ and $\unibich^{A}_{1\alpha'} = (\Id_{\hat{A}^{\univ}} \otimes \alpha')(\unibich^{A})_{13}$ in $\Mult(\hat{A}^{\univ} \otimes \Comp(\Hils) \otimes \Comp(\Hils'))$ and $\unibich^{B}_{2\beta}$ and $\unibich^{B}_{2\beta'}$ are defined similarly. Now two applications of \eqref{eq:pairs-universal} show that
\begin{align*}
\unibich^{f}_{12}\unibich^{A}_{1\alpha}\unibich^{A}_{1\alpha'}\unibich^{B}_{2\beta}\unibich^{B}_{2\beta'} 
  &= \unibich^{f}_{12}\unibich^{A}_{1\alpha}\unibich^{B}_{2\beta}\unibich^{A}_{1\alpha'}\unibich^{B}_{2\beta'} \\
  &= \unibich^{B}_{2\beta} \unibich^{A}_{1\alpha}\unibich^{g}_{12}\unibich^{A}_{1\alpha'}\unibich^{B}_{2\beta'} \\
  &= \unibich^{B}_{2\beta}\unibich^{A}_{1\alpha}\unibich^{B}_{2\beta'}\unibich^{A}_{1\alpha'} \unibich^{h}_{12} 
    = \unibich^{B}_{2\beta}\unibich^{B}_{2\beta'}\unibich^{A}_{1\alpha}\unibich^{A}_{1\alpha'} \unibich^{h}_{12}
 \end{align*}
 and hence $\unibich^{f}_{12} \unibich^{A}_{1\alpha''}\unibich^{B}_{2\beta''} = \unibich^{B}_{2\beta''}\unibich^{A}_{1\alpha''}\unibich^{h}_{12}$. 

The flip provides  isomorphisms in \eqref{eq:antipode-tensor} because the unitary antipodes reverse the comultiplications.

(2) Straightforward.
\end{proof}
Recall that a {\em bicategory} $\mathfrak{B}$ consists of a class of \emph{objects} $\mathrm{ob}\, \mathfrak{B}$, a category $\mathfrak{B}(f,g)$ for each $f,g \in \mathrm{ob}\, \mathfrak{B}$ whose objects and morphisms are called {\em 1-cells} and {\em 2-cells}, respectively, a functor $c_{f,g,h} \colon \mathfrak{B}(g,h) \times \mathfrak{B}(f,g) \to \mathfrak{B}(f,h)$ (``composition'') for each $f,g,h \in \mathrm{ob}\, \mathfrak{B}$, an object $1_{f} \in \mathfrak{B}(f,f)$ (``identity'') for each $f \in \mathrm{ob}\, \mathfrak{B}$, an isomorphism $a_{f,g,h,j}(\alpha,\beta,\gamma)$ from $c_{f,g,j}(c_{g,h,j}(\gamma,\beta),\alpha)$ to $ c_{f,h,j}(\gamma,c_{f,g,h}(\beta,\alpha))$ in $\mathfrak{B}(f,j)$ (``associativity'') for each triple of 1-cells $f \xrightarrow{\alpha} g \xrightarrow{\beta} h \xrightarrow{\gamma} j$ in $\mathfrak{B}$, and isomorphisms $l_{f}(\alpha) \colon c_{f,f,g}(\alpha,1_{f}) \to \alpha$ and $r_{g}(\alpha) \colon c_{f,g,g}(1_{g},\alpha) \to \alpha$ in $\mathfrak{B}(f,g)$ for each 1-cell $f \xrightarrow{\alpha} g$ in $\mathfrak{B}$, subject to several axioms \cite{Leinster:Higher-Operads}.
\begin{proposition}
Let $\G$ and $\G[H]$ be $\Cst$-quantum groups. 
There exists  a bicategory $\mathfrak{B}$, where the objects are all morphisms $f\colon\UG \to\UDG[H]$, 
the category   $\mathfrak{B}(f,g)$  is the category of $(f,g)$-pairs with intertwiners as morphisms, 
 the composition functors $\mathfrak{B}(g,h) \times \mathfrak{B}(f,g) \to \mathfrak{B}(f,h)$ are given by \eqref{eq:pairs-vertical-functor},
 the unit object $1_{f} \in \mathfrak{B}(f,f)$  is $(\varepsilon^{\univ}_{A},{\varepsilon}^{\univ}_{B})$, 
and  the  isomorphisms
  $a_{f,g,h,j}((\alpha,\beta),(\alpha',\beta'),(\alpha'',\beta''))$,
  $l_{f}((\alpha,\beta))$ and $r_{f}((\alpha,\beta))$  associated to pairs of representations
  on Hilbert spaces   $\Hils,\Hils'$ and $\Hils''$   are  the canonical isomorphisms
  \begin{align} \label{eq:hils-monoidal}
    \Hils \otimes (\Hils' \otimes \Hils'') \to (\Hils \otimes \Hils') \otimes\Hils'', \quad
    \C \otimes \Hils \to \Hils \quad \text{and} \quad \Hils \otimes \C \to \Hils.
  \end{align}
  \end{proposition}
\begin{proof}
  The isomorphisms in \eqref{eq:hils-monoidal} intertwine the representations of $A^{\univ}$ involved because 
  $(\Id_{A^{\univ}} \otimes \Delta^{\univ}_{A})  \Delta^{\univ}_{A} = (\Delta^{\univ}_{A} \otimes \Id_{A^{\univ}}) \Delta^{\univ}_{A}$, $
(\varepsilon^{\univ}_{A} \otimes \Id_{A^{\univ}})  \Delta^{\univ}_{A} = \Id_{A^{\univ}}$ and $
( \Id_{A^{\univ}} \otimes \varepsilon^{\univ}_{A}) \Delta^{\univ}_{A} = \Id_{A^{\univ}}$,
and likewise they intertwine the representations of $B^{\univ}$ involved. The coherence conditions that  these isomorphisms have to satisfy in order to obtain a bicategory reduce to the corresponding coherence conditions for the monoidal category of Hilbert spaces.
\end{proof}
From now on, we  suppress the isomorphisms in \eqref{eq:hils-monoidal} and  pretend the monoidal category of Hilbert spaces to be strict. Then the bicategory constructed above becomes a strict 2-category. We denote this 2-category by $\mathfrak{C}(\UG,\UDG[H])$.

\subsection{The cubical tricategory}
 We  now vary $\G$ and $\G[H]$ and assemble the associated 2-categories  $\mathfrak{C}(\UG,\UDG[H])$ into a   tricategory that is rather strict, namely, cubical \cite{Gordon-Power-Street:Tricategory}, \cite{gurski:dissertation}, or equivalently, into a category enriched over  2-categories, where the latter are equipped with the monoidal structure due to Gray \cite{Gray:lns}.

Let  $\G[F]=(A,\Delta_{A})$, $\G=(B,\Delta_{B})$ and $\G[H]=(C,\Delta_{C})$ be $\Cst$-quantum groups.
 \begin{lemma} \label{lemma:pairs-push-pull}
 Let $\phi\colon \UG[F] \to \UG$ and $\psi\colon \UG \to \UDG[H]$ be morphisms. Then there exist 
strict 2-functors 
$\phi^{*} \colon \mathfrak{C}(\UG,\UDG[H]) \to \mathfrak{C}(\UG[F],\UDG[H])$ and
$\psi_{*}\colon \mathfrak{C}(\G[F],\G) \to \mathfrak{C}(\G[F],\DuG[H])$
  such that for each morphism $f$, each pair of representations $(\alpha,\beta)$ and each intertwiner $T$,
\begin{align*}
  \phi^{*}  f&=f\circ \phi, &  \phi^{*}(\alpha,\beta) &=(\alpha\circ  \phi,\beta), & \phi^{*}T &=  T,\\
  \psi_{*}f &= \psi\circ f, & \psi_{*}(\alpha,\beta) &= (\alpha,\beta\circ \hat{\psi}), & \psi_{*}T&=T.
\end{align*}
\end{lemma}
 \begin{proof}
The \corr{verifications}{verification} is straightforward. For example, if $f,g \colon \UG\to \UDG[H]$ are morphisms and $(\alpha,\beta)$ is an $(f,g)$-pair, then $(\alpha \circ \phi,\beta)$ is an $(f\circ \phi,g\circ \phi)$-pair because
\begin{align*}
  \unibich^{f\phi}_{12}\unibich^{C}_{1(\alpha \phi)}\unibich^{B}_{2\beta} &= 
   W^{f}_{\hat  \phi 2}\unibich^{A}_{\hat  \phi \alpha}\unibich^{B}_{2\beta} 
= 
   \unibich^{B}_{2\beta} \unibich^{A}_{\hat  \phi \alpha}    W^{g}_{\hat  \phi 2} =
   \unibich^{B}_{2\beta} \unibich^{C}_{1 (\alpha \phi)}    \unibich^{g\phi}_{12},
    \end{align*}
where we used the relation $(\Id_{\hat  C^{\univ}} \otimes \phi)(\unibich^{C}) = \unibich^{\phi} =(\hat  \phi \otimes \Id_{A^{\univ}})( \unibich^{A})$.  \end{proof}
 \begin{lemma} \label{lemma:push-pull-intertwine}
Given morphisms $f,g \colon \UG[F]\to \UG$ and $f',g'\colon \UG\to \UDG[H]$, and an $(f,g)$-pair $(\alpha,\beta)$ and an $(f',g')$-pair $(\alpha',\beta')$, 
there exists an isomorphism
\begin{align*}
 U^{(\alpha',\beta')}_{(\alpha,\beta)} :=\hat{\unibich}^{B}_{\alpha'\beta}\Sigma  \colon f'_{*}(\alpha,\beta) \otimes g^{*}(\alpha',\beta') \to f^{*}(\alpha',\beta') \otimes g'_{*}(\alpha,\beta).
\end{align*}
 \end{lemma}
    \begin{proof}
Denote the underlying Hilbert spaces of $(\alpha,\beta)$ and $(\alpha',\beta')$ by $\Hils$ and $\Hils'$, and let
 \begin{align*}
       (\gamma,\delta)&= f'_{*}(\alpha,\beta) \otimes g^{*}(\alpha',\beta') = ((\alpha \otimes \alpha' g) \Delta^{\univ}_{A}, (\beta \hat{f'} \otimes \beta')\Delta^{\univ}_{B}), \\
                                                                                                                                      (\gamma',\delta')&= f^{*}(\alpha',\beta') \otimes g'_{*}(\alpha,\beta) =((\alpha'f \otimes \alpha)\Delta^{\univ}_{A}, (\beta' \otimes \beta\hat{g'})\Delta^{\univ}_{B}).
     \end{align*}
Then
by   \eqref{eq:corep-left}, \eqref{eq:corep-right} and
   \eqref{eq:morphism-bicharacter},
     \begin{align*}
       \unibich^{A}_{1\gamma} &= \unibich^{A}_{1\alpha} \unibich^{g}_{1\alpha'} \in \U(\hat{A}^{\univ} \otimes \Comp(\Hils)\otimes \Comp(\Hils')), &
                                                                                                         \unibich^{A}_{1\gamma'} &= \unibich^{f}_{1\alpha'}\unibich^{A}_{1\alpha} \in \U(\hat{A}^{\univ} \otimes \Comp(\Hils')\otimes \Comp(\Hils)), \\
       \unibich^{C}_{1\delta} &= \unibich^{f'}_{1\beta} \unibich^{C}_{1\beta' } \in \U(\hat{C}^{\univ} \otimes \Comp(\Hils) \otimes\Comp(\Hils')), &
                                                                                                                                          \unibich^{C}_{1\delta'}&= \unibich^{C}_{1\beta'} \unibich^{g'}_{1\beta} \in \U(\hat{C}^{\univ}\otimes \Comp(\Hils') \otimes \Comp(\Hils)),
     \end{align*}
and
 $\hat{\unibich}^{B}_{\alpha'\beta}\Sigma$ intertwines $(\gamma,\delta)$ and $(\gamma',\delta')$ because 
by \eqref{eq:pairs-universal},
     \begin{align*}
       \hat{\unibich}^{B}_{\alpha'\beta}\Sigma_{23}\unibich^{A}_{1\gamma} &=   \hat{\unibich}^{B}_{\alpha'\beta}\Sigma_{23}\unibich^{A}_{1\alpha}\unibich^{g}_{1\alpha'} =
                                                             \unibich^{f}_{1\alpha'}\unibich^{A}_{1\alpha}  \hat{\unibich}^{B}_{\alpha'\beta} \Sigma_{23} =       \unibich^{A}_{1\gamma'} \hat{\unibich}^{B}_{\alpha'\beta}\Sigma_{23}, \\
\hat{\unibich}^{B}_{\alpha'\beta}       \Sigma_{23} \unibich^{C}_{1\delta} &=\Sigma_{23}(\unibich^{B}_{\beta\alpha'})^{*} \unibich^{f'}_{1\beta}\unibich^{C}_{1\beta'} =
\Sigma_{23}\unibich^{C}_{1\beta'}\unibich^{g'}_{1\beta } (\unibich^{B}_{\beta\alpha'})^{*}  = \unibich^{C}_{1\delta'} \hat{\unibich}^{B}_{\alpha'\beta}\Sigma_{23}.
\end{align*}

 Finally, if $T$  intertwines $(\alpha,\beta)$ and some pair $f\xrightarrow{(\alpha'',\beta'') } g$, and $S$  intertwines $(\alpha',\beta')$ and some pair $f' \xrightarrow{(\alpha''',\beta''')} g'$,  then clearly $\hat{\unibich}^{B}_{\alpha'\beta}\Sigma(T\otimes S) = (S\otimes T)\hat{\unibich}^{B}_{\alpha'''\beta''}\Sigma$.
 \end{proof}
We can now define a second composition of pairs of representations as one part of a cubical functor  \cite{Gordon-Power-Street:Tricategory}; see also \cite{gurski:dissertation}.
\begin{proposition} \label{proposition:pairs-horizontal} Let $\G[F],\G$ and $\G[H]$ be $\Cst$-quantum groups.  There exists a cubical functor
  $\mathfrak{C}(\UG,\UDG[H]) \times \mathfrak{C}(\UG[F],\UG) \to \mathfrak{C}(\UG[F],\UDG[H])$, given on pairs of objects by $(f',f) \mapsto f'\circ f$, on pairs of representations
$f'\xrightarrow{(\alpha',\beta')} g'$ and $ f\xrightarrow{(\alpha,\beta)} g$ by
\begin{align*}
  ((\alpha',\beta'),(\alpha,\beta)) \mapsto (\alpha',\beta') \circ (\alpha,\beta):=  f^{*}(\alpha',\beta') \otimes g'_{*}(\alpha,\beta),
\end{align*}
and on pairs of intertwiners by $(T',T) \mapsto T \otimes T'$. 
\end{proposition} 
\begin{proof}
We  show that the functors and the unitary intertwiners obtained in Lemma \ref{lemma:pairs-push-pull} and Lemma \ref{lemma:push-pull-intertwine} satisfy the conditions in 
 \cite{gurski:dissertation}*{Proposition 5.2.2}.

Suppose given morphisms, pairs of representations and intertwiners as follows:
 \begin{align*}
   \xymatrix@C=60pt{
   f \rtwocell^{(\alpha,\beta)}_{(\gamma,\delta)}{\ T} & g
} \text{ in } \mathfrak{C}(\UG[F],\UG) \quad \text{and} \quad
   \xymatrix@C=60pt{
   f' \rtwocell^{(\alpha',\beta')}_{(\gamma',\delta')}{\ T'} & \corr{$g$}{g'}}
 \text{ in } \mathfrak{C}(\UG,\UDG[H]).
 \end{align*}
 Then the compositions of the 2-cells in
\begin{align*}
  \xymatrix@C=80pt@R=40pt{
  f' \circ  f \ar@/^7pt/[r]^{f'_{*}(\alpha,\beta)} \ar@{}[r]|{\Downarrow T}
\ar@/_7pt/[r]_{f'_{*}(\gamma,\delta)}
 \ar@/^7pt/[d]^{f^{*}(\alpha',\beta')} \ar@{}[d]|{ \stackrel{\scriptstyle T'}{\Leftarrow}}
\ar@/_7pt/[d]_{f^{*}(\gamma',\delta')} & f'\circ g \ar[d]^{g^{*}(\alpha',\beta')} \ar@/^10pt/@{=>}[ld]^(0.4){U^{(\alpha',\beta')}_{(\gamma,\delta)}}  \ar@{}[rd]|{\displaystyle\text{and}}&
f'\circ f  \ar[r]^{f'_{*}(\alpha,\beta)} \ar[d]_{f^{*}(\alpha',\beta')}   & f'\circ g  \ar@/^7pt/[d]^{f^{*}(\gamma',\delta')} \ar@{}[d]|{ \stackrel{\scriptstyle T'}{\Rightarrow}}
\ar@/_7pt/[d]_{f^{*}(\alpha',\beta')} \ar@/_10pt/@{=>}[ld]_(0.6){U^{(\alpha',\beta')}_{(\alpha,\beta)}} 
 \\
g' \circ f  \ar[r]_{g'_{*}(\gamma,\delta)}& g'\circ g  &
                                    g' \circ f    \ar@/^7pt/[r]^{f'_{*}(\alpha,\beta)} \ar@{}[r]|{\Downarrow T}
\ar@/_7pt/[r]_{f'_{*}(\gamma,\delta)} & g'\circ g 
  } 
\end{align*}
coincide because they are just
\begin{align*}
  (T'\otimes 1) U^{(\alpha',\beta')}_{(\gamma,\delta)} (T\otimes 1) &= (T' \otimes 1)\hat{\unibich}^{B}_{\alpha'\delta} \Sigma (T\otimes 1) 
 = (T'\otimes T) \hat{\unibich}^{B}_{\alpha'\beta}\Sigma = (T'\otimes T)U^{(\alpha',\beta')}_{(\alpha,\beta)}.
\end{align*}

Next, suppose that we have morphisms and pairs of braided-commuting representations
   \begin{align} \label{eq:cubical}
     f \xrightarrow{(\alpha,\beta)} g \xrightarrow{(\gamma,\delta)} h \text{ in } \mathfrak{C}(\UG[F],\UG) \quad \text{and} \quad
     f' \xrightarrow{(\alpha',\beta')} g' \xrightarrow{(\gamma',\delta')} h' \text{ in } \mathfrak{C}(\UG,\UDG[H]),
   \end{align}
with underlying Hilbert spaces $\Hils,\mathcal{K},\Hils',\mathcal{K}'$, respectively.  Write $(\phi,\psi)=(\gamma,\delta) \otimes (\alpha,\beta)$ and $(\phi',\psi') =(\gamma',\delta') \otimes (\alpha',\beta')$, and consider the diagram
\begin{align} \label{eq:cubical-compositions}
  \xymatrix@C=70pt{
  f'\circ f \ar[r]^{f'_{*}(\alpha,\beta)} \ar[d]_{f^{*}(\alpha',\beta')} & f'\circ  g \ar[d]|{g^{*}(\alpha',\beta')} \ar[r]^{f'_{*}(\gamma,\delta)}
\ar@{=>}[ld]|{U^{(\alpha',\beta')}_{(\alpha,\beta)}} & f'\circ  h \ar[d]^{h^{*}(\alpha',\beta')}
\ar@{=>}[ld]|{U^{(\alpha',\beta')}_{(\gamma,\delta)}} \\
  g'\circ f \ar[r]|{g'_{*}(\alpha,\beta)} \ar[d]_{f^{*}(\gamma',\delta')}   & g'\circ g \ar[r]|{g'_{*}(\gamma,\delta)}  \ar[d]|{g^{*}(\gamma',\delta')}
\ar@{=>}[ld]|{U^{(\gamma',\delta')}_{(\alpha,\beta)}}
& g'\circ h \ar[d]^{h^{*}(\gamma',\delta')} \ar@{=>}[ld]|{U^{(\gamma',\delta')}_{(\gamma,\delta)}} \\
h'\circ f \ar[r]_{h'_{*}(\alpha,\beta)} & h'\circ g \ar[r]_{h'_{*}(\gamma,\delta)} & h' \circ h.
   }
\end{align}
Then by \eqref{eq:corep-left} and \eqref{eq:corep-right}, 
\begin{align*}
\left(U^{(\alpha',\beta')}_{(\alpha,\beta)}\right)_{23}\left(U^{(\gamma',\delta')}_{(\alpha,\beta)}\right)_{12} &=
(\hat{\unibich}^{B}_{\alpha'\beta}\Sigma_{\Hils,\Hils'})_{23}(\hat{\unibich}^{B}_{\gamma'\beta}\Sigma_{\Hils,\mathcal{K}'})_{12} =
(\hat{\unibich}^{B}_{\phi'\beta}) \Sigma_{\Hils,\mathcal{K}'\otimes\Hils'} = U^{(\phi',\psi')}_{(\alpha,\beta)}
\end{align*}
and
\begin{align*}
\left(U^{(\alpha',\beta')}_{(\gamma,\delta)}\right)_{12}\left(U^{(\alpha',\beta')}_{(\alpha,\beta)}\right)_{23} &=
  (\hat{\unibich}^{B}_{\alpha'\delta}\Sigma_{\mathcal{K},\Hils'})_{12} (\hat{\unibich}^{B}_{\alpha'\beta}\Sigma_{\Hils,\Hils'})_{23} =
  (\hat{\unibich}^{B}_{\alpha'\psi}) \Sigma_{\mathcal{K}\otimes \Hils, \Hils'} =U^{(\alpha',\beta')}_{(\phi,\psi)},
\end{align*}
where $\Sigma_{\Hils,\Hils'}$ denotes the flip $\Hils \otimes \Hils' \to \Hils' \otimes \Hils$ et cetera.
Thus, the three axioms in \cite{gurski:dissertation}*{Proposition 5.2.2} hold and the assertion follows.
 \end{proof}
 Note that for \corr{every}{general} pairs of representations as in \eqref{eq:cubical}, the two compositions 
 \begin{align*}
   ((\alpha',\beta') \circ (\alpha,\beta)) \otimes ((\gamma',\delta')\circ (\gamma,\delta)) &= f^{*}(\alpha',\beta') \otimes g'_{*}(\alpha,\beta) \otimes g^{*}(\gamma',\delta') \otimes h'_{*}(\gamma,\delta), \\
    ((\alpha',\beta') \otimes (\gamma',\delta')) \circ ((\alpha,\beta) \otimes (\gamma,\delta)) &= f^{*}(\alpha',\beta') \otimes f^{*}(\gamma',\delta') \otimes h'_{*}(\alpha,\beta) \otimes h'_{*}(\gamma,\delta)
 \end{align*}
  are not equal, but naturally isomorphic via $1 \otimes U^{(\gamma',\delta')}_{(\alpha,\beta)} \otimes 1$; see also \eqref{eq:cubical-compositions}.
 \begin{theorem}
   There exists a cubical tricategory $\mathfrak{C}$, where the objects are universal $\Cst$-quantum groups, the 2-category of morphisms between two universal $\Cst$-quantum groups $\UG[F]$ and $\UG$ is $\mathfrak{C}(\UG[F],\UG)$,  the composition functors are as in Proposition \ref{proposition:pairs-horizontal},  and the composition of 2-cells is strictly associative.
    \end{theorem}
 \begin{proof}
This is straightforward. For example,  the composition of 2-cells is strictly associative because for any sequence of morphisms of $\Cst$-quantum groups and pairs of representations 
  \begin{align*}
     \xymatrix@C=90pt{
     \UG[E] \rtwocell^{f}_{g}{\qquad (\alpha,\beta)} & \UG[F] \rtwocell^{f'}_{g'}{\qquad(\alpha',\beta')} &\UG[G] \rtwocell^{f''}_{g''}{\qquad  \  (\alpha'',\beta'')} & \UG[H],
     }
   \end{align*}
 a short calculation shows that $((\alpha'',\beta'') \circ (\alpha',\beta')) \circ (\alpha,\beta)$ and $(\alpha'',\beta'') \circ ((\alpha',\beta') \circ (\alpha,\beta))$ both are equal to
$(\alpha,\beta  \widehat{g'}\widehat{g''})  \otimes (\alpha'f,\beta'\widehat{g''}) \otimes (\alpha''f'f,\beta'')$.
 \end{proof}

\section{Reduced braided-commuting representations}
\label{sec:commutation-reduced}

Let $\Qgrp{G}{A}$ and $\Qgrp{H}{B}$ be $\Cst$\nb-quantum groups as before. We now turn to consider braided-commuting pairs of representations of the (reduced)  $\Cst$-algebras  $A$ and $B$, which can be defined similarly as for their universal counterparts $A^{\univ}$ and $B^{\univ}$.
Let $\chi,\chi ' \in \U(\hat  A\otimes \hat  B)$ be bicharacters.
\begin{definition}
\label{definition:pairs-reduced}
 We say that two representations $\alpha$ of $A$ and $\beta$ of $B$ on the same Hilbert space $\Hils$  form \emph{a $(\chi,\chi')$\nb-pair} or \emph{$(\chi,\chi')$\nb-commute}  if
\begin{align*}
  \chi_{12}\multunit[A]_{1\alpha}\multunit[B]_{2\beta} = \multunit[B]_{2\beta}\multunit[A]_{1\alpha}\chi'_{12} 
  \qquad\text{in } \U(\hat  A\otimes \hat  B \otimes\Comp(\Hils)).
\end{align*}
We call $(\alpha,\beta)$ \emph{faithful} if both $\alpha$ and $\beta$ are faithful.
 \end{definition}
We will primarily be interested in the  case where $\chi$ or $\chi'$ is trivial.
Clearly, a $(1,1)$-pair is just a commuting pair of representations. 
\begin{definition}[{\cite{Meyer-Roy-Woronowicz:Twisted_tensor}, \cite{Roy:Codoubles}}]
\label{definition:heisenberg}  Let $\chi \in \U(\hat  A\otimes \hat  B)$ be a bicharacter. A
 \emph{Heisenberg pair for $\chi$} is a $(1,\chi)$\nb-pair, an \emph{anti-Heisenberg pair for $\chi$} is a $(\chi,1)$\nb-pair, and a \emph{Drinfeld pair for $\chi$} is a $(\chi,\chi)$\nb-pair.
  \end{definition}
We define intertwiners and equivalence of $(\chi,\chi')$-pairs similarly as before for $(f,g)$-pairs. Moreover, for every $(\chi,\chi')$-pair $(\alpha,\beta)$, the pair
$(\bar\alpha,\bar\beta)$ defined by $\bar{\alpha}(a)=\alpha(R_{A}(a))^{\transpose}$ and $\bar\beta(b)=\beta(R_{B}(b))^{\transpose}$ for all $a\in A$, $b\in B$ is a $(\chi',\chi)$-pair, and the tensor product of a $(\chi,\chi')$-pair $(\alpha,\beta)$ and a $(\chi',\chi'')$-pair $(\alpha',\beta')$,
\begin{align*}
  (\alpha,\beta) \otimes (\alpha',\beta') &:= ((\alpha\otimes \alpha')\circ \Delta_{A},(\beta\otimes \beta')\circ \Delta_{B})
\end{align*}
is a $(\chi,\chi'')$-pair again.  

\begin{lemma}
  For any two bicharacters $\chi,\chi' \in \U(\hat  A\otimes \hat  B)$, there exists a faithful $(\chi,\chi')$\nb-pair.
\end{lemma}
\begin{proof}
  By \cite{Meyer-Roy-Woronowicz:Twisted_tensor}*{Lemma 3.8}, there exist a faithful anti-Heisenberg pair for $\chi$ and a faithful Heisenberg pair for $\chi'$. The  tensor product of the two is a faithful $(\chi,\chi')$\nb-pair.
\end{proof}
There exist canonical  Heisenberg, anti-Heisenberg and Drinfeld pairs which are unique up to equivalence:
\begin{example}  
\label{example:canonical-heisenberg}
  Let $(\alpha,\beta)$ be a $(1,\chi)$-pair, that is, a Heisenberg pair for $\chi$, and denote by $\iota_{A},\iota_{B}$ the canonical morphisms from $A$ and $B$ to $A\otimes B$. Since $(\iota_{A},\iota_{B})$ is a $(1,1)$\nb-pair, the tensor product
  $(\alpha',\beta')\defeq 
   (\iota_{A},\iota_{B}) \otimes (\alpha,\beta)$
is a Heisenberg pair for $\chi$ again.  \cite{Meyer-Roy-Woronowicz:Twisted_tensor}*{Theorem 4.6} shows that $(\alpha',\beta')$ does not depend on $(\alpha,\beta)$ up to equivalence. We call~$(\alpha',\beta')$ a \emph{canonical Heisenberg pair} associated to~$\chi$, and~$(\bar {\alpha'},\bar {\beta'})$ a \emph{canonical anti\nb-Heisenberg pair} associated to~$\chi$.  The tensor product $(\bar {\alpha'},\bar {\beta'}) \otimes (\alpha',\beta')$ defines the \emph{canonical Drinfeld pair} associated to $\chi$, which plays a fundamental role in the construction of the generalised Drinfeld double in~\cite{Roy:Codoubles}.
\end{example}
The following result is a  strengthening of \cite{Meyer-Roy-Woronowicz:Twisted_tensor}*{Proposition 3.9}:
\begin{proposition}
\label{proposition:pairs-vertical-reduced}
  Let $(\alpha,\beta)$ and $(\alpha',\beta')$ be pairs of representations of $A$ and $B$ on   some Hilbert spaces~$\Hils$ and~$\Hils[K]$, and let $\chi,\chi'' \in \U(\hat  A\otimes \hat  B)$ be   bicharacters such that  $(\alpha,\beta) \otimes (\alpha',\beta')$ is a $(\chi,\chi'')$-pair. Then there exists a bicharacter $\chi'$ such that $(\alpha,\beta)$ is a $(\chi,\chi')$\nb-pair and  $(\alpha',\beta')$ is  a $(\chi',\chi'')$\nb-pair.
\end{proposition}
\begin{proof}
 Choose  a $(1,\chi)$\nb-pair $(\gamma,\delta)$ and a $(\chi',1)$\nb-pair $(\gamma',\delta')$. Then
 $(\gamma,\delta)\otimes (\alpha,\beta) \otimes (\alpha',\beta') \otimes (\gamma',\delta')$
is a $(1,1)$\nb-pair, and by  \cite{Meyer-Roy-Woronowicz:Twisted_tensor}*{Proposition 3.9}, there exists a bicharacter $\chi'$ such that $(\gamma,\delta) \otimes (\alpha,\beta)$ is a $(1,\chi')$\nb-pair and $(\alpha',\beta') \otimes (\gamma',\delta')$ is a $(\chi',1)$\nb-pair.  By  \eqref{eq:pairs-vertical-legs}, the first relation implies
\begin{align*}
  \multunit[A]_{1\gamma}\multunit[B]_{2\delta}\multunit[A]_{1\alpha}\multunit[B]_{2\beta} &=
  \multunit[A]_{1\gamma}\multunit[A]_{1\alpha}\multunit[B]_{2\delta}\multunit[B]_{2\beta} = \multunit[B]_{2\delta}\multunit[B]_{2\beta}\multunit[A]_{1\gamma}\multunit[A]_{1\alpha}\chi'_{12} =
  \multunit[B]_{2\delta}\multunit[A]_{1\gamma} \multunit[B]_{2\beta}\multunit[A]_{1\alpha}\chi'_{12}.
\end{align*}
Since $\multunit[A]_{1\gamma}\multunit[B]_{2\delta} = \multunit[B]_{2\delta}\multunit[A]_{1\gamma}\chi_{12}$, we can conclude that $(\alpha,\beta)$ is a $(\chi,\chi')$\nb-pair. A similar calculation shows that $(\alpha',\beta')$ is a $(\chi',\chi'')$\nb-pair.
\end{proof}

Let now $f,g$ be morphisms from $\UG$ to $\UDG[H]$ and let $\chi=W^{f}$, $\chi'=W^{g}$.
Then the $(\chi,\chi')$-pairs defined above correspond to $(f,g)$-pairs that are reduced in the following sense:
\begin{definition}
   We call an
  $(f,g)$-pair $(\alpha,\beta)$ \emph{reduced} if $\alpha$ and $\beta$
  factorize through the reducing homomorphisms
  $\Lambda_{A} \colon A^{\univ} \to A$ and
  $\Lambda_{B} \colon B^{\univ} \to B$, respectively.
\end{definition}
\begin{lemma} 
 \label{lemma:pairs-reduced-universal}
  Let $\alpha$ and $\beta$ be representations of $A$ and $B$ on the same Hilbert space. Then $(\alpha,\beta)$ is a $(\chi,\chi')$\nb-pair if and only if  $(\alpha\Lambda_{A},\beta\Lambda_{B})$ is an $(f,g)$\nb-pair.
\end{lemma}
\begin{proof}
  The if part is trivial. Suppose that $(\alpha,\beta)$ is a  $(\chi,\chi')$\nb-pair. Choose a $(1,\chi)$\nb-pair $(\alpha',\beta')$ and let  $(\alpha'',\beta''):=(\alpha',\beta')\otimes(\alpha,\beta)$. Then $(\alpha',\beta')$ and $(\alpha'',\beta'')$ are Heisenberg pairs for $\chi$ and $\chi'$, respectively. Denote the respective compositions with $\Lambda_{A}$ or $\Lambda_{B}$ by  $\tilde\alpha',\tilde\beta',\tilde
\alpha'',\tilde \beta''$, respectively. Then by \cite{Meyer-Roy-Woronowicz:Twisted_tensor}*{(4.3)}, 
     \begin{align*}
\unibich^{A}_{1\tilde \alpha'}\unibich^{B}_{2\tilde \beta'} =
    \unibich^{B}_{2\tilde\beta'}\unibich^{A}_{1\tilde\alpha'}\unibich^{f}_{12} \quad\text{and}\quad
\unibich^{A}_{1\tilde\alpha''}\unibich^{B}_{2\tilde\beta''} =
    \unibich^{B}_{2\tilde\beta''}\unibich^{A}_{1\tilde\alpha''}\unibich^{g}_{12},
  \end{align*}
and using \eqref{eq:pairs-vertical-legs}, we conclude that
$\unibich^{f}_{12}\unibich^{A}_{1\tilde \alpha}\unibich^{B}_{2\tilde \beta} =
    \unibich^{B}_{2\tilde\beta}\unibich^{A}_{1\tilde\alpha}\unibich^{g}_{12}$.
\end{proof}
With respect to the tensor product, reduced pairs of braided-commuting representations form a two-sided ideal. To prove this, 
we need the following well-known result, for which we did not find a convenient reference.
\begin{proposition} \label{prop:reduced-universal-coaction}
 There \corr{exists}{exist}  unique morphisms   $\Delta^{\red,\univ}_{A}$ and $\Delta^{\univ,\red}_{A}$  that make the following diagram commute,
\begin{align}
  \label{eq:delta-red-universal}
  \xymatrix@C=40pt@R=18pt{
A^{\univ} \otimes A^{\univ} \ar[d]_{\Id_{A^{\univ}} \otimes \Lambda_{A}} &    A^{\univ}\ar[r]^(0.4){\Delta^{\univ}_{A}} \ar[d]_{\Lambda_{A}} \ar[l]_(0.4){\Delta^{\univ}_{A}} &  A^{\univ}\otimes A^{\univ} \ar[d]^{\Lambda_{A} \otimes \Id_{A^{\univ}}} \\
A^{\univ} \otimes A &  \ar[l]_(0.4){\Delta^{\univ,\red}_{A}}   A \ar[r]^(0.4){\Delta^{\red,\univ}_{A}} & A \otimes A^{\univ}.
  }
\end{align}
\end{proposition}
\begin{proof}
Uniqueness is clear. We only prove existence of $\Delta^{\red,\univ}_{A}$.
 Relations \eqref{eq:corep-right} and \eqref{eq:corep-explicit}, applied to the bicharacter $\maxcorep[A]$, imply
\begin{align*}
  (\Id_{\hat  A} \otimes (\Lambda_{A} \otimes \Id_{A^{\univ}})\Delta^{\univ}_{A})(\maxcorep[A]) &= (\Id_{\hat  A} \otimes \Lambda_{A} \otimes \Id_{A^{\univ}})(\maxcorep[A]_{12}\maxcorep[A]_{13}) = W^{A}_{12}\maxcorep[A]_{13} = \maxcorep[A]_{23}W^{A}_{12}(\maxcorep[A]_{23})^{*}
\end{align*}
and hence
$(\Lambda_{A} \otimes \Id_{A^{\univ}})\Delta^{\univ}_{A}(a) =  \maxcorep[A](\Lambda_{A}(a) \otimes 1_{A^{\univ}})(\maxcorep[A])^{*}$.
\end{proof}
 \begin{corollary} \label{cor:pairs-reduced-ideal}
 Let $(\alpha,\beta)$ be an $(f,g)$\nb-pair and let $(\alpha',\beta')$ be a $(g,h)$\nb-pair. If one of the two is reduced, then so is $(\alpha,\beta) \otimes (\alpha',\beta')$.
 \end{corollary}

Let us call a morphism $f\in \Mor( A^{\univ}, \hat  B^{\univ})$ of $\Cst$-bialgebras \emph{reduced} if the compositions $\hat \Lambda_{B} \circ f\in \Mor( A^{\univ}, \hat  B)$ and $\hat \Lambda_{A} \circ \hat  f \in \Mor( B^{\univ}, \hat  A)$ factorize through $\Lambda_{A}$ and $\Lambda_{B}$, respectively, such that we obtain commutative diagrams
\begin{align*}
  \xymatrix@R=10pt{
     A^{\univ} \ar[r]^{f} \ar[d]_{\Lambda_{A}} & \hat  B^{\univ} \ar[d]^{\hat  \Lambda_{B}}  &&
     B^{\univ} \ar[r]^{\hat  f} \ar[d]_{\Lambda_{B}} & \hat  A^{\univ} \ar[d]^{\hat  \Lambda_{A}} 
     \\
   A \ar[r]^{f^{\red}} & \hat  B &  &  B \ar[r]^{\hat  f^{\red}} & \hat  A,}
\end{align*}
and denote by $\tau\in \Mor( A^{\univ}, \hat  B^{\univ})$ the trivial morphism, given by $a\mapsto \varepsilon^{\univ}_{A}(a)1_{\hat  B^{\univ}}$.
\begin{remark}
  \begin{enumerate}
  \item In case $\DuG[H]=\G$, the identity $f=\Id_{A^{\univ}}$  evidently is reduced.
  \item It may happen that $f$ factorizes through $\Lambda_{A}$ and $\hat  f$
    does not factorize through $\Lambda_{B}$. For example, if $\G$ is trivial,
    then $f=f^{\red}$ but $\hat  f=\varepsilon^{\univ}_{B}$ need not \corr{descent}{descend} to  $B$.
  \end{enumerate}
\end{remark}
The following result shows that   Heisenberg and anti-Heisenberg pairs are automatically reduced:
\begin{proposition} \label{proposition:heisenberg-reduced}
  If $f$ is reduced, then every Heisenberg pair and every anti-Heisenberg pair for $f$ is reduced.
\end{proposition}
\begin{proof}
Let $(\alpha,\beta)$ be an anti-Heisenberg pair for $f$. Then by  \eqref{eq:anti-hb-universal},
\begin{align*}
  (f \otimes \alpha)\Delta^{\univ}_{A}(a)  = \unibich^{B}_{1\beta}(1_{\hat  B^{\univ}} \otimes \alpha(a))(\unibich^{B}_{1\beta})^{*}.
\end{align*}
We apply $\hat  \Lambda_{B}$ on the first tensor factor and obtain
\begin{align*}
  (f^{\red}\circ\Lambda_{A} \otimes \alpha)\Delta^{\univ}_{A}(a) = \maxcorep[B]_{1\beta}(1_{\hat  B} \otimes \alpha(a))(\maxcorep[B]_{1\beta})^{*}.
\end{align*}
By Proposition \ref{prop:reduced-universal-coaction},  $(\Lambda_{A} \otimes \Id_{A^{\univ}})\Delta^{\univ}_{A}$ factorizes through $\Lambda_{A}$. Hence, so does $\alpha$. Repeating this argument for the $(\hat  f,\hat  \tau)$-pair  $(\beta,\alpha)$, we find that  $\beta$ factorizes through $\Lambda_{B}$. 
\end{proof}


Of course,   relations \eqref{eq:hb-universal}--\eqref{eq:id-anti-hb-universal} have reduced counterparts which include, for example, the following generalization of \cite{Meyer-Roy-Woronowicz:Homomorphisms}*{Theorem 5.3 (33)}:
\begin{lemma}
  \label{lemma:delta-implementation}
  Let $\chi \in \U(\hat  A \otimes \hat  B)$ be a bicharacter, denote by $\Delta_{R}$ and $\Delta_{L}$ the associated right and left quantum group homomorphisms, and let
 $(\pi,\hat \pi)$  a Heisenberg pair for $\multunit[A]$. Then for all $a\in A$,
  \begin{align*}
  (\pi \otimes \Id_{\hat  B} )\Delta_{R}(a) &= \chi_{\hat \pi 2}(\pi(a) \otimes 1)\chi_{\hat \pi 2}^{*}, \\
  (\Id_{\hat  B} \otimes \bar \pi)\Delta_{L}(a) &= \hat \chi_{1\bar {\hat \pi}}(1\otimes \bar \pi(a))\hat \chi_{1\bar {\hat \pi}}^{*}.
  \end{align*}
\end{lemma}
\begin{proof}
We only prove the second equation; the first one follows similarly. 
By Lemma \ref{lemma:pairs-reduced-universal}, $(\bar\pi,\bar{\hat{\pi}})$ is an anti-Heisenberg pair for the identity on $A$.   Denote by $f\in \Mor(A^{\univ},\hat  B^{\univ})$  the morphism corresponding to $\chi$ and
apply $ f\otimes \Id$ to \eqref{eq:id-anti-hb-universal} to find that
\begin{align} \label{eq:delta-implementation-universal}
  (f \otimes \bar \pi \Lambda_{A})\Delta^{\univ}_{A}(a) &= \hat \unibich^{A}_{f\bar {\hat \pi}}(1\otimes \bar \pi(\Lambda_{A}(a)))(\hat{\unibich}^{A}_{f\bar {\hat \pi}})^{*}.
\end{align}
Now, apply $\hat\Lambda_{B} \otimes \Id_{A}$ and use \eqref{eq:morphism-bicharacter} and \eqref{eq:morphism-deltal-deltar-universal} to get the desired relation.
\end{proof}

\section{The maximal twisted tensor product of $\mathrm{C}^{*}$-algebras}
  \label{sec:max}

Let  $\Qgrp{G}{A}$ and~$\Qgrp{H}{B}$ be $\Cst$\nb-quantum groups, $(C,\gamma)$ a $\G$-$\Cst$-algebra,~$(D,\delta)$ a  $\G[H]$\nb-$\Cst$\nb-algebra and  $\chi \in \U(\hat  A \otimes \hat  B)$  a bicharacter. Note that we do not require $\gamma$ or $\delta$ to be injective. 

We now define the maximal or universal counterpart to the minimal twisted tensor product of $(C,\gamma)$ and $(D,\delta)$ with respect to $\bichar$  introduced in  \cite{Meyer-Roy-Woronowicz:Twisted_tensor}. The  following commutation relation
   \eqref{eq:V_comm_rep} is the key.
\begin{lemma}
Let $E$ be a $\Cst$\nb-algebra, $\varphi\in\Mor(C,E)$ and~$\psi\in\Mor(D,E)$ such that
  \begin{equation}
   \label{eq:V_comm_rep}
    [(\varphi\otimes\bar \alpha)\gamma(c), 
     (\psi\otimes\bar \beta)\delta(d)] =0 \quad\text{for all } c\in C,d\in D
  \end{equation}
and one  $\chi$-anti-Heisenberg pair  $(\bar \alpha,\bar \beta)$. Then this relation holds for every $\chi$-anti-Heisenberg pair $(\bar \alpha,\bar \beta)$.
\end{lemma}
\begin{proof}
  Denote by $\iota_{A}$ and $\iota_{B}$  the canonical morphisms from $A$ and $B$ to $A\otimes B$ and regard $(\iota_{A},\iota_{B})$ as a $(1,1)$\nb-pair. Then  \eqref{eq:action-slices} implies that
   \eqref{eq:V_comm_rep} holds if for all $c\in C$ and $d\in D$, the elements
   \begin{align*}
 ((\varphi\otimes\bar \alpha) \circ \gamma  \otimes \bar \iota_{A}) \gamma(c) =
     (\varphi\otimes(\bar \alpha \otimes \bar \iota_{A})\Comult[A])\gamma(c) 
   \end{align*}
commutes with the element
   \begin{align*}
((\psi\otimes\bar \beta)\delta  \otimes \bar \iota_{B}) \delta(d)=
     (\psi \otimes (\bar \beta\otimes \bar \iota_{B})\Comult[B])\delta(d).
   \end{align*}
  But Example \ref{example:canonical-heisenberg} and relations \eqref{eq:antipode-tensor}, \eqref{eq:pairs-bar-equivalence} imply that the anti-Heisenberg pair
\begin{align*}
  ((\bar \alpha \otimes \bar \iota_{A})\Comult[A], (\bar \beta\otimes\bar \iota_{B})\Comult[B]) = (\bar \alpha,\bar \beta) \otimes (\bar \iota_{A},\bar \iota_{B}) \cong \corr{$\overline{(\iota_{A},\iota_{B})(\alpha,\beta)}$}{\overline{(\iota_{A},\iota_{B})\otimes (\alpha,\beta)}}
\end{align*}
does not depend on the anti-Heisenberg pair $(\bar \alpha,\bar \beta)$ up to equivalence. 
Note that in this argument, we could have replaced $(\bar \iota_{A},\bar \iota_{B})$  by the equivalent pair $(\iota_{A},\iota_{B})$ everywhere.
\end{proof}
 \begin{definition}  \label{definition:pairs-action}
A \emph{$\bichar$\nb-com\-mutative representation} of  $(C,\gamma)$
and~$(D,\delta)$  consists of  a $\Cst$\nb-algebra  $E$ and \corr{}{morphisms} $\varphi\in\Mor(C,E)$ and~$\psi\in\Mor(D,E)$  such that    \eqref{eq:V_comm_rep}
holds for some (and then for every) anti-Heisenberg pair $(\bar \alpha,\bar \beta)$ for $\chi$.  

A \emph{morphism} of $\bichar$-commutative representations $(E,\varphi,\psi)$ and $(E',\varphi',\psi')$ is a morphism $\Psi \in \Mor(E,E')$   satisfying $\varphi'=\Psi\circ\varphi$ and  $\psi'=\Psi\circ \psi$. If $\Psi$ can be chosen to be an isomorphism, we call $(\varphi,\psi)$ and $(\varphi',\psi')$
 \emph{equivalent}.
\end{definition}
Let us  consider some simple examples.
\begin{example} \label{example:one-pair} Let $\chi$ be the trivial bicharacter $1\in \U(\hat A \otimes \hat B)$.  Then a $1$-commutative representation is just a commuting pair of representations. Indeed, an anti-Heisenberg pair for $\chi$ is given by $\bar \alpha(a)=a\otimes 1$ and $\bar \beta(b)=1\otimes b$, and two representations $(\varphi,\psi)$ of $C$ and $D$ on a common $\Cst$-algebra $\chi$-commute if and only if $(\varphi\otimes \Id_{A})(\gamma(c))_{12}$ and $(\psi \otimes \Id_{B})(\delta(d))_{13}$ commute for all $c\in C$ and $d\in D$, which by \eqref{eq:action-slices} holds if and only if $\varphi(c)$ and $\psi(d)$ commute for all $c\in C$ and $d\in D$.
\end{example}

\begin{example}
Suppose that $(C,\gamma)=(A,\Comult[A])$ and $(D,\delta)=(B,\Comult[B])$. Then
two nondegenerate representations $\varphi$ and $\psi$   of $A$ and $B$,
respectively, on the same Hilbert space form a $\chi$-commutative representation
of $(A,\Comult[A])$ and $(B,\Comult[B])$ if and only if they \corr{form
  Heisenberg}{form a Heisenberg} pair for $\chi$ in the sense of Definition \ref{definition:heisenberg}.
Indeed, 
let $(\bar \alpha,\bar \beta)$ be an anti-Heisenberg pair for $\chi$, that is, a $(\chi,1)$\nb-pair.
Then Proposition \ref{proposition:pairs-vertical-reduced} implies that
$(\varphi \otimes \bar \alpha)\Comult[A](A)$  and $(\psi\otimes \bar \beta)\Comult[B](B)$ commute, that is, $(\varphi,\psi)\otimes(\bar \alpha,\bar \beta)$ is a $(1,1)$\nb-pair, if and only if $(\varphi,\psi)$  is a $(1,\chi)$\nb-pair.
\end{example}

\begin{example} \label{example:pairs-action-discrete}
  Let $\Gamma$ be a discrete group and suppose that $A=C_{0}(\Gamma)$ and $B:=\hat  A=\Cst_{r}(\Gamma)$ are equipped with the usual comultiplications.  Then  $\delta$ corresponds to a grading of $D$ by $G$ and $\gamma$ corresponds to a (left) action of $G$ on $C$, which we write as $(g,c)\mapsto g\cdot c$. In the notation of Example \ref{example:anti-hb-discrete}, an anti-Heisenberg pair $(\bar \alpha,\bar \beta)$ for $\multunit[A]=\sum_{g} \rho_{g} \otimes \delta_{g}$ is given by 
  $\bar \alpha(\delta_{h})=\delta_{h^{-1}}$ and $\bar \beta(\rho_{g})=\rho_{g}$. Hence, a pair of representations $(\varphi,\psi)$ of $C$ and $D$ is a $\multunit[A]$\nb-commutative representation if and only if for every $c\in C$ and every $d\in D$ of degree $g$,
  \begin{align*}
 \sum_{h}
    \varphi(h\cdot c)\psi(d)\otimes \delta_{h^{-1}}\rho_{g} =
\sum_{h'}\psi(d)\varphi(h'\cdot c) \otimes \rho_{g}\delta_{h'^{-1}},
  \end{align*}
that is, if and only if $\psi(d) \varphi(c')= \varphi(g\cdot c')\psi(d)$ for all $c'\in C$ and all $d\in D$ of degree $g$.
\end{example}

Every $\chi$-commutative representation is a crossed tensor product of $C$ and $D$ in the sense of \cite{Meyer-Roy-Woronowicz:Twisted_tensor}*{Definition 2.1}:
\begin{lemma} 
 \label{lem:Cstar_alg}  Let $(E,\varphi,\psi)$ be a $\chi$\nb-commutative representation of $(C,\gamma)$ and $(D,\delta)$. Then
  \begin{equation}
   \label{eq:Cstar_alg}
    \varphi(C)\cdot\psi(D)=\psi(D)\cdot\varphi(C)\subseteq\Mult(E).
  \end{equation}  
\end{lemma}
\begin{proof} Let~$(\bar \alpha,\bar \beta)$ be a $\bichar$\nb-anti\nb-Heisenberg  pair on~$\Hils$. Since $\bar \alpha(A)\cdot\Comp(\Hils)
  =\Comp(\Hils)$, 
   the Podle\'s condition~\eqref{eq:Podles_cond} for~$\gamma$ gives
  \begin{align*}
      (\Id_{C}\otimes\bar \alpha)\gamma(C)\cdot (1_{C}\otimes 
       \Comp(\Hils))
 &=C\otimes\Comp(\Hils).
  \end{align*} Similarly, $(\Id_{D}\otimes\bar \beta)\psi(D)\cdot (1_{D}\otimes 
  \Comp(\Hils))=D\otimes\Comp(\Hils)$.     
  Using \eqref{eq:V_comm_rep}, we conclude 
   \begin{align*}
\varphi(C)\cdot\psi(D)\otimes\Comp(\Hils)     
&=     (\varphi\otimes\bar \alpha)\gamma(C)\cdot
     (\psi\otimes\bar \beta)\delta(D)\cdot(1_{E}\otimes\Comp(\Hils))
\\ &=     (\psi\otimes\bar \beta)\delta(D)
     \cdot (\varphi\otimes\bar \alpha)\gamma(C)\cdot(1_{E}\otimes\Comp(\Hils))
\\ &  = \psi(D)\cdot\varphi(C)\otimes\Comp(\Hils).      
    \end{align*}
Slicing the second leg by~$\omega\in\Bound(\Hils)_{*}$ completes  the proof.
 \end{proof} 
 
 The $\bichar$\nb-commutative representations with morphisms as above form a category.  Standard cardinality arguments show that there exists a $\bichar$\nb-commutative representation which is universal in the sense that it is an initial object in this category. 
\begin{definition}
 \label{def:max_tens} The \emph{maximal twisted tensor product} of a $\G$\nb-$\Cst$\nb-algebra $(C,\gamma)$ and an $\G[H]$\nb-$\Cst$\nb-algebra $(D,\delta)$ with respect to a bicharacter $\chi\in \U(\hat  A\otimes \hat  B)$ is the $\Cst$\nb-algebra  
 \begin{align*}
  C\boxtimes^{\bichar}_{\maxtensor}D\defeq j^\univ_{C}(C)\cdot j^\univ_{D}(D) 
 \end{align*}
generated by a universal $\chi$\nb-commutative representation $(E^{\univ},j^{\univ}_C,j^{\univ}_D)$ of $(C,\gamma)$ and $(D,\delta)$.
\end{definition}
Note that by Lemma \ref{lem:Cstar_alg}, $(C\boxtimes^{\bichar}_{\maxtensor}D,j^{\univ}_C,j^{\univ}_D)$ is a crossed product of $C$ and $D$ in the sense of \cite{Meyer-Roy-Woronowicz:Twisted_tensor}*{Definition 2.1}.

By definition, we obtain for every $\bichar$\nb-commutative representation $(E,\varphi,\psi)$ of $(C,\gamma)$ and $(D,\delta)$ a unique morphism
\begin{align*}
  \varphi \chiprod \psi \in \Mor(C \boxtimes^{\chi}_{\max} D,E) \quad \text{such that } (\varphi \chiprod \psi)j^{\univ}_{C} = \varphi \text{ and } (\varphi \chiprod \psi)j^{\univ}_{D} = \psi.
\end{align*}

The assignment $((C,\gamma),(D,\delta)) \mapsto C\boxtimes^{\chi}_{\max} D$ extends to a bifunctor as follows.

If also $(C',\gamma')$ is a $\G$-$\Cst$-algebra and $(D',\delta')$ is a $\G[H]$-$\Cst$-algebra, and if $\rho \in \Mor(C,C')$ and $\theta \in \Mor(D,D')$ are equivariant, we obtain a commutative diagram 
\begin{align*}
  \xymatrix{ C
    \ar[r]^(0.4){j^{\univ}_C} \ar[d]_{\rho} & C\boxtimes^{\chi}_{\max} D \ar[d]|{\rho \boxtimes^{\chi}_{\max} \theta}  & D \ar[l]_(0.4){j^{\univ}_D } \ar[d]^{\theta} \\
    C' \ar[r]^(0.4){j^{\univ}_{C'}}  & C'\boxtimes^{\chi}_{\max} D' & D', \ar[l]_(0.4){j^{\univ}_{D'}}  }
\end{align*}
where the representations $j^{\univ}_{C'}\circ \rho$ and $j^{\univ}_{D'}\circ \theta$   $\chi$-commute  and
\begin{align*}
\rho \boxtimes^{\chi}_{\max} \theta =  (j_{C'}^{\univ}\circ \rho) \chiprod (j_{D'}^{\univ}\circ \theta).
\end{align*}
\begin{proposition}
  The assignments $((C,\gamma),(D,\delta)) \mapsto C\boxtimes^{\chi}_{\max} D$ and $(\rho,\theta) \mapsto \rho \boxtimes^{\chi}_{\max} \theta$ form a bifunctor  from
$\Cstcat(\G) \times\Cstcat(\G[H])$ to $\Cstcat$.
\end{proposition}

As one should expect, there exists a canonical quotient map from the maximal twisted tensor product $C\boxtimes^{\chi}_{\max} D$ to the minimal twisted tensor product $C\boxtimes^{\chi}_{\min} D$ introduced in \cite{Meyer-Roy-Woronowicz:Twisted_tensor}. To prove this, we use the following analogue of Proposition \ref{proposition:pairs-vertical-reduced}.
\begin{lemma} \label{lemma:pairs-vertical-action}
  Let $(\varphi,\psi)$ be a $\chi$\nb-commutative representation of $(C,\gamma)$ and $(D,\delta)$,  and let $(\alpha,\beta)$  be an $(\chi,\chi')$\nb-pair on some Hilbert space $\Hils$. Then 
  \begin{align*}
    (\varphi,\psi)\otimes(\alpha,\beta) &:=((\varphi\otimes \alpha)\gamma, (\psi\otimes \beta)\delta)
  \end{align*}
is a $\chi'$\nb-commutative representation of $(C,\gamma)$ and $(D,\delta)$. In particular, there exists a morphism
\begin{align*}
  (j^{\univ}_C \otimes \alpha)\gamma \chiprod (j^{\univ}_D\otimes \beta)\delta \in \Mor(C \boxtimes^{\chi'}_{\max} D, (C\boxtimes^{\chi}_{\max} D) \otimes \Comp(\Hils)),
\end{align*}
where $j^{\univ}_{C}$ and $j^{\univ}_{D}$ denote the canonical morphisms from $C$ and $D$, respectively, to $C\boxtimes^{\chi}_{\max} D$.
\end{lemma}
\begin{proof}
Let $(\bar  \alpha',\bar  \beta')$ be a $(\chi',1)$\nb-pair. Then $(\alpha,\beta)\otimes (\bar \alpha',\bar  \beta')$ is a $(\chi,1)$\nb-pair and hence $((\varphi \otimes \alpha)\gamma\otimes \bar \alpha')\gamma(c)$ and $((\psi\otimes\beta)\delta\otimes \bar \beta')\delta(d)$ commute for all $c\in C$ and $d\in D$.
\end{proof}
The minimal twisted tensor product $C \boxtimes^{\chi}_{\min} D$ of $(C,\gamma)$ and $(D,\delta)$ with respect to $\chi$ was introduced in \cite{Meyer-Roy-Woronowicz:Twisted_tensor} as follows. 
Choose a $\bichar$\nb-Heisenberg pair $(\alpha,\beta)$ on $\Hils$ and  define morphisms $j_{C}$ and $j_{D}$ from $C$ and $D$ to $C\otimes D\otimes\Comp(\Hils)$ by 
 \begin{align*} 
   j_{C}(c) \defeq (\Id_{C}\otimes\alpha)\gamma(c)_{13}, \quad
j_{D}(d) \defeq (\Id_{D}\otimes\beta)\delta(d)_{23} 
\quad\text{for all~$c\in C$, $d\in D$.}
 \end{align*} 
Then  the minimal twisted  tensor product is  the $\Cst$\nb-algebra
\begin{align*}
  C \boxtimes^{\chi}_{\min} D = j_{C}(C) \cdot j_{D}(D) \subseteq \Mult(C\otimes D\otimes\Comp(\Hils)).
\end{align*}
 This $\Cst$\nb-algebra does not depend on the choice of $(\alpha,\beta)$ \cite{Meyer-Roy-Woronowicz:Twisted_tensor}*{Section 4}. 
\begin{proposition} 
 \label{prop:qmap} 
For every $\G$-$\Cst$-algebra $(C,\gamma)$ and $\G[H]$-$\Cst$-algebra $(D,\delta)$, 
there exists a unique quotient map $C\boxtimes^{\chi}_{\maxtensor}D\to C\boxtimes^{\chi}_{\mintensor}D$  that makes the following diagram commute:
\begin{align*}
\xymatrix@C=40pt@R=15pt{ C \ar[r]^{j_{C}^\univ} \ar[rd]_{j_{C}} &   C\boxtimes^{\chi}_{\maxtensor} D
  \ar@{->}[d] &   D \ar[l]_{j_{D}^\univ} \ar[ld]^{j_{D}}\\
  & C\boxtimes^{\chi}_{\mintensor} D. &
}
\end{align*}
These quotient maps form a natural transformation from the maximal to the minimal twisted tensor product.
\end{proposition}
\begin{proof}
The natural morphisms $\iota_{C},\iota_{D}$ from $C$ and $D$ to $C\otimes D$ form a 1-commutative representation by Example \ref{example:one-pair}, and $(j_{C},j_{D}) = (\iota_{C},\iota_{D})\otimes(\alpha,\beta)$
is a $\chi$-commutative representation by Lemma \ref{lemma:pairs-vertical-action}. The desired quotient map is $j_{C} \chiprod j_{D}$.
\end{proof}

Let $\chi \in \U(\hat  A \otimes \hat  B)$ be a bicharacter as before, and $\hat \chi=\sigma(\chi)^{*} \in \U(\hat  A \otimes \hat  B)$ its opposite.
\begin{proposition}  \label{prop:symmetry}
  There exists a natural isomorphism $C \boxtimes^{\chi}_{\maxtensor} D \cong D \boxtimes^{\hat {\chi}}_{\maxtensor} C$ which intertwines the canonical maps of $C$ and $D$.
\end{proposition}
\begin{proof}
By Lemma  \ref{lemma:pairs-symmetry}, $(\bar \alpha,\bar \beta)$ is an anti-Heisenberg pair for $\chi$ if and only if $(\bar \beta,\bar \alpha)$ is an anti-Heisenberg pair for $\hat \chi$. Therefore,
two  representations $(\varphi,\psi)$ form a $\chi$\nb-commutative representation of $(C,\gamma)$ and $(D,\delta)$  if and only if $(\psi,\varphi)$ is a $\hat \chi$\nb-commutative representation of $(D,\delta)$ and $(C,\gamma)$.
\end{proof}

In the case where $\G$ is a finite group and $\G[H]$ is its dual, we can describe the commutation relations between elements of $C$ and $D$ in the maximal twisted tensor product  as follows.
\begin{example}
  Suppose that $\Gamma$ is a finite group, that $A=C_{0}(\Gamma)$ and $B=\hat  A=\Cst_{r}(\Gamma)$ as in Example \ref{example:pairs-action-discrete} and that $\chi=\multunit[A]$. Then $\gamma$ and $\delta$ correspond to an action of $\Gamma$ on $C$ and a grading on $D$,  and the algebraic tensor product $C \odot D$ can be endowed with the structure of a \Star{}algebra with multiplication and involution given by
  \begin{align} \label{eq:twisted-group-case}
    (c \odot d)(c'\odot d') &= c(g\cdot c')\odot dd' &&\text{and} & (c\odot d)^{*} &= (g^{-1}\cdot c^{*}) \odot d^{*}
  \end{align}
for all $c,c' \in C$ and $d,d'\in D$ such that $d$ has degree $g\in \Gamma$. Example  \ref{example:pairs-action-discrete} shows that $\chi$\nb-commutative representations of $C$ and $D$ correspond to nondegenerate representations of $C\odot D$, and therefore the maximal twisted tensor product  $C\boxtimes_{\max}^{\chi} D$ is  canonically isomorphic to the enveloping $\Cst$\nb-algebra $\Cst(C \odot D)$.

 In the case $\Gamma=\Z/2\Z$, the coactions  $\gamma$ and $\delta$ correspond to  $\Z/2\Z$\nb-gradings on $C$ and $D$, and \eqref{eq:twisted-group-case} takes the form
 \begin{align*}
   (c\odot d)(c'\odot d') &= (-1)^{|c'||d|}cc'\odot dd', & (c\odot d)^{*} &= (-1)^{|c||d|} c^{*}\odot d^{*},
 \end{align*}
where $|x|\in \{0,1\}$ denotes the degree of a homogeneous element $x$.
\end{example}

\section{Some properties and special cases of the maximal twisted tensor product}
\label{sec:prop}

Throughout this section,
let $\Qgrp{G}{A}$ and $\Qgrp{H}{B}$ be $\Cst$\nb-quantum groups with a bicharacter $\chi \in \U(\hat  A\otimes \hat  B)$ as before.

 \subsection{Exactness}
 Let $(C,\gamma)$ be a $\G$\nb-$\Cst$\nb-algebra with  an ideal $I\subseteq C$   that is $\G$-invariant in the sense
 that $\gamma(c)= c \otimes 1$ for all $c\in I$.  Denote by $i \colon  I\to C$ the inclusion and by $q \colon C \to C/I$ the quotient
 map. By assumption on $I$,  $\gamma$ descends to a $*$-homomorphism $\tilde\gamma
 \colon C/I \to \Mult(C/I\otimes A)$. Clearly, $\tilde \gamma$ is a coaction, satisfies the Podle\'s condition, and makes $q$ equivariant.
\begin{proposition}
 \label{prop:exactness} For every $\G[H]$-$\Cst$-algebra $(D,\delta)$, the sequence
\begin{align*}
 0 \to I \boxtimes_{\maxtensor}^{\chi} D \xrightarrow{i \boxtimes_{\maxtensor}^{\chi} \Id_{D}} C \boxtimes_{\maxtensor}^{\chi} D
\xrightarrow{q \boxtimes_{\maxtensor}^{\chi} \Id_{D}} (C/I) \boxtimes_{\maxtensor}^{\chi} D \to 0
\end{align*} is exact.
\end{proposition}
\begin{proof} 
 Denote  by $E_{0},E_{1}$ and $E_{2}$, respectively, the $\Cst$-algebras in the sequence above, read from the left to the right, and  by 
\begin{align*}
  \varphi^{\univ}_{0} \in \Mor(I, E_{0}), \quad
  \varphi^{\univ}_{1} \in \Mor(C, E_{1}), \quad
  \varphi^{\univ}_{2} \in \Mor(C/I,  E_{2}), \quad \psi^{\univ}_{i} \in \Mor(D,E_{i})
\end{align*}
for $i=0,1,2$ the canonical morphisms.

First, the map $q \boxtimes_{\maxtensor}^{\chi} \Id_{D}$ is surjective because 
\begin{align*} 
E_{2} = \varphi^{\univ}_{2}(C/I) \cdot \psi^{\univ}_{2}(D)
       &= \varphi^{\univ}_{2}(q(C))\cdot \psi^{\univ}_{2}(D) \\ 
       &= (q \boxtimes_{\maxtensor}^{\chi} \Id_{D})(\varphi^{\univ}_{1}(C)\cdot \psi^{\univ}_{1}(D)) 
       = (q \boxtimes_{\maxtensor}^{\chi} \Id_{D})(E_{1}).
\end{align*}

Next, we show that the map $i \boxtimes_{\maxtensor}^{\chi} \Id_{D}$ is injective. Since the natural map  $r\colon C \to\Mult(I)$ is equivariant and $(\varphi^{\univ}_{0},\psi^{\univ}_{0})$ is a $\chi$-commutative representation, also $(\varphi^{\univ}_{0}\circ r,\psi^{\univ}_{0})$ is a $\chi$-commutative representation. The induced morphism $(\varphi^{\univ}_{0}\circ r) \chiprod \psi^{\univ}_{0}$ from $C
\boxtimes_{\maxtensor}^{\chi} D$ to $I\boxtimes_{\maxtensor}^{\chi} D$  is a left inverse to $i \boxtimes_{\maxtensor}^{\chi} \Id_{D}$.

Let us finally  prove exactness in the middle.  Clearly, the ideal
\begin{align*} J\defeq (i \boxtimes_{\maxtensor}^{\chi} \Id_{D})(I \boxtimes_{\maxtensor}^{\chi} D) =\varphi^{\univ}_{1}(I)\cdot \psi^{\univ}_{1}(D)\subseteq C\boxtimes_{\maxtensor}^{\chi} D
\end{align*} 
is  contained in $\ker (q \boxtimes_{\maxtensor}^{\chi} \Id_{D})$.  To deduce the converse inclusion, consider the  natural maps
\begin{align*}
  \check \varphi_{2}^{\univ} &\colon C \xrightarrow{\varphi^{\univ}_{1}} \Mult(C \boxtimes_{\maxtensor}^{\chi} D) \to \Mult((C\boxtimes_{\maxtensor}^{\chi} D)/J), \\
  \tilde \psi^{\univ}_{2} &\colon D \xrightarrow{\psi^{\univ}_{1}} \Mult(C\boxtimes_{\maxtensor}^{\chi} D) \to \Mult((C \boxtimes_{\maxtensor}^{\chi} D)/J).
\end{align*} 
Since $\corr{$\varphi^{\univ}_{2}$}{\varphi^{\univ}_{1}}(I)(C\boxtimes_{\maxtensor}^{\chi} D) \subseteq J$, the map $\check
\varphi^{\univ}_{2}$ factorizes through the quotient map $q\colon C\to C/I$ and yields a map ${\tilde{\varphi}}^{\univ}_{2}\colon C/I \to \Mult((C \boxtimes_{\maxtensor}^{\chi} D)/J)$. Since   $(\varphi^{\univ}_{1},\psi^{\univ}_{1})$ is a $\chi$\nb-commutative representation and  the quotient map $q$ is equivariant,  $(\check \varphi^{\univ}_{2},\tilde
\psi^{\univ}_{2})$ and
$({\tilde{\varphi}}^{\univ}_{2},\tilde \psi^{\univ}_{2})$ are $\chi$\nb-commutative representations.  The induced morphism
\begin{align*}
\pi={\tilde{\varphi}}^{\univ}_{2} \chiprod \tilde \psi^{\univ}_{2} \in \Mor((C/I)\boxtimes_{\maxtensor}^{\chi} D, (C \boxtimes_{\maxtensor}^{\chi} D)/J)
\end{align*} 
makes the following diagram commute,
\begin{align*}
  \xymatrix@R=10pt{C\boxtimes_{\maxtensor}^{\chi} D \ar[rr]^{q \boxtimes_{\maxtensor}^{\chi} \Id_{D}} \ar[rd] && (C/I)\boxtimes_{\maxtensor}^{\chi} D \ar[ld]^{\pi} \\ & (C\boxtimes_{\maxtensor}^{\chi} D)/J,}
\end{align*} 
whence $\ker (q \boxtimes_{\maxtensor}^{\chi} \Id_{D}) \subseteq J$.
\end{proof}

\subsection{Relation with the universal crossed product}
 The universal crossed product construction can be regarded as a special case of a maximal twisted tensor product as follows.

Let $(C,\gamma)$ be a $\G$-$\Cst$-algebra.
Recall that a \emph{covariant representation} of $(C,\gamma)$ on a Hilbert space $\Hils$ consists of a nondegenerate representation $\varphi\colon C\to \Bound(\Hils)$ and a right representation $\corep{U}\in \U(\Comp(\Hils)\otimes A)$ of $\G$ on $\Hils$ satisfying
\begin{align}
  \label{eq:covariance}
  \corep{U}(\varphi(c) \otimes 1)\corep{U}^{*} &= (\varphi \otimes \Id_{A})\gamma(c) \quad\text{for all } c\in C.
\end{align}
 The \emph{universal crossed product} $C\rtimes \hat  A^{\univ}$ is the $\Cst$\nb-algebra $\varphi(C) \cdot \rho(\hat  A^{\univ})$ generated by a universal covariant representation $(\varphi,\corep{U})$ of $(C,\gamma)$, where $\rho$ is the representation of $\hat  A^{\univ}$ on $\mathcal{H}$ determined by $(\rho \otimes \Id_{A})(\dumaxcorep[A])=\corep{U}$. Here,  $\varphi$ is faithful if and only 
if $\gamma$ is.

  The $\Cst$\nb-algebra $\hat  A^{\univ}$ of the universal dual quantum group $\UDG$ of $\G$ can be  regarded as a $\DuG$\nb-$\Cst$\nb-algebra via the coaction
  \begin{align*}
   \delta'\defeq(\Id_{\hat  A^\univ}\otimes\hat {\Lambda}_{A})\circ\hat \Delta^{\univ}_{A} \in \Mor(\hat  A^{\univ}, \hat {A}^\univ\otimes\hat {A}),
  \end{align*}
 where $\hat {\Lambda}_{A}\in\Mor(\hat {A}^\univ,\hat {A})$ denotes the reducing morphism. Note that $\delta'$ need not be injective.

We now consider the maximal twisted tensor product of $(C,\gamma)$ and $(\hat  A^{\univ},\delta')$ with respect to the bicharacter $W^{A} \in \U(\hat  A\otimes  A)$.
 \begin{theorem}
  \label{the:cov_rep_vs_br_comm}  There exists a unique isomorphism $C \boxtimes^{\multunit[A]}_{\max} \hat  A^{\univ} \cong C \rtimes \hat  A^{\univ}$ that intertwines the canonical morphisms of $C$ and $\hat  A^{\univ}$ to both sides.
 \end{theorem}    
\begin{proof}
  It suffices to prove the following assertion: If $\varphi$ and $\rho$ are representations of $C$ and $\hat  A^{\univ}$ on some Hilbert space $\Hils$, then  $(\varphi,\rho)$ is a $\multunit[A]$\nb-commutative representation if and only if $\varphi$ and $\corep{U}=(\rho\otimes \Id_{A})(\dumaxcorep[A])$ form a covariant representation of $(C,\gamma)$.

So, suppose that $\varphi$ and $\rho$ are representations of $C$ and $\hat  A^{\univ}$ on a Hilbert space $\Hils$ and
   let $(\bar \alpha,\bar \beta)$ be a faithful anti-Heisenberg pair for $\multunit[A]$ on a Hilbert space $\Hils[K]$.
Since $\hat  A^{\univ}$ is generated by slices of $\dumaxcorep[A]$, the representations
$(\varphi,\rho)$  form a $\multunit[A]$\nb-commutative representation if and only if
$(\varphi \otimes \bar \alpha)\gamma(C)_{12}$ commutes with
\begin{align*}
  ((\rho \otimes \bar \beta)\delta' \otimes \Id_{A})(\dumaxcorep[A])
\end{align*}
in $\Mult(\Comp(\Hils) \otimes \Comp(\Hils[K]) \otimes A)$.  Since $\dumaxcorep[A]$ is a bicharacter, the operator above is equal to
\begin{align*}
    (\rho \otimes \bar \beta \hat \Lambda_{A} \otimes \Id_{A})(\dumaxcorep[A]_{23}\dumaxcorep[A]_{13}) = \multunit[A]_{\bar \beta 3}\dumaxcorep[A]_{\rho 3}.
\end{align*}
Thus, $(\varphi,\rho)$  is a $\multunit[A]$\nb-commutative representation if and only if 
$(\varphi \otimes \bar \alpha)\gamma(C)_{12}$ commutes with $\multunit[A]_{\bar \beta 3}\dumaxcorep[A]_{\rho 3}$ or, equivalently, if and only if
\begin{align} \label{eq:aux-universal-crossed}
  \dumaxcorep[A]_{\rho 2}(\varphi \otimes \bar \alpha)\gamma(c)_{13}  (\dumaxcorep[A]_{\rho 2})^{*} = 
\Dumultunit[A]_{2\bar \beta}(\varphi \otimes \bar \alpha)\gamma(c)_{13}(\Dumultunit[A]_{2\bar \beta})^{*}
\end{align}
for all $c\in C$. 
Since $(\bar \alpha,\bar \beta)$ is an anti-Heisenberg pair for $\multunit[A]$,  \eqref{eq:id-anti-hb-universal} implies
\begin{align*}
  \Dumultunit[A]_{2\bar \beta}(\varphi \otimes \bar \alpha)\gamma(c)_{13} (\Dumultunit[A]_{2\bar \beta})^{*} =
  (\varphi \otimes (\Id_{A} \otimes \bar \alpha)\Comult[A])\gamma(c) = ((\varphi \otimes \Id_{A})\gamma \otimes \bar \alpha)\gamma(c).
\end{align*}
Slicing  the third tensor factor above and in \corr{\eqref{eq:aux-universal-crossed} and}{\eqref{eq:aux-universal-crossed}, and} using \eqref{eq:action-slices}, we conclude that $(\varphi,\rho)$ is a $\multunit[A]$\nb-commutative representation if and only if for all $c\in C$,
\begin{align*}
  \dumaxcorep[A]_{\rho 2}(\varphi(c) \otimes \Id_{A}) (\dumaxcorep[A]_{\rho 2})^{*} &= (\varphi \otimes \Id_{A})\gamma(c). \qedhere
\end{align*}
\end{proof}

\subsection{The quasi-triangular case}
\label{subsec:triangular}
Suppose that $\G$ is quasi-triangular in the following sense.
\begin{definition}[\cite{Meyer-Roy-Woronowicz:Twisted_tensor_2}*{Definition 3.1}] \label{def:quasi-triangular}
A $\Cst$\nb-quantum group $\G=(A,\Comult[A])$ is \emph{quasi-triangular} if it comes with a fixed bicharacter $\Rmat \in  \U(\hat {A} \otimes \hat {A})$, called its \emph{$R$-matrix}, satisfying
\begin{align}
  \label{eq:quasi-triangular}
  \Rmat (\sigma \circ \hat \Delta_{A}(\hat a)) \Rmat^{*} = \hat\Delta_{A}(\hat a) \quad\text{for all } \hat a \in \hat A.
\end{align}
\end{definition}
A short calculation shows that  \eqref{eq:quasi-triangular} is equivalent to the relation
\begin{align} \label{eq:dual-r-matrix}
  \Rmat_{23}\widehat{W}^{A}_{13}\widehat{W}^{A}_{12} = \widehat{W}^{A}_{12}\widehat{W}^{A}_{13}\Rmat_{23} 
  \qquad\text{in~$\U(A\otimes\hat {A}\otimes\hat {A})$},
\end{align}
which in turn is equivalent to  $(\Id_{A},\Id_{A})$ being an $(\Rmat,\Rmat)$\nb-pair, that is, a Drinfeld pair for $\Rmat$.

Suppose that  $ \Rmat\in \U(\hat {A} \otimes \hat {A})$ is an $R$-matrix. 
\begin{proposition} \label{prop:quasi-triangular}
  Let $(C,\gamma_{C})$ and $(D,\gamma_{D})$ be $\G$\nb-$\Cst$\nb-algebras. Then there exists a unique continuous coaction $\gamma_{C\boxtimes D}$ of $\G$ on $C \boxtimes^{\Rmat}_{\maxtensor} D$ that makes the canonical morphisms $j_{C}^{u}$ and $j_{D}^{u}$ from  $C$ and $D$ to $C \boxtimes^{\Rmat}_{\maxtensor} D$ equivariant.
\end{proposition}
\begin{proof}
 Lemma  \ref{lemma:pairs-vertical-action}, applied to $(j^{\univ}_C,j^{\univ}_D)$ and the $(\Rmat,\Rmat)$\nb-pair $(\Id_{A},\Id_{A})$, shows that 
$\varphi:= (j^{\univ}_C \otimes \Id_{A}) \circ \gamma_{C}$
and $\psi:= (j^{\univ}_D \otimes \Id_{A}) \circ \gamma_{D}$ form an $\Rmat$\nb-commutative representation. The induced morphism
\begin{align*}
  \gamma_{(C\boxtimes D)}&:=  \varphi \chiprod \psi \in \Mor( C\boxtimes^{\Rmat}_{\maxtensor} D, (C\boxtimes^{\Rmat}_{\maxtensor} D ) \otimes A)
\end{align*}
is easily seen to be a coaction and to satisfy the Podle\'s condition.
\end{proof}
Denote by $\tau_{\G}$ the trivial coaction of $\G$ on $\C$.
\begin{theorem}  \label{theorem:quasi-triangular}
Let $\G$ be a quasi-triangular $\Cst$-quantum group with $R$-matrix $\Rmat$. Then
 the assignment
\begin{align*}
 ((C,\gamma_{C}),(D,\gamma_{D})) \mapsto (C\boxtimes^{\Rmat}_{\max} D,\corr{$\gamma_{C\boxtimes^{R}_{\max} D}$}{\gamma_{C\boxtimes D}})
\end{align*}
 extends to a bifunctor $\Cstcat(\G)\times \Cstcat(\G) \to \Cstcat(\G)$ which endows $\Cstcat(\G)$ with the structure of a monoidal category. Its unit is $(\C,\tau_{\G})$. 
\end{theorem}
\begin{proof}
Clearly, the assignment extends to a bifunctor. To show that $\Cstcat(\G)$ becomes monoidal,
it suffices to prove the following two assertions:
  \begin{enumerate}
    \item  For any $\G$\nb-$\Cst$\nb-algebra $(C,\gamma)$, the canonical morphisms  to $(C,\gamma) \boxtimes^{\Rmat}_{\maxtensor} (\C,\tau_{\G})$ and $(\C,\tau_{\G}) \boxtimes^{\Rmat}_{\maxtensor} (C,\gamma)$ are isomorphisms.
    \item For any $\G$\nb-$\Cst$\nb-algebras $(C,\gamma_{C})$, $(D,\gamma_{D})$, $(E,\gamma_{E})$, there exists a unique isomorphism of $\G$\nb-$\Cst$\nb-algebras
      \begin{align*}
        (C \boxtimes^{\Rmat}_{\maxtensor} D) \boxtimes^{\Rmat}_{\maxtensor} E \to C \boxtimes^{\Rmat}_{\maxtensor} (D\boxtimes^{\Rmat}_{\maxtensor} E)
      \end{align*}
that intertwines the canonical maps of $C,D$ and $E$ to these $\Cst$\nb-algebras.
  \end{enumerate}
Both follow easily from Yoneda-type arguments. For example, to prove (2), it suffices to note that for every $\Cst$\nb-algebra $F$ with morphisms $\pi_{C},\pi_{D},\pi_{E}$ from $C,D,E$, respectively, to $F$, the following conditions are equivalent:
  \begin{itemize}
  \item $(\pi_{C},\pi_{D})$ and $(\pi_{C} \chiprod \pi_{D},\pi_{E})$ are $\Rmat$\nb-commutative representations;
  \item $(\pi_{C},\pi_{D})$,  $(\pi_{C},\pi_{E})$ and $(\pi_{D},\pi_{E})$ are  $\Rmat$\nb-commutative representations;
  \item $(\pi_{D},\pi_{E})$ and $(\pi_{C},\pi_{D} \chiprod \pi_{E})$ are $\Rmat$\nb-commutative representations. \qedhere
  \end{itemize}
\end{proof}

We can also define the notion of braided commutativity for $\G$\nb-$\Cst$\nb-algebras.
In the  von-Neumann algebraic setting, the corresponding notion was introduced in \cite{Enock-Timmermann:yd}*{Definition 2.5.3}.
\begin{definition}
  Let $\G$ be a quasi-triangular $\Cst$\nb-quantum group.
  A $\G$\nb-$\Cst$\nb-algebra $(C,\gamma)$ is \emph{braided-commutative} if $ (\Id_{C},\Id_{C})$ is an $\Rmat$\nb-commutative representation or, equivalently, if there exists a morphism
  \begin{align*}
   C \boxtimes^{\Rmat}_{\maxtensor} C \to C, \quad
   j^{\univ}_C(c)j^{\univ}_D(c') \mapsto cc', 
  \end{align*}
  where $j^{\univ}_C$ and $j^{\univ}_D$ denote the two  canonical morphisms from $C$ to $C \boxtimes^{\Rmat}_{\maxtensor} C$.
\end{definition}

\section{An isomorphism of two crossed products}
\label{sec:iso}

Let $\Qgrp{G}{A}$ \corr{be $\Qgrp{H}{B}$}{and $\Qgrp{H}{B}$ be} $\Cst$-quantum
groups with a bicharacter $\chi\in\U(\hat {A}\otimes\hat {B})$, and let
$(C,\gamma)$ be a $\G$-$\Cst$-algebra and $(D,\delta)$ an $\G[H]$-$\Cst$-algebra
as before.

Then the maximal tensor product $C \otimes_{\maxtensor} D$ carries a natural coaction of the product $\Cst$-quantum group
$\G\times\G[H]:=(\hat {A}\otimes\hat {B},\flip_{23}(\DuComult[A]\otimes\DuComult[B]))$ and we can form the crossed product
\begin{align}
  \label{eq:crossed-untwisted}
  (C \otimes_{\maxtensor} D) \rtimes (\hat  A\otimes \hat  B).
\end{align}
The maximal twisted tensor product $C\boxtimes^{\chi}_{\maxtensor} D$ can informally be regarded as a deformation of $C\otimes_{\maxtensor} D$ with respect to $\chi$.  Likewise, there exists a deformation of $\G\times \G[H]$ with respect to $\chi$,  the generalised
 Drinfeld double~$\GenDrinfdouble{\G}{\G[H]}{\bichar}=\Bialg{\DrinAlg_{\bichar}}$ 
 associated to the bicharacter $\chi$ in \cite{Roy:Codoubles}. We show that like the 
 minimal twisted tensor product $C \boxtimes^{\chi}_{\maxtensor} D$, see  \cite{Roy:Codoubles}*{Theorem 6.3}, also the maximal twisted tensor product carries a natural coaction of $\GenDrinfdouble{\G}{\G[H]}{\bichar}$. Moreover, we show that the associated crossed product
 \begin{align} \label{eq:crossed-twisted}
  (C \boxtimes_{\maxtensor} D) \rtimes \CodoubAlg_{\chi}
\end{align}
is naturally isomorphic to the crossed product \eqref{eq:crossed-untwisted}. 

Recall that the  $\Cst$\nb-quantum group  $\GenDrinfdouble{\G}{\G[H]}{\bichar}$ 
 comes with two morphisms
 $\rho \colon A \to \DrinAlg_{\bichar}$ and $
  \theta \colon B \to \DrinAlg_{\bichar}$
of $\Cst$\nb-bialgebras such that $(\rho,\theta)$ form a Drinfeld pair for $\chi$ and $\rho(A) \cdot \theta(B) = \DrinAlg_{\bichar}$ \cite{Roy:Codoubles}.
\begin{proposition} \label{prop:drinfeld-action}
There exists a unique coaction of $\GenDrinfdouble{\G}{\G[H]}{\bichar}$ on $C\boxtimes^{\chi}_{\max} D$ that makes the following diagram commute, and 
this coaction is continuous:
\begin{align*}
  \xymatrix{
    C \ar[d]_{\gamma} \ar[r]^{j^{\univ}_C} & C\boxtimes^{\chi}_{\max} D \ar[d] &  D\ar[l]_{j^{\univ}_D}  \ar[d]^{\delta} \\
    C \otimes A \ar[r]^(0.3){j^{\univ}_C \otimes \rho} & (C \boxtimes^{\chi}_{\max} D) \otimes 
\DrinAlg_{\bichar} & D \otimes B \ar[l]_(0.3){j^{\univ}_D \otimes \theta}
  }
\end{align*}
\end{proposition}
\begin{proof}
 Lemma \ref{lemma:pairs-vertical-action}, applied to the $\chi$\nb-commutative representation  $(j^{\univ}_C,j^{\univ}_D)$ and the $(\chi,\chi)$\nb-pair $(\rho,\theta)$, yields the desired morphism $(j^{\univ}_C \otimes \rho) \gamma \chiprod (j^{\univ}_D \otimes \theta)\delta$. A routine computation shows that this  morphism is a coaction and satisfies the Podle\'s condition.
\end{proof}
We thus find:
 \begin{theorem} \label{the:gen_drinf_coact}
The maximal twisted tensor product  is a bifunctor
\begin{align*}
\boxtimes_{\maxtensor}^{\chi}\colon \Cstcat(\G)\times\Cstcat(\G[H])\to 
 \Cstcat(\GenDrinfdouble{\G}{\G[H]}{\bichar}).  
\end{align*}
\end{theorem}   

Let us now turn to the crossed products \eqref{eq:crossed-untwisted} and \eqref{eq:crossed-twisted}. First, we recall their definition.

Denote by $\iota^{\univ}_{C},\iota^{\univ}_{D}$ and \corr{$j^{u}_{C},j^{u}_{D}$}{$j^{\univ}_{C},j^{\univ}_{D}$} the canonical morphisms from $C$ and $D$  to $C\maxotimes D$ and to $C\boxtimes_{\maxtensor}^{\chi} D$, respectively.

Choose faithful Heisenberg pairs $(\pi,\hat \pi)$ and $(\eta,\hat \eta)$  for $\G$ and $\G[H]$ on Hilbert spaces $\Hils$ and $\Hils[K]$, respectively. Then   $(\pi \otimes \eta,\hat \pi\otimes \hat \eta)$ is a Heisenberg pair for $\G\times \G[H]$, and the reduced crossed product $(C\maxotimes D) \rtimes (\hat  A\otimes \hat  B)$ can be identified with the $\Cst$\nb-subalgebra of~$\Mult((C\maxotimes D) \otimes \Comp(\Hils)\otimes\Comp(\Hils[K]))$
 generated by all elements of the form
\begin{align} \label{eq:elements-crossed-untwisted}
  \dot c &:= (\iota^{\univ}_{C} \otimes \pi)(\gamma(c))_{12}, & \dot d &:= (\iota^{\univ}_{D} \otimes \eta)(\delta(d))_{13}, & \dot \omega &:= (\hat \pi \otimes \hat  \eta)(\omega)_{23},
\end{align}
where $c\in C$, $d\in D$ and $\omega\in \hat  A \otimes \hat  B$.

Following \cite{Roy:Codoubles}, we next define a $\chi$\nb-Heisenberg pair $(\alpha,\beta)$ on $\Hils[K] \otimes \Hils$  by
\begin{align*}
    \alpha(a) &= 1\otimes \pi(a), & \beta(b) &= (\eta \otimes \hat \pi)\hat \Delta_{R}(b),
\end{align*}
see \cite{Roy:Codoubles}*{Proposition 2.35}. Denote by $(\bar \alpha,\bar \beta)$ the associated $\chi$\nb-anti-Heisenberg pair and define, as in \cite{Roy:Codoubles}*{Proposition 3.10},  representations
$\rho, \theta,\xi,\zeta$ of  $A,B,\hat  A,\hat  B$, respectively, on $\conj{\Hils[K]} \otimes \conj{\Hils} \otimes \Hils[K] \otimes \Hils$  by
\begin{align*}
  \rho(a) &= (\bar \alpha \otimes \alpha)\Comult[A](a), &
\theta(b) &= (\bar \beta \otimes \beta)\Comult[B](b), \\
  \xi(\hat  a)&= 1\otimes 1\otimes  1 \otimes\hat \pi(\hat  a), &
  \zeta(\hat  b) &= 1\otimes 1 \otimes \hat \eta(\hat  b)\otimes 1.
\end{align*}
Then the reduced crossed product $(C\boxtimes_{\maxtensor}^{\chi} D) \rtimes \DrinAlg_{\chi}$ can be identified with the $\Cst$\nb-subalgebra of 
$\Mult((C\boxtimes^{\bichar}_{\maxtensor} D) \otimes \Comp(\conj{\Hils[K]}) \otimes \Comp(\conj{\Hils}) \otimes \Comp(\Hils[K]) \otimes \Comp(\Hils))$
generated by all elements of the form
\begin{align*}
  \ddot c&= (\corr{$j^u_C$}{j^{\univ}_C} \otimes \rho)\gamma(c), & \ddot d &= (\corr{\qquad $j^u_D$}{j^{\univ}_D} \otimes \theta)\delta(d), &
  \ddot \omega &= (\hat \eta \otimes \hat \pi)(\sigma(\omega))_{45},
\end{align*}
where $c\in C$, $d\in D$ and $\omega \in \hat  A \otimes \hat  B$. Moreover, the $\Cst$-quantum group $\GenDrinfdouble{\G}{\G[H]}{\bichar}=\Bialg{\DrinAlg_{\bichar}}$  arises from the modular multiplicative unitary
\begin{align*}
   \Multunit^{\mathcal{D}_{\chi}} = W^{A}_{\xi \rho} W^{B}_{\zeta \theta} \in \U(\overline{\Hils[K]} \otimes \overline{\Hils[H]} \otimes \Hils[K] \otimes \Hils[H]),
\end{align*}
see \cite{Roy:Codoubles}*{Theorem 4.1}.
\begin{lemma} 
 \label{lemm:const_phi}
  There exists a non-degenerate $*$-homomorphism
\begin{align*}
\Phi \colon (C \maxotimes D) \rtimes (\hat  A \otimes \hat  B) \to (C \boxtimes^{\chi}_{\maxtensor} D) \rtimes \CodoubAlg_{\chi}
\end{align*}
 such that for all $c\in C$,
 $d\in D$ and $\omega \in \hat  A\otimes \hat  B$,
    \begin{align*}
\Phi(\dot c) &=  \ddot{c},  & \Phi(\dot d) &= \ddot{\chi} \ddot{d} \ddot{\chi}^{*},
& \Phi(\dot \omega) &= \ddot{\omega}.
    \end{align*}
\end{lemma}
\begin{proof}
Since $(\corr{$j^{u}_{C},j^{u}_{D}$}{j^{\univ}_{C},j^{\univ}_{D}})$ is a  $\chi$-commutative representation, the morphisms $\varphi_{C}$ and $\varphi_{D}$ from $C$ and $D$, respectively, to $(C\boxtimes^{\bichar}_{\maxtensor} D) \otimes \Comp(\conj{\Hils[K]}\otimes\conj{\Hils})$ given by
\begin{align*}
  \varphi_{C}(c) &:=(\corr{$j^u_C$}{j^{\univ}_C} \otimes \bar \alpha)\gamma(c) &&\text{and} & \varphi_{D}(d)&:=(\corr{\qquad $j^u_D$}{j^{\univ}_D} \otimes \bar \beta)\delta(d)
\end{align*}
commute and induce a morphism $\varphi$ from $C\maxotimes D$ to $(C\boxtimes_{\maxtensor}^{\chi} D) \otimes \Comp(\conj{\Hils[K]}\otimes\conj{\Hils})$.  Since the representations $(\eta\otimes \pi)\circ \sigma$ and $(\hat \eta \otimes \hat \pi)\circ \sigma$
form a  Heisenberg pair for $\G\times \G[H]$, we obtain a morphism
\begin{align*}
 \Phi \in \Mor((C \maxotimes D)\rtimes(\hat  A\otimes \hat  B),   (C\boxtimes_{\maxtensor}^{\chi} D) \otimes \Comp(\conj{\Hils[K]}) \otimes \Comp(\conj{\Hils})\otimes \Comp(\Hils[K])\otimes\Comp(\Hils))
\end{align*}
satisfying
\begin{align*}
  \Phi(\dot c) &= (\varphi_{C} \otimes \pi)\gamma(c)_{1235} = (\corr{$\iota^u_C$}{j^{\univ}_C} \otimes (\bar  \alpha \otimes \pi)\Comult[A])\gamma(c)_{1235} = \ddot c, \\
  \Phi(\dot d) &= (\varphi_{D} \otimes \eta)\delta(d)_{1234} = (\corr{$\iota^u_D$}{j^{\univ}_D} \otimes (\bar  \beta \otimes \eta)\Comult[B])\delta(d)_{1234}, \\
  \Phi(\dot \omega) &= (\hat \eta \otimes \hat \pi)\sigma(\omega)_{45} = \ddot \omega
\end{align*}
for all $c\in C$, $d\in D$ and $\omega\in \hat  A\otimes \hat  B$. But  by   Lemma~\ref{lemma:delta-implementation} and definition of $\beta$,
\begin{align} \label{eq:beta}
(\hat  \eta \otimes \hat \pi)(\hat  \chi)  (\eta(b) \otimes 1)  (\hat \eta\otimes \hat \pi)(\hat \chi^{*}) &=
 (\eta \otimes \hat \pi)(\hat \Delta_{R}(b)) = \beta(b),
\end{align}
and hence
\begin{align*}
  \ddot\chi^{*}\Phi(\dot d)\ddot\chi &=
(\corr{$\iota^u_C$}{j^{\univ}_C} \otimes (\bar  \beta \otimes \beta)\Comult[B])\delta(d)_{12345} 
 =(\corr{$\iota^u_C$}{j^{\univ}_C} \otimes \theta)\delta(d)_{12345}  =
  \ddot d. \qedhere
\end{align*}
\end{proof}
To show that $\Phi$ is an isomorphism, we shall construct its inverse and use the following result.
\begin{lemma}
 \label{lemm:inv_phi}
  There exists a representation $\lambda$ of $A\otimes B$ on $\Hils[K] \otimes \Hils \otimes \overline{\Hils[K]} \otimes \overline{\Hils}$ such that for all $a \in A$ and $b\in B$,
  \begin{align*}
    \lambda( a\otimes b)= (\alpha \otimes \bar \alpha)\Comult[A](a)(\beta\otimes \bar \beta)\Comult[B](b).
  \end{align*}
\end{lemma}
\begin{proof}
Denote by $(\bar \pi,\bar {\hat \pi})$ and $(\bar \eta,\bar {\hat \eta})$ the  anti-Heisenberg pairs   associated to $(\pi,\hat \pi)$ and $(\eta,\hat \eta)$, respectively, and write
\begin{align*}
  \pi^{(2)}&:=(\pi \otimes \bar \pi)\Comult[A], & \hat \pi^{(2)}&:=(\hat \pi \otimes \bar {\hat \pi})\DuComult[A], & \eta^{(2)}&:= (\eta\otimes \bar \eta)\Comult[B].
\end{align*}
Then there exists a representation $\kappa$ of $A\otimes B$ such that
\begin{align*}
  \kappa(a\otimes b) = \pi^{(2)}(a)_{23}\eta^{(2)}(b)_{14}.
\end{align*}
Let $U:=  (\hat \eta \otimes \hat \pi^{(2)})(\hat  \chi)_{123}$. Since $\pi^{(2)}(A)$ and $\hat \pi^{(2)}(\hat  A)$ commute   \cite{Meyer-Roy-Woronowicz:Twisted_tensor}*{Proposition 3.15},  
\begin{align*}
  U \kappa(a\otimes 1) U^{*} &= \pi^{(2)}(a)_{23} = (\alpha \otimes \bar \pi)(\Comult[A](a))_{123}.
\end{align*}
   On the other hand, Lemma \ref{lemma:delta-implementation} implies that
   \begin{align*}
  U\kappa(1\otimes b) U^{*}=
 (\eta \otimes \hat \pi^{(2)} \otimes \bar  \eta)(\hat \Delta_{R} \otimes \Id_{B})\Comult[B](b).
   \end{align*}
Here, \eqref{eq:right_homomorphism} and the relation $(\hat{\Delta}_{R} \otimes \Id_{B})\Delta_{B}=(\Id_{B} \otimes \hat{\Delta}_{L})\Delta_{B}$   \cite{Meyer-Roy-Woronowicz:Homomorphisms}*{Lemma 5.7} imply
\begin{align*}
  (\Id_{B} \otimes \hat{\Delta}_{A} \otimes \Id_{B})(\hat{\Delta}_{R} \otimes \Id_{B})\Delta_{B} =
(\hat{\Delta}_{R} \otimes \Id_{A} \otimes \Id_{B})(\hat{\Delta}_{R} \otimes \Id_{B})\Delta_{B} = (\hat{\Delta}_{R} \otimes \hat{\Delta}_{L})\Delta_{B},
\end{align*}
whence
 \begin{align*}
U\kappa(1\otimes b) U^{*}
 &= (\eta\otimes \hat{\pi} \otimes \bar{\hat{\pi}} \otimes \bar \eta)(\hat  \Delta_{R} \otimes \hat \Delta_{L})\Comult[B](b) 
= (\beta \otimes \bar{\hat{\pi}} \otimes \bar \eta)(\Id_{B} \otimes \hat \Delta_{L})\Comult[B](b).
\end{align*}
 Flipping the third and fourth tensor factor, we obtain the desired representation $\lambda$ because $  \bar \alpha(a) = 1\otimes \bar \pi(a)$  and by \eqref{eq:antipode-morphisms}, $\hat  \Delta_{R} \circ R_{B} = \sigma (R_{\hat  A} \otimes R_{B}) \hat  \Delta_{L}$ and hence
$\bar \beta(b) = (\bar \eta \otimes \bar {\hat \pi})\sigma\hat \Delta_{L}(b)$.
\end{proof}

\begin{theorem} \label{theorem:crossed}
  The reduced crossed products $ (C\maxotimes D) \rtimes (\hat  A \otimes \hat  B)$ and 
  $(C\boxtimes_{\maxtensor}^{\chi} D) \rtimes \DrinAlg_{\chi}$ are isomorphic.
\end{theorem}
\begin{proof}
  We construct an inverse to $\Phi$ as follows. The  morphisms $\psi_{C}$ and $\psi_{D}$
 from $C$ and $D$ to $(C\maxotimes D) \otimes \Comp(\Hils[K]\otimes\Hils)$ given by
 \begin{align*}
   \psi_{C} &:=(\iota^{\univ}_{C} \otimes \alpha)\gamma \quad \text{and}\quad
   \psi_{D} := (\iota^{\univ}_{D} \otimes \beta)\delta
 \end{align*}
form a $\chi$\nb-commutative representation  by  Lemma \ref{lemma:pairs-vertical-action} and induce a morphism 
$ \psi=\psi_{C} \chiprod \psi_{D}$ from   $C\boxtimes^{\chi}_{\max} D$  to $(C\maxotimes D) \otimes \Comp(\Hils[K] \otimes \Hils)$. 
This, in turn, yields a morphism $\Psi$ from $(C \boxtimes^{\chi}_{\max} D) \rtimes \CodoubAlg_{\chi}$ to 
\begin{align} \label{eq:target-psi}
 (C\maxotimes D) \otimes  \Comp(\Hils[K]) \otimes  \Comp( \Hils) \otimes \Comp(\overline{\Hils[K]}) \otimes  \Comp(\overline{\Hils}) \otimes \Comp(\Hils[K]) \otimes \Comp(\Hils)
\end{align}
such that for all $c\in C$, $d\in D$,  and $\omega \in \hat  A \otimes \hat  B$,
\begin{align*}
  \Psi(\ddot c) &= (\psi_{C} \otimes \rho)\gamma(c) = (\iota^{\univ}_{C} \otimes  (\alpha \otimes \bar  \alpha \otimes \alpha)\Comult[A]^{(2)})\gamma(c), \\ 
  \Psi(\ddot d) &= (\psi_{D} \otimes \theta)\delta(d) =(\iota^{\univ}_{D} \otimes (\beta \otimes \bar  \beta \otimes \beta)\Delta^{(2)}_{B})\delta(d), \\
  \Psi(\ddot \omega) &= (\hat  \eta \otimes \hat \pi)\sigma(\omega)_{67},
\end{align*}
where~$\Comult^{(2)}= (\Id\otimes\Comult)\Comult$. By the preceding Lemma~\ref{lemm:inv_phi}, we can define representations  $\kappa$ and $\hat \kappa$ of  $A\otimes B$ and $\hat  A \otimes \hat  B$ on $\Hils[K] \otimes \Hils \otimes \overline{\Hils[K]} \otimes \overline{\Hils} \otimes \Hils[K] \otimes \Hils$ by the formulas
\begin{align*}
  \kappa(a\otimes b) &:= (\alpha \otimes \bar  \alpha \otimes \pi)\Comult[A]^{(2)}(a)_{12346}(\beta\otimes \bar \beta \otimes\eta)\Comult[B]^{(2)}(b)_{12345},  \\
  \hat \kappa(\hat  a\otimes \hat  b) &:= (\hat  \eta \otimes \hat \pi)( \hat{b} \otimes \hat{a})_{56},
\end{align*}
and $(\kappa,\hat \kappa)$ forms a faithful Heisenberg pair for $\G\times \G[H]$. We therefore obtain an embedding $\Xi$ of $(C \maxotimes D)  \rtimes (\hat  A \otimes \hat  B)$ into the $\Cst$-algebra \eqref{eq:target-psi}
such that
\begin{align*}
  \Xi(\dot c) &=  (\iota^{\univ}_{C} \otimes  (\alpha \otimes \bar  \alpha \otimes \pi)\Comult[A]^{(2)})\gamma(c)_{123457} = \Psi(\ddot c), \\
  \Xi(\dot d) &= (\iota^{\univ}_{D} \otimes (\beta \otimes \bar \beta \otimes \eta)\Delta^{(2)}_{B})\delta(d)_{123456}, \\
  \Xi(\dot \omega) &= \Psi(\ddot \omega) =  (\hat  \eta \otimes \hat \pi)\sigma(\omega)_{67}.
\end{align*}
Now, we conclude from \eqref{eq:beta} that
$\Psi(\ddot d) = \Xi(\dot \chi^{*}) \Xi(\dot d)\Xi(\dot \chi)$. Evidently, the composition $\Xi^{-1} \circ \Psi$   is inverse to $\Phi$.
\end{proof}
The arguments above can be adapted to the minimal twisted tensor product:
\begin{theorem} \label{theorem:crossed-reduced}
  There exists an isomorphism  \corr{$\Phi_{r}$}{$\Phi_{\red}$} that makes the following diagram  commute,
where the vertical maps are the canonical quotient maps:  \begin{align*}
    \xymatrix@R=15pt{
(C\maxotimes D) \rtimes (\hat  A \otimes \hat  B) \ar[r]^{\Phi} \ar[d] & (C\boxtimes_{\maxtensor}^{\chi} D) \rtimes \CodoubAlg_{\chi}  \ar[d] \\
(C\otimes D) \rtimes (\hat  A \otimes \hat  B) \ar[r]^{\corr{$\Phi_r$}{\Phi_{\red}}} & (C\boxtimes_{\mintensor}^{\chi} D) \rtimes \CodoubAlg_{\chi}.
    }
  \end{align*}
  \end{theorem}
  \begin{proof}
    A straightforward modification of the proof above yields  embeddings $\Psi_{\mintensor}$ and $\Xi_{\mintensor}$ from
    $(C\boxtimes^{\chi}_{\mintensor} D) \rtimes \CodoubAlg_{\chi}$ and $ (C\otimes D) \rtimes (\hat  A \otimes \hat  B)$, respectively, into  
    \begin{align*}
  \Mult((C \otimes D) \otimes\Comp(\Hils[K] \otimes \Hils \otimes \conj{\Hils[K]} \otimes\conj{\Hils} \otimes \Hils[K] \otimes \Hils))
    \end{align*}
    such that, denoting by $\dot c,\dot d,\dot \omega$ and $\ddot c,\ddot d, \ddot \omega$ the canonical images of $c\in C$, $d\in D$ and $\omega \in \hat  A\otimes \hat  B$ in  $(C\boxtimes_{\mintensor}^{\chi} D) \rtimes \CodoubAlg_{\chi}$ and $ (C\otimes D) \rtimes (\hat  A \otimes \hat  B)$, respectively, 
    \begin{align*}
      \Psi_{\mintensor}(\ddot c) &= \Xi_{\mintensor}(\dot c), & \Psi_{\mintensor}(\ddot d)&= \Xi_{\mintensor}(\dot \chi^{*} \dot d\dot \chi), & \Psi_{\mintensor}(\ddot \omega) &= \Xi_{\mintensor}(\dot \omega).  \qedhere
    \end{align*}
  \end{proof}
As before, denote by
 $\dot \chi$, $\dot C=\{\dot c: c\in C\}$ and $\dot D=\{\dot d :d\in D\}$  the natural images of $\chi$, $C$ and $D$, respectively, in the crossed product  $(C\maxotimes D) \rtimes (\hat  A \otimes \hat  B)$.
\begin{corollary} \label{corollary:crossed}
  Suppose that  the coaction of $\GenDrinfdouble{\G}{\G[H]}{\chi}$ on $C\boxtimes^{\chi}_{\maxtensor} D$ is injective. Then:
  \begin{enumerate}
  \item  $\Phi$ maps
  $C\boxtimes_{\max}^{\chi} D$ isomorphically to $[\dot \chi \dot C\dot \chi^{*} \cdot  \dot D] \subseteq (C\maxotimes D) \rtimes (\hat  A \otimes \hat  B)$.
\item   If $C$ is nuclear, then the canonical map $C\boxtimes_{\maxtensor}^{\chi} D \to C\boxtimes_{\mintensor}^{\chi} D$ is an isomorphism.
  \end{enumerate}
\end{corollary}
\begin{proof}
Assertion  (1)  follows immediately from Theorem \ref{theorem:crossed}. Suppose that $C$ is nuclear. Then 
we can identify $(C\maxotimes D) \rtimes (\hat  A \otimes \hat  B) \cong (C\otimes D) \rtimes (\hat  A \otimes \hat  B)$ using the quotient map, and
 $\Phi$ and \corr{$\Phi_{r}$}{$\Phi_{\red}$} map   $C\boxtimes_{\max}^{\chi} D$ and $C \boxtimes_{\min}^{\chi} D$, respectively, isomorphically to the same $\Cst$-subalgebra $[\dot \chi \dot C\dot \chi^{*} \cdot  \dot D]$.
\end{proof}
\begin{remark} \label{remark:functoriality}
If the coaction of $\GenDrinfdouble{\G}{\G[H]}{\chi}$ on $C\boxtimes^{\bichar}_{\maxtensor} D$ is injective, one can use the isomorphism $\varphi$ of Corollary \ref{corollary:crossed} (1) and functoriality  of the maximal tensor product and the reduced crossed product to construct a twisted maximal tensor product $f\boxtimes g$ for
equivariant \Star{}homomorphisms/completely positive maps/completely positive contractions $f$ and $g$ on $C$ and $D$, respectively.
\end{remark}

Let us end this section with an application of Corollary \ref{corollary:crossed} to the case  where  $\G$ and $\G[H]$ are duals of locally compact abelian groups $G$ and $H$, respectively. In that case, the minimal twisted tensor product $C\boxtimes^{\chi}_{\min} D$ can be regarded as a Rieffel deformation of the minimal tensor product $C\otimes D$ as  defined in  \cite{Kasprzak:Rieffel_deformation}.  This result carries over to the universal setting easily as follows.

Let $G$ and $H$ be locally compact abelian groups with Pontrjagin duals $\hat  G$ and $\hat  H$, respectively. Let $C$ be a $\hat  G$\nb-$\Cst$\nb-algebra, $D$ an $\hat  H$\nb-$\Cst$\nb-algebra and  $\chi \in\Contb(G \times H,\T)$  \corr{be a}{a}  bicharacter. Then the maximal tensor product $C \maxotimes D$  carries the product action of $\Gamma :=  \hat  G\times  \hat  H$,  the formula
\begin{align} \label{eq:psi}
  \Psi((g,h),(g',h')) := \chi(g,h')
\end{align}
defines a bicharacter $\Psi$ on $\Gamma$, which we can regard as a 2-cocycle, and as in \cite{Kasprzak:Rieffel_deformation}, we can form a Rieffel deformation of  $C\maxotimes D$ with respect to $\Psi$ in the form of a $\Cst$\nb-subalgebra 
\begin{align*}
  (C\maxotimes D)^{\Psi} \subseteq \Mult((C \maxotimes D) \rtimes \Gamma).  
\end{align*}
The following explicit description of this Rieffel deformation was obtained already in \cite{Meyer-Roy-Woronowicz:Twisted_tensor}*{Theorem 6.2}, but we include the proof for convenience of the reader.   Note that the bicharacter $\Psi$ above is denoted by $\Psi'$ in \cite{Meyer-Roy-Woronowicz:Twisted_tensor}*{Theorem 6.2}, but the difference is inessential.
For elements of $\Mult((C \maxotimes D) \rtimes \Gamma)$,  we use the notation \eqref{eq:elements-crossed-untwisted} as before.
\begin{lemma} \label{lemma:rieffel}
  $(C\maxotimes D)^{\Psi}=    [\dot \chi \dot C\chi^{*}  \dot D]$ as $\Cst$\nb-subalgebras of  $\Mult((C \maxotimes D) \rtimes \Gamma)$.
\end{lemma}
\begin{proof} 
We follow \cite{Meyer-Roy-Woronowicz:Twisted_tensor}*{proof of Theorem 6.3} and only switch $\Psi$ and $\Psi'$.

The Rieffel deformation  $(C\maxotimes D)^{\Psi}$ is defined as a $\Cst$\nb-subalgebra
of the crossed product $(C\maxotimes D) \rtimes \Gamma$ by means of the unitaries
\begin{align*}
  U_{g,h} &\in \Contb(G\times H,\T), \quad U_{g,h}(g',h')  =\Psi((g',h'),(g,h)) = \chi(g',h),
\end{align*}
 see \cite{Kasprzak:Rieffel_deformation}. Since $C \maxotimes D = [(C\maxotimes 1)(1\maxotimes D)]$, \cite{Kasprzak:Rieffel_coaction}*{Lemma 3.4} implies that
 \begin{align}
   (C\maxotimes D)^{\Psi} = [(C\maxotimes 1)^{\Psi} (1\maxotimes D)^{\Psi}]. \label{eq:rieffel-aux-1}
 \end{align}
Since the unitaries  $U_{g,h}$ lie  in the subalgebra
$\Contb(G,\T)\otimes 1$ and $\hat {G}$ acts trivially on $D$, the images of $U_{g,h}$ in $(C \maxotimes D) \rtimes \Gamma$ commute with  $\dot D$. Therefore,  
\begin{align} \label{eq:rieffel-aux-2}
 (1\maxotimes D)^{\Psi}=  \dot D \subseteq \Mult((C \maxotimes D) \rtimes \Gamma).
\end{align}

The 2-cocycle $\Psi$ is cohomologous to the 2-cocycle $\Psi'$ defined by
\begin{align*}
  \Psi'((g,h),(g',h')):=\chi^{-1}(g',h).
\end{align*}
Indeed,
\begin{align*}
  (\partial \chi)((g,h),(g',h')) &:= \frac{\chi(gg',hh')}{\chi(g,h)\chi(g',h')} = \chi(g,h')\chi(g',h)
\end{align*}
and hence $\Psi= (\partial \chi)\Psi'$. By \cite{Kasprzak:Rieffel_deformation}*{Lemmas 3.4 and 3.5}, we get
\begin{align} \label{eq:rieffel-aux-3}
  (C \maxotimes 1)^{\Psi} = \chi   (C \maxotimes 1)^{\Psi'} \chi^{*}
\end{align}
in $\Mult((C\maxotimes D)\rtimes \Gamma)$. Now, a similar argument as above shows that $(C\maxotimes 1)^{\Psi'} = \dot C \subseteq \Mult((C\maxotimes D) \rtimes \Gamma)$.
Combining formulas \eqref{eq:rieffel-aux-1}--\eqref{eq:rieffel-aux-3}, the assertion follows.
\end{proof}

\begin{theorem}
  Let $C$ be a $\hat G$-$\Cst$-algebra and $D$ a $\hat H$-$\Cst$-algebra. Then
there exists an isomorphism $C\boxtimes_{\max}^{\chi} D \to (C\maxotimes D)^{\Psi}$ that intertwines the canonical embeddings of $C$ and $D$.
\end{theorem}
\begin{proof}
Combine the preceding result and Corollary \ref{corollary:crossed} (1). Note that here, the coaction of $\DrinAlg_{\chi}$
is injective because it just corresponds to an action of $G\times H$.
\end{proof}

\section{Passage to coactions  of universal quantum groups}
\label{sec:universal}

The results that we would like to present next involve the push-forward of coactions along morphisms of $\Cst$-quantum groups. Such a push-forward, however, can only be defined under  additional assumptions on the coaction, like injectivity, see \cite{Meyer-Roy-Woronowicz:Homomorphisms} and the Appendix,  which we  are unable to verify in the cases of interest to us. 

We therefore switch to coactions of universal quantum groups, which subsume injective, continuous coactions of $\Cst$-quantum groups and  where the push-forward is straightforward.

 Indeed, 
let $\Qgrp{G}{A}$ be a $\Cst$-quantum group and let $(C,\gamma)$ be a (continuous) coaction of the universal $\Cst$-bialgebra $(A^{\univ},\Delta^{\univ}_{A})$. If $(D,\Delta_{D})$ is another $\Cst$-bialgebra and $f\in \Mor(A^{\univ},D)$ is a morphism of $\Cst$-bialgebras, then
\begin{align*}
  f_{*}\gamma:=(\Id_{C} \otimes f)\gamma\in \Mor( C, C\otimes 
    D)
\end{align*}
is a  (continuous) coaction again. 
In the case where $(D,\Delta_{D})=(A,\Delta_{A})$ and $f=\Lambda_{A}$, we write
\begin{align*}
  \gamma^{\red}:=(\Lambda_{A})_{*}\gamma = (\Id_{C} \otimes \Lambda_{A})\gamma \in \Mor(C,C\otimes A),
\end{align*}
and then the assignment  $(C,\gamma) \mapsto (C,\gamma^{\red})$ identifies the  \emph{normal} and continuous coactions of $(A^{\univ},\Delta^{\univ}_{A})$ with the injective and continuous coactions of $(A,\Delta_{A})$.


The construction of the maximal twisted tensor product lifts to coactions of universal $\Cst$-quantum groups as follows.

Suppose that $\Qgrp{G}{A}$ and $\Qgrp{H}{B}$ are $\Cst$-quantum groups with a morphism of $\Cst$-bialgebras $f\in \Mor(A^{\univ},\hat B^{\univ})$, 
 and denote by $\chi=W^{f} \in \U(\hat A \otimes \hat B)$ the corresponding bicharacter.
 Let $(C,\gamma)$ be a $\UG$-$\Cst$-algebra and $(D,\delta)$  a $\UG[H]$-$\Cst$-algebra.  
 \begin{lemma} \label{lemma:f-pairs}
   Let $E$ be a $\Cst$-algebra with morphisms $\varphi \in \Mor(C,E)$  and $\psi \in \Mor(D,E)$, and suppose that $(\bar\alpha,\bar\beta)$ is an anti-Heisenberg pair for $f$ on some Hilbert space $\Hils$. Then the following conditions are equivalent:
   \begin{enumerate}
   \item $(\phi \otimes \bar\alpha)\gamma(c)$ and $(\psi \otimes \bar\beta)\delta(d)$ commute for all $c\in C$, $d\in D$.
   \item $(\varphi,\psi)$ is a $\chi$-commutative representation of $(C,\gamma^{\red})$ and $(D,\delta^{\red})$.
   \end{enumerate}
 \end{lemma}
   Note that if $f$ is reduced, then the assertion follows immediately from Proposition \ref{proposition:heisenberg-reduced}.
 \begin{proof}
   Since the coactions $\gamma^{\red}$ and $\delta^{\red}$ are (strongly) continuous, they are also weakly continuous. Hence, (1) is equivalent to the commutation of the elements
   \begin{align*}
     ((\varphi\otimes \bar\alpha)\gamma \otimes \Id_{A})\gamma^{\red}(c)_{123}
     =(\varphi \otimes (\bar\alpha \otimes \Lambda_{A})\Delta^{\univ}_{A})\gamma(c)_{123}
   \end{align*}
and
\begin{align*}
     ((\psi \otimes \bar\beta) \delta \otimes \Id_{B})\delta^{\red}(d)_{124} =
     (\psi \otimes (\bar\beta \otimes \Lambda_{B})\Delta^{\univ}_{B})\delta(d)_{124}
\end{align*}
in $\Mult(E\otimes \Comp(\Hils) \otimes A \otimes B)$ for all $c\in C$ and $d\in D$. By Corollary \ref{cor:pairs-reduced-ideal}, 
\begin{align*}
  ((\bar\alpha \otimes \Lambda_{A})\Delta^{\univ}_{A}, (\bar\beta\otimes\Lambda_{B})\Delta^{\univ}_{B}) = (\bar\alpha,\bar\beta) \otimes (\Lambda_{A},\Lambda_{B})
\end{align*}
is a reduced anti-Heisenberg pair for $f$ and hence, by Lemma \ref{lemma:pairs-reduced-universal}, of the form $(\bar\alpha' \Lambda_{A},\bar\beta'\Lambda_{B})$ for some anti-Heisenberg pair $(\bar\alpha',\bar\beta')$ for $\chi$. Now, (1) is equivalent to commutation of
$(\varphi\otimes \bar\alpha')\gamma^{\red}(c)$ and $(\psi\otimes \bar\beta')\delta^{\red}(d)$ for all $c\in C$ and $d\in D$, which is (2).
 \end{proof}
Thanks to this result, we can quickly  define an \emph{$f$-commutative representation} of $(C,\gamma)$ and $(D,\delta)$ to be a $\chi$-commutative representation of $(C,\gamma^{\red})$ and $(D,\delta^{\red})$,  and  the maximal twisted tensor product of $(C,\gamma)$ and $(D,\delta)$ with respect to $f$ to be the $\Cst$-algebra
\begin{align*}
  (C,\gamma)\boxtimes^{f}_{\max} (D,\delta):= (C,\gamma^{\red})\boxtimes^{\chi}_{\max} (D,\delta^{\red}).
\end{align*}

The construction of the maximal twisted tensor product is functorial with respect to the $\Cst$-quantum groups involved in the following sense.  Denote
by $\hat f \in \Mor(B^{\univ},\hat A^{\univ})$ the dual morphism of $f$; see Theorem \ref{the:equivalent_notion_of_homomorphisms}.
\begin{lemma} \label{lemma:pairs-action-push-pull}
   Let $E$ be a $\Cst$-algebra with morphisms $\varphi \in \Mor(C,E)$  and $\psi \in \Mor(D,E)$.
  Then the following conditions are equivalent:
  \begin{enumerate}
  \item  $(\varphi,\psi)$ is an $f$\nb-commutative representation of $(C,\gamma)$ and $(D,\delta)$;
  \item $(\varphi,\psi)$ is an $\Id_{\hat B^{\univ}}$-commutative representation of $(C,f_{*}\gamma)$ and $(D,\delta)$;
  \item $(\varphi,\psi)$ is an $\Id_{A^{\univ}}$\nb-commutative representation of $(C,\gamma)$ and $(D,\hat f_{*}\delta)$.
  \end{enumerate}
\end{lemma}
\begin{proof}
We only prove equivalence of (1) and (2); equivalence of (1) and (3) follows similarly.  Choose an anti-Heisenberg pair $(\bar\alpha,\bar\beta)$ for $\Id_{\hat B^{\univ}}$. Then by Lemma \ref{lemma:f-pairs}, (2) holds if and only if $(\varphi\otimes \bar\alpha f)\gamma(c)$ commutes with $(\psi\otimes\bar\beta)\delta(d)$ for all $c\in C$ and $d\in D$. But by Lemma \ref{lemma:pairs-push-pull},   $(\bar\alpha f,\bar \beta)$ is an anti-Heisenberg pair  for $f$, and so,  by Lemma \ref{lemma:f-pairs} again, this commutation relation is equivalent with (1).
\end{proof}
We obtain the following immediate consequence:
\begin{theorem} \label{theorem:qgp-functoriality}
 There exist canonical isomorphisms
  \begin{align*}
(C,f_{*}\gamma) \boxtimes^{\Id}_{\maxtensor} (D,\gamma) \cong    (C,\gamma) \boxtimes^{f}_{\maxtensor} (D,\delta) \cong (C,\gamma) \boxtimes^{\Id}_{\maxtensor} (D,\hat f_{*}\delta)
  \end{align*}
which intertwine the canonical morphisms from $C$ and $D$, respectively, to the three $\Cst$-algebras above.
\end{theorem}
\begin{corollary}
Let $\G_{i}=(A_{i},\Delta_{i})$ be a $\Cst$-quantum group for $i=1,2,3,4$, let $f^{(i)} \in \Mor(A^{\univ}_{i},A^{\univ}_{i+1})$ be morphisms of $\Cst$-bialgebras for $i=1,2,3$, and let $(C,\gamma)$ be a $\G_{1}$-$\Cst$-algebra and $(D,\delta)$ a $\G_{4}$-$\Cst$-algebra. Write $f:=f^{(3)} \circ f^{(2)} \circ f^{(1)}$. Then there exists a canonical isomorphism
\begin{align*}
  (C,\gamma) \boxtimes_{\max}^{f} (D,\delta) \cong (C, f^{(1)}_{*}\gamma) \boxtimes^{f^{(2)}}_{\max} (D,\hat f^{(3)}_{*}\delta).
\end{align*}
\end{corollary}
\begin{proof}
Use the sequence of isomorphisms
\begin{align*}
  (C,\gamma) \boxtimes^{f}_{\max} (D,\delta) &\cong (C,f^{(3)}_{*}f^{(2)}_{*}f^{(1)}_{*}\gamma) \boxtimes^{\Id}_{\max} (D,\delta)  \\ & \cong (C,f^{(2)}_{*}f^{(1)}_{*}\gamma) \boxtimes^{\Id}_{\max} (D,\hat f^{(3)}_{*}\delta) \cong
  (C,f^{(1)}_{*}\gamma) \boxtimes^{f^{(2)}}_{\max} (D,\hat f^{(3)}_{*}\delta). \qedhere
\end{align*}
\end{proof}
By Proposition \ref{prop:drinfeld-action}, the maximal twisted tensor product $C\boxtimes_{\maxtensor}^{\chi} D$ carries a canonical coaction of the generalised Drinfeld double $\GenDrinfdouble{\G}{\G[H]}{\bichar}=(\DrinAlg_{\chi},\Delta_{\DrinAlg_{\chi}})$.  We show that this coaction lifts to the universal level.  Recall that the $\Cst$-algebra $\DrinAlg_{\chi}$ is generated by the images of two morphisms $\rho \in \Mor(A, \DrinAlg_{\bichar})$ and $ \theta \in \Mor(B, \DrinAlg_{\bichar})$ of $\Cst$\nb-bialgebras which form a Drinfeld pair for $\chi$ \cite{Roy:Codoubles}.  By \cite{Meyer-Roy-Woronowicz:Homomorphisms}*{Section 4}, the compositions $\rho\circ \Lambda_{A}$ and $\theta \circ \Lambda_{B}$ lift uniquely to morphisms $\rho^{\univ} \in \Mor(A^{\univ},\DrinAlg_{\bichar}^{\univ})$ and $\theta^{\univ} \in \Mor(B^{\univ},\DrinAlg_{\bichar}^{\univ})$ of $\Cst$-bialgebras.
\begin{lemma}\label{lemma:lifted-drinfeld}
 $(\rho^{\univ},\sigma^{\univ})$ is a Drinfeld pair for $f$.  
\end{lemma}
\begin{proof}
  Denote by $\maxcorep[A]$ and $\maxcorep[B]$ the maximal corepresentations.
By Lemma \ref{lemma:pairs-universal-equivalence}, it suffices to show that the products \begin{align} \label{eq:lifted-drinfeld}
  \maxcorep[A]_{1\rho^{\univ}}\maxcorep[B]_{2\theta^{\univ}} \quad \text{and} \quad \chi_{12}^{*}\maxcorep[B]_{2\theta^{\univ}}\maxcorep[A]_{1\rho^{\univ}}\chi_{12}
\end{align}
 in $\U(\hat A \otimes \hat  B \otimes \DrinAlg_{\bichar}^{\univ})$ coincide.   Since $\rho^{\univ}$ and $\theta^{\univ}$ are morphisms of $\Cst$-bialgebras, both products
 are right corepresentations.  We apply the reducing morphism to $ \DrinAlg_{\bichar}^{\univ}$ and obtain the right corepresentations
$W^{A}_{1\rho}W^{B}_{2\theta}$ and $\chi_{12}^{*}W^{B}_{2\theta}W^{A}_{1\rho}\chi_{12}$, respectively,
which coincide because $(\rho,\theta)$ is a $(\chi,\chi)$-pair.  By \cite{Meyer-Roy-Woronowicz:Homomorphisms}*{Lemma 4.13}, the products \eqref{eq:lifted-drinfeld} have to  coincide.
\end{proof}
Now, the proofs of Proposition \ref{prop:drinfeld-action} and \ref{the:gen_drinf_coact} carry over to the universal setting and we obtain the following results:
\begin{proposition} \label{proposition:lifted-double-action}
Let $(C,\gamma)$ be a $\UG$-$\Cst$-algebra and $(D,\delta)$ be a $\UG[H]$-$\Cst$-algebra. Then   there exists a unique coaction of $\GenDrinfdoubleU{\G}{\G[H]}{\bichar}$ on $C\boxtimes^{f}_{\max} D$ that makes the  following diagram commute, and 
this coaction is continuous:
\begin{align*}
  \xymatrix@C=40pt{
    C \ar[d]_{\gamma} \ar[r]^{j^{\univ}_C} & C\boxtimes^{f}_{\max} D \ar[d] &  D\ar[l]_{j^{\univ}_D}  \ar[d]^{\delta} \\
    C \otimes A^{\univ} \ar[r]^(0.35){j^{\univ}_C \otimes \rho^{\univ}} & (C \boxtimes^{f}_{\max} D) \otimes 
\DrinAlg_{\bichar}^{\univ} & D \otimes B^{\univ} \ar[l]_(0.35){j^{\univ}_D \otimes \theta^{\univ}}
  }
\end{align*}
\end{proposition}
 \begin{theorem} \label{the:lifted_drinf_coact}
The maximal twisted tensor product  extends to a bifunctor
\begin{align*}
\boxtimes_{\maxtensor}^{f}\colon \Cstcat(\UG)\times\Cstcat(\UG[H])\to 
 \Cstcat(\GenDrinfdoubleU{\G}{\G[H]}{\bichar}).  
\end{align*}
\end{theorem}   

Let us finally consider the case where $\G$ is quasi-triangular with $R$-matrix $\Rmat \in \U(\hat A \otimes \hat A)$. 
\begin{theorem} \label{thm:triangular-universal}
  Let $\G$ be a quasi-triangular $\Cst$-quantum group with $R$-matrix $\Rmat$ and denote by $f = f_{\Rmat}\in \Mor(A^{\univ},\hat A^{\univ})$ the corresponding morphism of $\Cst$-bialgebras. 
  \begin{enumerate}
  \item   Let $(C,\gamma_{C})$ and $(D,\gamma_{D})$ be $\UG$\nb-$\Cst$\nb-algebras. Then there exists a unique continuous coaction $\gamma_{C\boxtimes D}$ of $\UG$ on $C \boxtimes^{f}_{\maxtensor} D$ that makes the canonical morphisms $\corr{$j^u_C$}{j^{\univ}_C}$ and $\corr{ $j^u_D$}{j^{\univ}_D}$ from  $C$ and $D$ to $C \boxtimes^{f}_{\maxtensor} D$ equivariant.  
  \item 
 The assignment
$((C,\gamma_{C}),(D,\gamma_{D})) \mapsto (C\boxtimes^{f} D,\gamma_{C\boxtimes D})$
 extends to a bifunctor
 \begin{align*}
  \Cstcat(\UG)\times \Cstcat(\UG) \to \Cstcat(\UG) 
 \end{align*}
 which endows $\Cstcat(\UG)$ with the structure of a monoidal category. Its unit is $(\C,\tau_{\G})$. 
  \end{enumerate}
\end{theorem}
\begin{proof}
  As observed after \eqref{eq:dual-r-matrix}, $(\Id_{A},\Id_{A})$ is a Drinfeld pair for $\Rmat$. A similar argument like the one used in the proof of Lemma \ref{lemma:lifted-drinfeld} shows that $(\Id_{A^{\univ}},\Id_{A^{\univ}})$ is an $(f,f)$-pair. Now, (1) and (2) follow by similar argument as in the proofs of Proposition \ref{proposition:lifted-double-action} and of Theorem \ref{theorem:quasi-triangular}.
\end{proof}

\section{Yetter--Drinfeld C*-algebras}
 \label{sec:yd} 

 For every quasi-triangular $\Cst$-quantum group $\G$,  the maximal twisted tensor product endows the  category of $\G$-$\Cst$-algebras with a monoidal structure, as we saw in Subsection \ref{subsec:triangular} and Theorem \ref{thm:triangular-universal}. More generally, we now consider  Yetter-Drinfeld $\Cst$-algebras and their maximal twisted tensor products, and thus  obtain not a monoidal category but a bicategory.

In the reduced setting,  Yetter-Drinfeld $\Cst$-algebras were introduced in \cite{Nest-Voigt:Poincare} and generalized in \cite{Roy:Codoubles}. We need to work in the universal setting,  because the following  constructions  will involve the push-forward of coactions  along morphisms of $\Cst$-quantum groups in situations where we do not know whether this is well-defined in the reduced setting.

Let $\G=(A,\Comult[A])$ and $\G[H]=(B,\Comult[B])$ be $\Cst$\nb-quantum groups and let $f\in \Mor(A^{\univ},B^{\univ})$ be a morphism of $\Cst$-bialgebras with corresponding bicharacter
$\unibich^{f}=  (\Id_{\hat A^{\univ}} \otimes f)(\unibich^{A})$.

The definition of Yetter-Drinfeld $\Cst$-algebras over $f$ involves the  twisted flip map
 \begin{align*}
     \sigma^{\univ}_{f} &\colon B^{\univ} \otimes \hat A^{\univ} \to
     \hat A^{\univ} \otimes B^{\univ}, \quad
     b^{\univ} \otimes \hat a^{\univ}\mapsto \unibich^{f}(\hat a^{\univ} \otimes b^{\univ})(\unibich^{f})^{*}.
 \end{align*}
Since $\unibich^{\hat{f}}=
\Sigma (\unibich^{f})^{*} \Sigma$, we have
\begin{align} \label{eq:twisted-flip-inverse}
  (\sigma^{\univ}_{f})^{-1} = \sigma^{\univ}_{\hat f}.
\end{align}
Moreover,  \eqref{eq:corep-left} implies  the following cocycle relation:
\begin{align}
  \label{eq:twisted-flip-cocycle}
  (\Id_{\hat A^{\univ}} \otimes \sigma^{\univ}_{f})(\sigma^{\univ}_{f} \otimes \Id_{\hat A^{\univ}})(\Id_{B^{\univ}} \otimes \hat\Delta^{\univ}_{A}) =(\hat\Delta^{\univ}_{A} \otimes \Id_{B^{\univ}}) \sigma^{\univ}_{f}.
\end{align}

Now, we can give the following the universal counterpart to \cite{Roy:Codoubles}*{Definition 7.2}.
 \begin{definition}
  An \emph{$f$-Yetter-Drinfeld $\Cst$-algebra} is a $\Cst$-algebra with continuous coactions $\gamma$ of $\UDG$ and $\delta$ of $\UG[H]$ satisfying
   \begin{align} \label{eq:yd-universal}
     (\gamma \otimes \Id_{B}) \circ \delta =  (\Id_{C} \otimes \sigma^{\univ}_{f}) \circ (\delta \otimes \Id_{\hat {A}}) \circ \gamma.
   \end{align}
   A \emph{morphism} of $f$\nb-Yetter-Drinfeld $\Cst$\nb-algebras $(C,\gamma_{C},\delta_{C})$ and $(D,\gamma_{D},\delta_{D})$ is a morphism $\phi \in \Mor(C,D)$ that is equivariant with respect to the respective coactions of $\UDG$ and $\UG[H]$.
 \end{definition}
 Denote by $\YDcat(f)$ the category of $f$\nb-Yetter-Drinfeld $\Cst$\nb-algebras. 
 \begin{remark}
Denote by $\hat f\in \Mor(\hat B^{\univ}, \hat A^{\univ})$ the morphism dual to $f$. Then  \eqref{eq:twisted-flip-inverse}  implies that  the assignment $(C,\gamma,\delta) \mapsto (C,\delta,\gamma)$ defines an isomorphism 
\begin{align}
  \YDcat{(f)} \to \YDcat{(\hat f)}.
\end{align}
 \end{remark}
Denote by $\chi=W^{f}=(\hat \Lambda_{A} \otimes \Lambda_{B})(\unibich^{f})$  the reduced bicharacter corresponding to $f$. Then the reduced $\chi$-Yetter-Drinfeld $\Cst$-algebras defined in  \cite{Roy:Codoubles}*{Definition 7.2} form a full subcategory of  $\YDcat(f)$:
 \begin{proposition}
Let $C$ be a $\Cst$-algebra with normal continuous coactions $\gamma$ of $\UDG$ and $\delta$ of  $\UG[H]$.  Then the following assertions are equivalent:
\begin{enumerate}
\item $(C,\gamma,\delta)$ is an $f$-Yetter-Drinfeld $\Cst$-algebra;
\item  $(C,\gamma^{\red},\delta^{\red})$ is a
  $\chi$-Yetter-Drinfeld $\Cst$-algebra.
\end{enumerate}   
 \end{proposition}
\begin{proof}
We need to show that \corr{\eqref{eq:yd-universal} if}{\eqref{eq:yd-universal} holds if} and only if
   \begin{align} \label{eq:yd-reduced}
     (\gamma^{\red} \otimes \Id_{B}) \circ \delta^{\red} =  (\Id_{C} \otimes \sigma_{ \chi}) \circ (\delta^{\red} \otimes \Id_{\hat {A}}) \circ \gamma^{\red},
   \end{align}
 where  $\sigma_{\chi}(b\otimes \hat a )=\chi(\hat a \otimes b)\chi^{*}$. 
 Clearly, \eqref{eq:yd-reduced} follows from \eqref{eq:yd-universal} upon application of $\Id_{C} \otimes \hat \Lambda_{A} \otimes \Lambda_{B}$. Conversely, suppose \eqref{eq:yd-reduced}.
We first show that
  \begin{align} \label{eq:drinfeld-lift-1}
(\Id_{C} \otimes \tilde \sigma)    (\delta^{\red} \otimes \Id_{\hat A^{\univ}})\gamma = (\gamma \otimes \Id_{B})\delta^{\red},
  \end{align}
where  $\tilde{\sigma} \colon B \otimes \hat{A}^{\univ} \to \hat{A}^{\univ} \otimes B$ is given by $b\otimes \hat  a\mapsto \tilde \chi(\hat  a\otimes b)\tilde \chi^{*}$ with $\tilde\chi=(\Id_{\hat  A^{\univ}} \otimes   \Lambda_{B})(\unibich^{f})$. 
Denote by $\tilde \Delta_{A} \colon \hat  A \to \hat  A\otimes \hat  A^{\univ}$  the canonical coaction. Then
$(\Id_{C} \otimes \tilde \Delta_{A})\gamma^{\red}  = (\Id_{C} \otimes \gamma^{\red})\gamma$ and
\begin{align*}
  (\gamma^{\red}\otimes \Id_{B}\otimes \Id_{\hat A^{\univ}})(\delta^{\red}\otimes \Id_{\hat A^{\univ}})\gamma &= (\Id_{C} \otimes \sigma_{\chi} \otimes \Id_{\hat A^{\univ}})(\delta^{\red} \otimes \Id_{B}\otimes \Id_{\hat A^{\univ}})(\gamma^{\red}\otimes \Id_{\hat A^{\univ}})\gamma \\ 
&= (\Id_{C} \otimes \sigma_{\chi} \otimes \Id_{\hat A^{\univ}})(\delta^{\red} \otimes    \tilde \Delta_{A})\gamma^{\red}  \\
&= (\Id_{C} \otimes \sigma_{\chi} \otimes \Id_{\hat A^{\univ}})(\Id_{C} \otimes \Id_{B} \otimes  \tilde \Delta_{A})(\Id_{C} \otimes \sigma_{\chi}^{-1})(\gamma^{\red} \otimes \Id_{B})\delta^{\red}.
\end{align*}
Now, \eqref{eq:twisted-flip-cocycle} implies
\begin{align*}
  (\sigma_{\chi} \otimes \Id_{\hat A^{\univ}})(\Id_{B} \otimes \tilde \Delta_{A})\sigma_{\chi}^{-1} = (\Id_{\hat A} \otimes \tilde \sigma^{-1})(\tilde \Delta_{A} \otimes \Id_{B})
\end{align*}
as morphisms from $\hat A \otimes B$ to $\hat A \otimes B \otimes \hat A^{\univ}$, and hence
\begin{align*}
    (\gamma^{\red}\otimes \Id_{B}\otimes \Id_{\hat A^{\univ}})(\delta^{\red}\otimes \Id_{\hat A^{\univ}})\gamma &= 
 (\Id_{C} \otimes \Id_{\hat A} \otimes \tilde \sigma^{-1})(\Id_{C} \otimes \tilde \Delta_{A} \otimes \Id_{B})(\gamma^{\red} \otimes \Id_{B})\delta^{\red} \\
 &=
 (\gamma^{\red} \otimes \Id_{B} \otimes \Id_{\hat A^{\univ}})(\Id_{C} \otimes \tilde \sigma^{-1})(\gamma \otimes \Id_{B})\delta^{\red}.
\end{align*}
Since $\gamma$ is normal, $\gamma^{\red}$ is injective and \eqref{eq:drinfeld-lift-1}  follows. Now, denote by  $\tilde \Delta_{B}\colon B\to B\otimes B^{\univ}$ the canonical coaction. Then
\begin{align*}
  (\delta^{\red} \otimes \Id_{\hat A^{\univ}}\otimes \Id_{B^{\univ}})(\gamma \otimes \Id_{B^{\univ}})\delta &= 
  (\Id_{C} \otimes \tilde \sigma^{-1} \otimes \Id_{B^{\univ}})(\gamma \otimes \Id_{B} \otimes \Id_{B^{\univ}})(\delta^{\red} \otimes \Id_{B^{\univ}})\delta  \\
  &=  (\Id_{C} \otimes \tilde \sigma^{-1} \otimes \Id_{B^{\univ}})(\gamma \otimes  \tilde \Delta_{B})\delta^{\red} \\
  &= (\Id_{C} \otimes \tilde \sigma^{-1} \otimes \Id_{B^{\univ}})(\Id_{C} \otimes \Id_{\hat A^{\univ}} \otimes \tilde \Delta_{B})(\Id_{C} \otimes\tilde\sigma)(\delta^{\red}\otimes \Id_{\hat A^{\univ}})\gamma.
\end{align*}
Now,  \eqref{eq:twisted-flip-inverse} and \eqref{eq:twisted-flip-cocycle}, applied to $\hat f$ instead of $f$,  imply
\begin{align*}
  (\tilde \sigma^{-1} \otimes \Id_{B^{\univ}})(\Id_{\hat A^{\univ}} \otimes \tilde \Delta_{B})\tilde\sigma = 
(\Id \otimes \sigma_{f}^{-1})(\tilde \Delta_{B} \otimes \Id)
\end{align*}
 and hence
\begin{align*}
  (\delta^{\red} \otimes \Id_{\hat A^{\univ}}\otimes \Id_{B^{\univ}})(\gamma \otimes \Id_{B^{\univ}})\delta &=  (\Id_{C} \otimes \Id_{\hat A^{\univ}} \otimes \sigma_{f}^{-1})(\Id_{C} \otimes \tilde \Delta_{B} \otimes \Id_{\hat A^{\univ}})(\delta^{\red} \otimes \Id_{\hat A^{\univ}})\gamma \\
  &= (\Id_{C} \otimes \Id_{B} \otimes \sigma_{f}^{-1})(\delta^{\red} \otimes \Id_{B^{\univ}} \otimes \Id_{\hat A^{\univ}})(\delta \otimes \Id_{\hat A^{\univ}})\gamma.
\end{align*}
Since $\delta^{\red}$ is injective, we can conclude the desired relation \eqref{eq:yd-universal}.
\end{proof}

Suppose now that we have three $\Cst$-quantum groups $\Qgrp{G}{A}$, $\Qgrp{H}{B}$ and $\Qgrp{I}{C}$  with  morphisms $f\in \Mor(A^{\univ},B^{\univ})$ and $g\in \Mor(B^{\univ},C^{\univ})$ of $\Cst$-bialgebras.
\begin{lemma} \label{lemma:yd-push-pull}
  \begin{enumerate}
  \item Let $(D,\gamma_{D}, \delta_{D})$ be an
    $f$-Yetter-Drinfeld $\Cst$-algebra. Then the triple $(D,\gamma_{D},g_{*}\delta_{D})$ is a $(g\circ f)$-Yetter-Drinfeld $\Cst$-algebra. 
  \item Let $(E,\gamma_{E},\delta_{E})$ be a $g$-Yetter-Drinfeld $\Cst$-algebra. Then $(E,\hat f^{*}\gamma_{E},\delta_{E})$ is a $(g\circ f)$-Yetter-Drinfeld $\Cst$-algebra.
\end{enumerate}
\end{lemma}
\begin{proof}
We only prove (1); a similar argument applies to (2). By \eqref{eq:morphism-bicharacter}, the  universal bicharacters $\unibich^{f}$ and $\unibich^{g\circ f}$  corresponding to $f$ and $g\circ f$, respectively, are related by the equation
\begin{align*}
  (\Id_{\hat A^{\univ}} \otimes g)(\unibich^{f}) =
  (\Id_{\hat A^{\univ}} \otimes (g\circ f))(\unibich^{A}) =
 \unibich^{g\circ f},
\end{align*}
and hence
\begin{align}
    (\Id_{\hat A^{\univ}} \otimes g) \sigma^{\univ}_{f}(b \otimes
    \hat a) &= (\Id_{\hat A^{\univ}} \otimes g) (\unibich^{f}(\hat a
    \otimes b) (\unibich^{f})^{*}) =\sigma^{\univ}_{(g\circ f)}(g(b) \otimes \hat a)
\end{align}
for all $b\in B^{\univ}$ and $\hat a \in \hat A^{\univ}$. Now, we apply  $\Id_{D} \otimes \Id_{\hat A^{\univ}} \otimes g$ to  \eqref{eq:yd-universal}
and conclude that
\begin{align*}
  (\gamma_{D} \otimes \Id_{C^{\univ}})g_{*}\delta_{D} &= (\Id_{D} \otimes \sigma^{\univ}_{(g\circ f)})(g_{*}\delta_{D} \otimes \Id_{\hat A^{\univ}})\gamma_{D}. \qedhere
\end{align*}
\end{proof}
Clearly, the assignments   $(D,\gamma_{D}, \delta_{D}) \mapsto (D,\gamma_{D}, g_{*}\delta_{D})$ and
 $(E,\gamma_{E},\delta_{E}) \mapsto  (E,\hat f_{*}\gamma_{E},\delta_{E})$ extend to functors
\begin{align*}
 g_{*}& \colon \YDcat{(f)} \to \YDcat{(g\circ f)} &&\text{and} & f^{*}&\colon \YDcat{(g)}  \to \YDcat{(g\circ f)},
\end{align*}
respectively.
\begin{proposition} \label{proposition:yd-product}
 Let $(D,\gamma_{D},\delta_{D})$ be an $f$\nb-Yetter-Drinfeld $\Cst$\nb-algebra and $(E,\gamma_{E},\delta_{E})$  a $g$\nb-Yetter-Drinfeld $\Cst$\nb-algebra. Denote by $D\boxtimes^{\Id}_{\max} E$ the maximal twisted tensor product of $D$ and $E$ formed with respect to $\delta_{D},\gamma_{E}$ and the identity morphism on $B^{\univ}$.
Then there exist unique continuous right coactions $\gamma_{D\boxtimes E}$ and  $\delta_{D\boxtimes E}$ of $\UDG$  and $\G[I]^{\univ}$, respectively, on $D\boxtimes^{\Id}_{\max} E$ such that 
 \begin{align*}
   \gamma_{(D\boxtimes E)} \circ j^{\univ}_{D} &= (j^{\univ}_{D} \otimes \Id_{\hat A^{\univ}})\gamma_{D}, &
   \gamma_{(D\boxtimes E)} \circ j^{\univ}_{E} &= (j^{\univ}_{E} \otimes \hat f)\gamma_{E},   \\
   \delta_{(D\boxtimes E)} \circ j^{\univ}_{D} &= (j^{\univ}_{D} \otimes g)\delta_{D}, &
   \delta_{(D\boxtimes E)} \circ j^{\univ}_{E} &= (j^{\univ}_{E} \otimes \Id_{C^{\univ}})\delta_{E}.
 \end{align*}
Moreover, the triple  $(D \boxtimes^{\Id}_{\max} E,\gamma_{D\boxtimes E},\delta_{D\boxtimes E})$ is a $(g\circ f)$\nb-Yetter-Drinfeld $\Cst$\nb-algebra.
\end{proposition}
\begin{proof}
  Uniqueness is clear since $j^{\univ}_{D}(D) \cdot j^{\univ}_{E}(E) =  D\boxtimes^{\Id}_{\max} E$.

To prove existence of $\delta_{(D\boxtimes E)}$, we need to show that the representations
\begin{align*}
  \phi:=(j^{\univ}_{D} \otimes g)\delta_{D} \quad\text{and}\quad \psi:=(j^{\univ}_{E} \otimes \Id_{C^{\univ}})\delta_{E}
\end{align*}
form an $\Id_{B^{\univ}}$-commutative representation. By Lemma
\ref{lemma:f-pairs}, it suffices to choose an anti-Heisenberg pair
$(\bar\pi,\bar{\hat{\pi}})$  \corr{\tiny for the reduced bicharacter $W^{B}$, which
  corresponds to the}{for the} identity morphism on $B^{\univ}$, and  to show that
 the elements
\begin{align*}
  (\phi \otimes \bar\pi)\delta_{D}(d) = (j^{\univ}_{D} \otimes (g \otimes \bar\pi)\Delta^{\univ}_{D})\corr{$\delta_{d}$}{\delta_{D}}
\end{align*}
and
\begin{align*}
  (\psi \otimes \bar{\hat{\pi}})\gamma_{E}(e) = ((j^{\univ}_{E} \otimes \Id_{C^{\univ}})\delta_{E} \otimes \bar{\hat{\pi}})\gamma_{E}(e)
\end{align*}
commute for all $d\in D$ and $e\in E$.  We use \eqref{eq:delta-implementation-universal} and the Yetter-Drinfeld condition for $(E,\gamma_{E},\delta_{E})$ to rewrite these elements in the form
\begin{align*}
  (\phi \otimes \bar\pi)\delta_{D}(d) =  (\hat\unibich^{B}_{g\bar{\hat{\pi}}})_{23} (j^{\univ}_{D} \otimes  \bar\pi)\delta_{D}(d)_{13}   (\hat\unibich^{B}_{g\bar{\hat{\pi}}})^{*}_{23},
\end{align*}
and
\begin{align*}
  (\psi \otimes \bar{\hat{\pi}})\gamma_{E}(e) &= 
  (j^{\univ}_{E} \otimes (\Id_{C^{\univ}} \otimes \bar{\hat{\pi}})(\sigma^{\univ}_{g})^{-1})(\gamma_{E} \otimes \Id_{C^{\univ}})\delta_{E}(e)  \\
&= (\hat \unibich^{B}_{g\bar{\hat{\pi}}})_{23}
  (j^{\univ}_{E} \otimes (\Id_{C^{\univ}} \otimes \bar{\hat{\pi}})\sigma)(\gamma_{E} \otimes \Id_{C^{\univ}})\delta_{E}(e)  (\hat \unibich^{B}_{g\bar{\hat{\pi}}})^{*}_{23}.
\end{align*}
Now, $(j^{\univ}_{D} \otimes  \bar\pi)\delta_{D}(d)_{13}$ commutes with $  (j^{\univ}_{E} \otimes (\Id_{C^{\univ}} \otimes \bar{\hat{\pi}})\sigma)(\gamma_{E} \otimes \Id_{C^{\univ}})\delta_{E}(e)$ because 
$(j^{\univ}_{D} \otimes \bar\pi)\delta_{D}(d)$ commutes with $(j^{\univ}_{E} \otimes \bar{\hat\pi})\gamma_{E}(e')$ for all $e'\in E$. 

 The universal property of $D\boxtimes^{\Id}_{\max} E$  yields a  morphism
 \corr{$\delta_{(D\boxtimes E)}$}{$\delta_{D\boxtimes E}$} as desired, and it is easy to see that this morphism
 is a continuous coaction. Existence of  $\gamma_{D\boxtimes E}$  follows similarly.

Finally, the relation  $j^{\univ}_{D}(D) \cdot j^{\univ}_{E}(E) =  D\boxtimes^{\Id}_{\max} E$ and Lemma \ref{lemma:yd-push-pull} imply that the triple  $(D \boxtimes^{\Id}_{\max} E,\gamma_{D\boxtimes E},\delta_{D\boxtimes E})$ is a $(g\circ f)$\nb-Yetter-Drinfeld $\Cst$\nb-algebra. 
\end{proof}

We thus obtain  a bifunctor
\begin{align} \label{eq:yd-bifunctor}
  \YDcat(f) \times \YDcat(g) \to \YDcat(g\circ f)
\end{align}
which sends a pair of Yetter-Drinfeld $\Cst$\nb-algebras $((D,\gamma_{D},\delta_{D}),(E,\gamma_{E},\delta_{E}))$ to
\begin{align*}
(D,\gamma_{D},\delta_{D}) \boxtimes_{\maxtensor}^{\Id} (E,\gamma_{E},\delta_{E}) := (D \boxtimes^{\Id}_{\max} E,\gamma_{D\boxtimes E},\delta_{D\boxtimes E}),
\end{align*}
and a pair  $(\phi,\psi)$ of morphisms to \corr{$\psi\boxtimes \psi$}{$\phi\boxtimes \psi$}. 

Letting $\G,\G[H],\G[I]$ and $f,g$ vary, we obtain a bicategory: 
\begin{theorem}
  There exists a bicategory $\YDcat$, where
 the $0$-objects are $\Cst$\nb-quantum groups,
 the category of $1$\nb-morphisms between two  $\Cst$\nb-quantum groups
    $\G=(A,\Comult[A])$ and $\G[H]=(B,\Comult[B])$ is the disjoint
    union of the categories $\YDcat(f)$ for all morphisms $f \in
    \Mor(A^{\univ},B^{\univ})$ of $\Cst$-bialgebras,
and the horizontal composition is given by the bifunctors
    in \eqref{eq:yd-bifunctor}.
\end{theorem}
\begin{proof}
The main points to prove are existence of units and associativity of the horizontal composition.

  For every $\Cst$-quantum group $\G=(A,\Comult[A])$, the $\Cst$\nb-algebra $\C$, equipped with the trivial coactions $\tau_{\UDG}$ of $\UDG$ and $\tau_{\UG}$ of $\UG$, is an $\Id_{A^{\univ}}$\nb-Yetter-Drinfeld $\Cst$\nb-algebra, and this  is the identity of $\G$  in the sense that for every morphism $f$ as above and every $f$\nb-Yetter-Drinfeld $\Cst$\nb-algebra $(C,\gamma_{C},\delta_{C})$, one has natural isomorphisms
\begin{align*}
  (\C,\tau_{\UDG},\tau_{\UG}) \boxtimes_{\maxtensor}^{\Id_{A}} (C,\gamma_{C},\delta_{C}) \cong (C,\gamma_{C},\delta_{C}) \cong (C,\gamma_{C},\delta_{C}) \boxtimes_{\maxtensor}^{\Id_{A}} (\C,\tau_{\UDG[H]},\tau_{\UG[H]}).
\end{align*}

Let us \corr{prove is associativity}{prove associativity}. Suppose that $\G=(A_{i},\Delta_{i})$, where $i=1,\ldots,4$, are $\Cst$\nb-quantum
groups with morphisms $f^{(i)} \in \Mor(A_{i}^{\univ},A_{i+1}^{\univ})$ of $\Cst$-bialgebras
and $f^{(i)}$-Yetter-Drinfeld $\Cst$-algebras $(C_{i},\gamma_{i},\delta_{i})$
 for $i=1,\ldots,3$. Denote by $\Id_{i}$ the identity on $A^{\univ}_{i}$. We claim that there exists a unique isomorphism of $(f^{(3)}\circ f^{(2)} \circ f^{(1)})$\nb-Yetter-Drinfeld algebras
\begin{align}
  C_{1} \boxtimes_{\maxtensor}^{\Id_{2}} (C_{2}\boxtimes_{\maxtensor}^{\Id_{3}} C_{3}) 
  \to (C_{1}\boxtimes_{\maxtensor}^{\Id_{2}} C_{2}) \boxtimes_{\maxtensor}^{\Id_{3}} C_{3}
\end{align}
that intertwines the canonical morphisms from each $C_{i}$ to these $\Cst$\nb-algebras. This follows from a similar Yoneda-type argument as used in the proof of Theorem \ref{theorem:quasi-triangular}. Indeed, suppose that $F$ is a $\Cst$-algebra   with morphisms $\pi_{i} \colon C_{i} \to F$ for $i=1,\ldots,3$ such that
 \begin{enumerate}
 \item $(\pi_{1},\pi_{2})$  is an  $\Id_{2}$-commutative representation of $(C_{1},\delta_{1})$ and $(C_{2},\gamma_{2})$, and
 \item $(\pi_{1} \chiprod \pi_{2},\pi_{3})$ is an $\Id_{3}$-commutative representation of $(C_{1}\boxtimes_{\maxtensor}^{\Id_{2}} C_{2},\delta_{C_{1}\boxtimes C_{2}})$ and $(C_{3},\gamma_{3})$.
\end{enumerate}
Now, $C_{1}\boxtimes_{\maxtensor}^{\Id_{2}} C_{2}$ is generated by the images of the canonical morphisms from $C_{1}$ and $C_{2}$ , and the first morphism is equivariant with respect to $f^{(2)}_{*}\delta_{1}$ and $\delta_{C_{1}\boxtimes C_{2}}$, while the second one is equivariant with respect to $\delta_{2}$ and $\delta_{C_{1}\boxtimes C_{2}}$. Hence, (1)--(2) are equivalent to (1) together with the following two conditions:
\begin{itemize}
\item[(2a)]  $(\pi_{1},\pi_{3})$ is an $f_{2}$-commutative representation of $(C_{1},\delta_{1})$ and $(C_{3},\gamma_{3})$, and
\item[(2b)] $(\pi_{2},\pi_{3})$ is an $\Id_{3}$-commutative representation of $(C_{2},\delta_{2})$ and $(C_{3},\gamma_{3 })$.
\end{itemize}
But if (2b) holds, then a similar argument shows that (1) and (2a) are equivalent to
\begin{itemize}
\item[(3)] $(\pi_{1},\pi_{2} \chiprod \pi_{3})$ is an $\Id_{2}$-commutative representation of $(C_{1},\delta_{1})$ and $(C_{2}\boxtimes_{\maxtensor}^{\Id_{3}} C_{3},\gamma_{C_{2}\boxtimes C_{3}},\delta_{C_{2}\boxtimes C_{3}})$.
\end{itemize}
\corr{Thus, there we obtain a bijectiion}{Thus, we obtain a bijection} between the morphisms from 
$C_{1} \boxtimes_{\maxtensor}^{\Id_{2}} (C_{2}\boxtimes_{\maxtensor}^{\Id_{3}} C_{3})$ to $F$ and 
the morphisms from $(C_{1}\boxtimes_{\maxtensor}^{\Id_{2}} C_{2}) \boxtimes_{\maxtensor}^{\Id_{3}} C_{3}$ to $F$, and this bijection is compatible with the canonical morphisms from each $C_{i}$ to the two domains.
\end{proof}



\appendix

\section{Normal coactions of universal $\Cst$-bialgebras}

 This appendix summarises the relation between coactions of a $\Cst$-quantum group and coactions  of its universal counterpart. It does not contain any new results but is included for convenience of the reader because we did not find a good reference besides \cite{Fischer:Quantum_group_and_equivariant_kk-theory}.

Let $\Qgrp{G}{A}$ be a $\Cst$-quantum group.
\begin{definition}
  A coaction $\gamma$ of $(A^{\univ},\Delta^{\univ}_{A})$ on a $\Cst$-algebra $C$ is \emph{normal} if the morphism
  \begin{align}
   \gamma^{\red}:=(\Id_{C} \otimes \Lambda_{A}) \circ \gamma \colon C \to C\otimes A
  \end{align}
is injective. Denote by $\Cstcat^{\normal}(\UG)$ the full subcategory of $\Cstcat(\UG)$ of all normal and continuous coactions.
\end{definition}

$\G$-$\Cst$-algebras with injective underlying coaction can be identified with normal $\UG$-$\Cst$-algebras as follows. 
The assignment $(C,\gamma) \mapsto (C,\gamma^{\red})$ evidently defines a functor 
\begin{align}
  \label{eq:1}
  \Cstcat^{\normal}(\UG) \to \Cstcat^{\mathrm{i}}(\G),
\end{align}
where $\Cstcat^{\mathrm{i}}(\G)$ denotes the full subcategory of $\Cstcat(\G)$ formed by all {injective} coactions,
and this functor is an isomorphism \cite{Fischer:Quantum_group_and_equivariant_kk-theory}. To describe the inverse, we use the coaction \corr{$\Delta^{\red,\univ}$}{$\Delta^{\red,\univ}_A$} of $\UG$ on $A$ obtained in Proposition \ref{prop:reduced-universal-coaction}.
Clearly, $\Lambda_{A}$ is an equivariant  morphism  from $(A^{\univ},\Delta^{\univ}_{A})$ to  $(A,\corr{$\Delta^{\red,\univ}$}{\Delta^{\red,\univ}_A})$.

\begin{theorem}
  Let $(C,\gamma)$ be a $\G$-$\Cst$-algebra. Suppose that $\gamma$ is injective. Then there exists a unique  coaction  $\gamma^{\univ}$ of $(A^{\univ},\Delta_{A}^{\univ})$ on $C$ such that the following diagram commutes,
  \begin{align*}
    \xymatrix@C=40pt@R=20pt{
      C \ar[r]^{\gamma} \ar[d]_{\gamma^{\univ}} & C\otimes A \ar[d]^{\Id_{C} \otimes \Delta^{\red,\univ}_{A}}  \\
      C \otimes A^{\univ}\ar[r]^{\gamma \otimes \Id_{A^{\univ}}} & C \otimes A \otimes A^{\univ},
    }
  \end{align*}
and $(C,\gamma^{\univ})$ is a normal $\UG$-$\Cst$-algebra.
The assignment $(C,\gamma) \mapsto (C,\gamma^{\univ})$ extends to a functor $\Cstcat^{\mathrm{i}}(\G) \to \Cstcat^{\normal}(\UG)$ which  is inverse to  the functor given by $(C,\gamma) \mapsto (C,\gamma^{\red})$.
\end{theorem}
\begin{proof}
  Essentially, but not literally, this is contained in
  \cite{Fischer:Quantum_group_and_equivariant_kk-theory}*{Section 3.3}. To get
  existence of $\gamma^{\univ}$ and  that  $(C,\gamma^{\univ})$ is a normal
  $\UG$-$\Cst$-algebra, one can simply copy the proof of
  \cite{Meyer-Roy-Woronowicz:Homomorphisms}*{Theorem 6.1}, replacing $\Delta_{R}$
  and $\Delta_{L}$ with \corr{$\Delta^{\red,\univ}$}{$\Delta^{\red,\univ}_A$} and
  \corr{\\ $\Delta^{\univ,\red}$}{$\Delta^{\univ,\red}_A$}, respectively. The relation
  \begin{align*}
   (\gamma \otimes \Id_{A})(\Id_{C} \otimes  \Lambda_{A})\gamma^{\univ} = (\Id_{C} \otimes \Delta_{A})\gamma = (\gamma \otimes \Id_{A})\gamma
  \end{align*}
and injectivity of $\gamma$ imply $(\Id_{C} \otimes \Lambda_{A})\gamma^{\univ}=\gamma$. Finally, if $\gamma=(\Id_{C} \otimes \Lambda_{A})\gamma'$ for some normal coaction $\gamma'$ of $(A^{\univ},\Delta^{\univ}_{A})$ on $C$, then
\begin{align*}
  (\gamma \otimes \Id_{A^{\univ}})\gamma' &= (\Id_{C} \otimes \Lambda_{A} \otimes \Id_{A^{\univ}})(\gamma' \otimes \Id_{A^{\univ}})\gamma'  \\ &= (\Id_{C} \otimes \Delta^{\red,\univ}_{A} \circ \Lambda_{A})\gamma' =(\Id_{C} \otimes \Delta^{\red,\univ}_{A})\gamma
\end{align*}
and hence $\gamma^{\univ}=\gamma'$.
\end{proof}

\section{Push-forward of  weakly continuous coactions along morphisms of $\Cst$-quantum groups}

In this section we consider the push-forward of coactions along morphisms of $\Cst$-quantum groups, but not for \emph{injective} coactions as in \cite{Meyer-Roy-Woronowicz:Homomorphisms}, but for \emph{weakly continuous} ones.

Let $\Qgrp{G}{A}$  and $\Qgrp{H}{B}$ be  $\Cst$-quantum groups with a morphism from $\G$ \corr{to a $\DuG[H]$}{to $\DuG[H]$}  in the form of a bicharacter $\chi \in \U(\hat  A \otimes \hat  B)$, and let $(C,\gamma)$ be \corr{an coaction}{a coaction} of $\G$.

If $\gamma$ is \emph{injective},  it was shown in  \cite{Meyer-Roy-Woronowicz:Homomorphisms}*{Theorem 69}
that there exists a unique injective continuous coaction $\chi_{*}\gamma$ of $\DuG[H]$ on $C$ 
 that makes the following diagram commute,
\begin{align} 
\label{eq:pushed-action-deltar}
  \xymatrix@C=40pt@R=20pt{
    C \ar[r]^{\gamma} \ar[d]_{\chi_{*}\gamma} & C\otimes A \ar[d]^{\Id_{C} \otimes \Delta_{R}} \\
    C \otimes \hat  B \ar[r]^(0.4){\gamma \otimes \Id_{\hat  B}} & C \otimes A \otimes \hat  B,
  }
\end{align}
where  $\Delta_{R} \colon A \to A\otimes \hat {B}$  denotes the
 right quantum group \corr{homomorphisms}{homomorphism} associated to $\chi$. 
\begin{definition}
  Let $(C,\gamma)$ be a weakly  continuous coaction of $\G$. We say that \emph{$\chi_{*}\gamma$ exists} if there exists a morphism
$\chi_{*}\gamma \in \Mor(C,C\otimes \hat  B)$
  that makes the following  diagram commute:
\begin{align} 
 \label{eq:pushed-action-deltal}
  \xymatrix@C=40pt@R=20pt{
    C \ar[r]^{\gamma} \ar[d]_{\gamma} & C\otimes A \ar[d]^{\Id_{C} \otimes \Delta_{L}} \\
    C \otimes A \ar[r]^(0.4){{\chi_{*}\gamma} \otimes \Id_{A}} & C \otimes \hat  B\otimes A  .
  }
\end{align}

\end{definition}
\begin{example}
  \begin{enumerate}
  \item If $\gamma$ is injective, then  the 
action $\chi_{*}\gamma$ 
 defined in   \cite{Meyer-Roy-Woronowicz:Homomorphisms}  makes diagram \eqref{eq:pushed-action-deltal} commute; see the proof of   \cite{Meyer-Roy-Woronowicz:Homomorphisms}*{Theorem 69}.
  \item If $\gamma$ lifts to \corr{an coaction}{a coaction} $\gamma^{\univ}$ of $(A^{\univ},\Comult[A]^{\univ})$ such that $\gamma=(\Id_{C} \otimes \Lambda_{A})\gamma^{\univ}$, then $\chi_{*}\gamma$ exists and is equal to  $(\Id_{C} \otimes \hat \Lambda_{B} f)\gamma^{\univ}$, where $\hat \Lambda_{B} \colon \hat  B^{\univ} \to \hat  B $ denotes the reducing morphism and $f\colon A^{\univ} \to \hat  B^{\univ}$ denotes the morphism of $\Cst$\nb-bialgebras corresponding to $\chi$, because 
    \begin{align*}
      (\Id_{C} \otimes \Delta_{L})\gamma &= (\Id_{C} \otimes \Delta_{L}\Lambda_{A})\gamma^{\univ} \\
      &= (\Id_{C} \otimes (\hat \Lambda_{B}^{\univ}f \otimes \Lambda_{A})\Delta^{\univ}_{A})\gamma^{\univ} \\
      &= (\Id_{C} \otimes \hat  \Lambda_{B}f \otimes \Lambda_{A})(\gamma^{\univ} \otimes \Id_{A^{\univ}})\gamma^{\univ} \\
      &= ((\Id_{C} \otimes \hat  \Lambda_{B}f)\gamma^{\univ} \otimes \Id_{A}) \gamma.
    \end{align*}
    For example, the coaction $\gamma:=(\Id_{A^{\univ}} \otimes \Lambda)\Delta^{\univ}_{A}$ of $\G$ on $\hat  A^{\univ}$ has such a lift $\gamma^{\univ}=\Delta^{\univ}_{A}$, and $\gamma$ is injective if and only if the reducing morphism $\Lambda$ is injective.  For a comparison of  coactions of $(A^{\univ},\Delta^{\univ}_{A})$ and of $(A,\Delta_{A})$, see also \cite{Fischer:Quantum_group_and_equivariant_kk-theory}, but note that only  injective ones are considered there.
  \end{enumerate}
\end{example}

\begin{proposition}
  Let $(C,\gamma)$ be a weakly continuous coaction of $\G$. If $\chi_{*}\gamma$ exists, then this morphism is uniquely
    determined, a weakly continuous coaction of $\DuG[H]$, and diagram
    \eqref{eq:pushed-action-deltar} commutes. If $\gamma$ is continuous,
    \corr{the so}{then so} is $\chi_{*}\gamma$.
\end{proposition}
\begin{proof}
Since $\gamma$ is weakly continuous, the map $\chi_{*}\gamma$ is uniquely
determined by \eqref{eq:pushed-action-deltal}. To see that it is \corr{an
  coaction}{a coaction}, note that
  \begin{align*}
    (\chi_{*}\gamma \otimes \Id_{\hat  B} \otimes \Id_{A})(\chi_{*}\gamma \otimes \Id_{A})\gamma &=
       (\chi_{*}\gamma \otimes  \Delta_{L})\gamma \\
       &= (\Id_{C} \otimes \Id_{\hat  B} \otimes \Delta_{L})(\Id_{C} \otimes \Delta_{L})\gamma \\
&= (\Id_{C} \otimes \DuComult[B] \otimes \Id_{A})(\Id_{C} \otimes \Delta_{L})\gamma  \\
&= (\Id_{C} \otimes \DuComult[B] \otimes \Id_{A})(\chi_{*}\gamma \otimes \Id_{A})\gamma.
\end{align*}
Slicing the third tensor factor, we find that $\chi_{*}\gamma$ indeed is \corr{an
  coaction}{a coaction}.

The following computation shows that~$\chi_{*}\gamma$ 
makes diagram \eqref{eq:pushed-action-deltar} commute, and  uses the relation
$(\Id_{A}\otimes\Delta_{L})\Comult[A]=(\Delta_{R}\otimes\Id_{A})\Comult$ 
\cite{Meyer-Roy-Woronowicz:Homomorphisms}*{Lemma 5.7 (36)}:
  \begin{align*}
    (\gamma \otimes \Id_{\hat  B} \otimes \Id_{A})(\chi_{*}\gamma \otimes \Id_{A})\gamma &=
    (\gamma \otimes \Delta_{L})\gamma \\
    &= (\Id_{C} \otimes \Id_{A} \otimes \Delta_{L})(\Id_{C} \otimes \Comult[A])\gamma \\
&= (\Id_{C} \otimes \Delta_{R} \otimes \Id_{A})(\Id_{C} \otimes \Comult[A])\gamma   \\
&= (\Id_{C} \otimes \Delta_{R} \otimes \Id_{A})(\gamma \otimes \Id_{A})\gamma.
  \end{align*}

  To see that $\chi_{*}\gamma$ is weakly continous,  note that
  \begin{align*}
    \{ (\Id_{C} \otimes \omega)\chi_{*}\gamma(c) : \omega \in \hat B', c\in C\} &\subseteq
    \{ (\Id_{C} \otimes \omega \otimes \omega')(\chi_{*}\gamma \otimes \Id_{A})\gamma(c) : \omega \in \hat B', \omega' \in A'\} \\
&=    \{ (\Id_{C} \otimes (\omega \otimes \omega')\Delta_{L})\gamma(c) : \omega \in \hat B', \omega' \in A'\}.
  \end{align*}
  Since $\Delta_{L}$ is injective, functionals of the form $(\omega \otimes \omega')\Delta_{L}$ above are linearly dense in $A'$. Since $\gamma$ is weakly continuous, we can conclude that $\chi_{*}\gamma$ is weakly continuous as well.

Finally,
suppose that $\chi_{*}\gamma$ is continuous. Then the Podle\'s condition for $\gamma$ and $\Delta_{L}$ implies
  \begin{align*}
    (\chi_{*}\gamma \otimes \Id_{A})\gamma(C) \cdot (1 \otimes \hat  B \otimes A) &=
    (\Id_{C} \otimes \Delta_{L})\gamma(C) \cdot (1 \otimes \Delta_{L}(A)(\hat  B\otimes A)) \\
    &= (\Id_{C} \otimes \Delta_{L})(\gamma(C) (1 \otimes A)) \cdot (1 \otimes \hat  B \otimes A) \\
    &=C \otimes \Delta_{L}(A) (\hat  B\otimes A) \\
    &= C\otimes \hat  B\otimes A.
  \end{align*}
  Slicing on the third tensor factor, we find that $\chi_{*}\gamma(C)(1\otimes \hat  B)=C\otimes \hat  B$.
\end{proof}
Let us now consider the iteration.
\begin{proposition} 
\label{proposition:push-action-iterate}
  Let $(C,\gamma)$ be a weakly continuous coaction of $\G$  such that $\chi_{*}\gamma$ exists.
 Suppose that $\mathbb{I}=(D,\Delta_{D})$ is a $\Cst$\nb-quantum group with a bicharacter $\chi' \in \U(B \otimes \hat  D)$. Then $\chi'_{*}(\chi_{*}\gamma)$ exists if and only if $(\chi'\ast \chi)_{*}\gamma$ exist, and in that case, both coincide.
\end{proposition}
\begin{proof}
Let $\chi''=\chi' \ast \chi$ and denote by $\Delta_{L},\Delta'_{L},\Delta''_{L}$ the  left quantum group homomorphisms associated to  $\chi'$ and $\chi,\chi',y
\chi''$, respectively. Then a left-handed analogue of \cite{Meyer-Roy-Woronowicz:Homomorphisms}*{Proposition 6.3} shows that
  \begin{align*}
(\Delta'_{L} \otimes \Id_{A})\Delta_{L}&=    (\Id_{\hat  D} \otimes \Delta_{L}) \Delta''_{L}
  \end{align*}
and hence
\begin{align} 
 \label{eq:push-action}
    (\Id_{C} \otimes \Delta'_{L} \otimes \Id_{A})(\chi_{*}\gamma \otimes \Id_{A})\gamma   
&= (\Id_{C} \otimes  \Id_{D} \otimes \Delta_{L})(\Id_{C} \otimes \Delta''_{L})\gamma.
\end{align}

Suppose that $\chi'_{*}(\chi_{*}\gamma)$ exists. Then the left hand side above is equal to
\begin{align*}
  (\chi'_{*}(\chi_{*}\gamma)  \otimes \Id_{\hat  B}\otimes \Id_{A})(\chi_{*}\gamma \otimes \Id_{A})\gamma  
&=
(\chi'_{*}(\chi_{*}\gamma) \otimes \Delta_{L})\gamma
\end{align*}
Since $\Delta_{L}$ is injective, we can conclude that $(\chi'_{*}(\chi_{*}\gamma) \otimes \Id_{A})\gamma=(\Id_{C} \otimes \Delta''_{L})\gamma$, whence $\chi''_{*}\gamma$ exists and equals $\chi'_{*}(\chi_{*}\gamma)$.

Conversely, if $\chi''_{*}\gamma$  exists, then the right hand side in \eqref{eq:push-action} is equal to
  \begin{align*}
    (\Id_{C} \otimes  \Id_{D} \otimes \Delta_{L})(\chi''_{*}\gamma \otimes \Id_{A})\gamma 
&= (\chi''_{*}\gamma \otimes \Id_{A})(\Id_{C} \otimes \Delta_{L}) \gamma   \\
&= (\chi''_{*}\gamma \otimes \Id_{\hat  B} \otimes \Id_{A})(\chi_{*}\gamma \otimes \Id_{A}) \gamma.
  \end{align*}
Slicing the third tensor factor, we conclude that $(\Id_{C} \otimes \Delta'_{L})\chi_{*}\gamma  = (\chi''_{*}\gamma \otimes \Id_{\hat  B})\chi_{*}\gamma$ so that
$\chi'_{*}(\chi_{*}\gamma)$ exists and equals $\chi''_{*}\gamma$.
\end{proof}

\proof[Acknowledgements] We would like to thank the referee for his thorough proof-reading and many helpful comments that lead to \corr{an improvement}{improvements} in the presentation.


\begin{bibdiv}
  
\begin{biblist}
\bib{Baaj-Skandalis:Unitaires}{article}{
  author={Baaj, Saad},
  author={Skandalis, Georges},
  title={Unitaires multiplicatifs et dualit\'e pour les produits crois\'es de $C^*$\nobreakdash -alg\`ebres},
  journal={Ann. Sci. \'Ecole Norm. Sup. (4)},
  volume={26},
  date={1993},
  number={4},
  pages={425--488},
  issn={0012-9593},
  review={\MRref {1235438}{94e:46127}},
  eprint={http://www.numdam.org/item?id=ASENS_1993_4_26_4_425_0},
}

\bib{Baaj-Skandalis-Vaes:Non-semi-regular}{article}{
  author={Baaj, Saad},
  author={Skandalis, Georges},
  author={Vaes, Stefaan},
  title={Non-semi-regular quantum groups coming from number theory},
  journal={Comm. Math. Phys.},
  volume={235},
  date={2003},
  number={1},
  pages={139--167},
  issn={0010-3616},
  review={\MRref {1969723}{2004g:46083}},
  doi={10.1007/s00220-002-0780-6},
}

\bib{Enock-Timmermann:yd}{article}{
  author={Enock, Michel},
  author={Timmermann, Thomas},
  title={Measured quantum transformation groupoids},
  date={2016},
  journal={J. Noncommut. Geometry},
  volume={10},
  number={3},
  pages={1134-1214},
  issn={1661-6952},
  review={\MRref {3554845}{}},
  doi={10.4171/JNCG/257},
}

\bib{Fischer:Quantum_group_and_equivariant_kk-theory}{thesis}{
  author={Fischer, Robert},
  title={Volle verschr\"ankte Produkte f\"ur Quantengruppen und \"aquivariante KK-Theorie},
  institution={Westf. Wilhelms-Universit\"at M\"unster},
  type={phdthesis},
  date={2003},
  eprint={http://nbn-resolving.de/urn:nbn:de:hbz:6-85659526538},
}

\bib{Gordon-Power-Street:Tricategory}{article}{
  author={Gordon, Robert},
  author={Power, John},
  author={Street, Ross},
  title={Coherence for tricategories},
  journal={Mem. Amer. Math. Soc.},
  volume={117},
  date={1995},
  number={558},
  pages={vi+81},
  issn={0065-9266},
  review={\MRref {1261589}{}},
  doi={10.1090/memo/0558},
}

\bib{Gray:lns}{book}{
  author={Gray, John W.},
  title={Formal category theory: adjointness for {$2$}-categories},
  series={Lecture Notes in Mathematics, Vol. 391},
  publisher={Springer-Verlag, Berlin-New York},
  date={1974},
  page={xii+282},
  review={\MRref {0371990}{}},
}

\bib{gurski:dissertation}{thesis}{
  author={Gurski, Michael Nicholas},
  title={An algebraic theory of tricategories},
  type={phdthesis},
  institution={The University of Chicago},
  publisher={ProQuest LLC, Ann Arbor, MI},
  date={2006},
  pages={209},
  isbn={978-0542-71041-4},
  review={\MRref {2717302}{}},
  eprint={http://search.proquest.com/docview/304955293},
}

\bib{Kasprzak:Rieffel_deformation}{article}{
  author={Kasprzak, Pawe\l },
  title={Rieffel deformation via crossed products},
  journal={J. Funct. Anal.},
  volume={257},
  date={2009},
  number={5},
  pages={1288--1332},
  issn={0022-1236},
  review={\MRref {2541270}{2010e:46073}},
  doi={10.1016/j.jfa.2009.05.013},
}

\bib{Kasprzak:Rieffel_coaction}{article}{
  author={Kasprzak, Pawe\l },
  title={Rieffel deformation of group coactions},
  journal={Comm. Math. Phys.},
  volume={300},
  date={2010},
  number={3},
  pages={741--763},
  issn={0010-3616},
  review={\MRref {2736961}{2012e:46150}},
  doi={10.1007/s00220-010-1093-9},
}

\bib{Kustermans:LCQG_universal}{article}{
  author={Kustermans, Johan},
  title={Locally compact quantum groups in the universal setting},
  journal={Internat. J. Math.},
  volume={12},
  date={2001},
  number={3},
  pages={289--338},
  issn={0129-167X},
  review={\MRref {1841517}{2002m:46108}},
  doi={10.1142/S0129167X01000757},
}

\bib{Kustermans-Vaes:LCQG}{article}{
  author={Kustermans, Johan},
  author={Vaes, Stefaan},
  title={Locally compact quantum groups},
  journal={Ann. Sci. \'Ecole Norm. Sup. (4)},
  volume={33},
  date={2000},
  number={6},
  pages={837--934},
  issn={0012-9593},
  review={\MRref {1832993}{2002f:46108}},
  doi={10.1016/S0012-9593(00)01055-7},
}

\bib{Leinster:Higher-Operads}{book}{
  title={Higher operads, higher categories},
  author={Leinster, Tom},
  series={London Mathematical Society Lecture Note Series},
  volume={298},
  publisher={Cambridge University Press},
  place={Cambridge},
  date={2004},
  pages={xiv+433},
  isbn={0-521-53215-9},
  review={\MRref {2094071}{2005h:18030}},
  note={DOI 10.1017/CBO9780511525896},
}

\bib{Majid:Quantum_grp}{book}{
  author={Majid, Shahn},
  title={Foundations of quantum group theory},
  publisher={Cambridge University Press},
  place={Cambridge},
  date={1995},
  pages={x+607},
  isbn={0-521-46032-8},
  review={\MRref {1381692}{97g:17016}},
  note={DOI 10.1017/CBO9780511613104},
}

\bib{Masuda-Nakagami-Woronowicz:C_star_alg_qgrp}{article}{
  author={Masuda, Tetsuya},
  author={Nakagami, Y.},
  author={Woronowicz, Stanis\l aw Lech},
  title={A $C^*$\nobreakdash -algebraic framework for quantum groups},
  journal={Internat. J. Math},
  volume={14},
  date={2003},
  number={9},
  pages={903--1001},
  issn={0129-167X},
  review={\MRref {2020804}{2004j:46100}},
  doi={10.1142/S0129167X03002071},
}

\bib{Meyer-Roy-Woronowicz:Homomorphisms}{article}{
  author={Meyer, Ralf},
  author={Roy, Sutanu},
  author={Woronowicz, Stanis\l aw Lech},
  title={Homomorphisms of quantum groups},
  journal={M\"unster J. Math.},
  volume={5},
  date={2012},
  pages={1--24},
  issn={1867-5778},
  review={\MRref {3047623}{}},
  eprint={http://nbn-resolving.de/urn:nbn:de:hbz:6-88399662599},
}

\bib{Meyer-Roy-Woronowicz:Twisted_tensor}{article}{
  author={Meyer, Ralf},
  author={Roy, Sutanu},
  author={Woronowicz, Stanis\l aw Lech},
  title={Quantum group-twisted tensor products of \(\textup C^*\)\nobreakdash -algebras},
  journal={Internat. J. Math.},
  volume={25},
  date={2014},
  number={2},
  pages={1450019, 37},
  issn={0129-167X},
  review={\MRref {3189775}{}},
  doi={10.1142/S0129167X14500190},
}

\bib{Meyer-Roy-Woronowicz:Twisted_tensor_2}{article}{
  author={Meyer, Ralf},
  author={Roy, Sutanu},
  author={Woronowicz, Stanis\l aw Lech},
  title={Quantum group-twisted tensor products of \(\textup {C}^*\)\nobreakdash -algebras II},
  journal={J. Noncommut. Geom.},
  volume={10},
  date={2016},
  number={3},
  pages={859--888},
  issn={1661-6952},
  review={\MRref {3554838}{}},
  doi={10.4171/JNCG/250},
}

\bib{Nest-Voigt:Poincare}{article}{
  author={Nest, Ryszard},
  author={Voigt, {Ch}ristian},
  title={Equivariant Poincar\'e duality for quantum group actions},
  journal={J. Funct. Anal.},
  volume={258},
  date={2010},
  number={5},
  pages={1466--1503},
  issn={0022-1236},
  review={\MRref {2566309}{2011d:46143}},
  doi={10.1016/j.jfa.2009.10.015},
}

\bib{Roy:Codoubles}{article}{
  author={Roy, Sutanu},
  title={The Drinfeld double for $C^*$\nobreakdash -algebraic quantum groups},
  journal={J. Operator Theory},
  volume={74},
  date={2015},
  number={2},
  pages={485–515},
  issn={0379-4024},
  review={\MRref {3431941}{}},
  doi={10.7900/jot.2014sep04.2053},
}

\bib{Soltan-Woronowicz:Remark_manageable}{article}{
  author={So\l tan, Piotr Miko\l aj},
  author={Woronowicz, Stanis\l aw Lech},
  title={A remark on manageable multiplicative unitaries},
  journal={Lett. Math. Phys.},
  volume={57},
  date={2001},
  number={3},
  pages={239--252},
  issn={0377-9017},
  review={\MRref {1862455}{2002i:46072}},
  doi={10.1023/A:1012230629865},
}

\bib{Soltan-Woronowicz:Multiplicative_unitaries}{article}{
  author={So\l tan, Piotr Miko\l aj},
  author={Woronowicz, Stanis\l aw Lech},
  title={From multiplicative unitaries to quantum groups. II},
  journal={J. Funct. Anal.},
  volume={252},
  date={2007},
  number={1},
  pages={42--67},
  issn={0022-1236},
  review={\MRref {2357350}{2008k:46170}},
  doi={10.1016/j.jfa.2007.07.006},
}

\bib{Vaes:Induction_Imprimitivity}{article}{
  author={Vaes, Stefaan},
  title={A new approach to induction and imprimitivity results},
  journal={J. Funct. Anal.},
  volume={229},
  date={2005},
  number={2},
  pages={317--374},
  issn={0022-1236},
  review={\MRref {2182592}{2007f:46065}},
  doi={10.1016/j.jfa.2004.11.016},
}

\bib{Woronowicz:Mult_unit_to_Qgrp}{article}{
  author={Woronowicz, Stanis\l aw Lech},
  title={From multiplicative unitaries to quantum groups},
  journal={Internat. J. Math.},
  volume={7},
  date={1996},
  number={1},
  pages={127--149},
  issn={0129-167X},
  review={\MRref {1369908}{96k:46136}},
  doi={10.1142/S0129167X96000086},
}

  \end{biblist}
\end{bibdiv}

\end{document}